\documentclass[a4paper]{article} 

\usepackage[dvipsnames]{xcolor}
\usepackage{amsmath}
\usepackage{amssymb}
\usepackage{amsthm}
\usepackage{Baskervaldx}
\usepackage[labelfont=bf]{caption}
\usepackage[nocompress]{cite}
\usepackage[T1]{fontenc}
\usepackage[margin=25mm]{geometry}
\usepackage{graphicx}
\usepackage{tikz}
	\usetikzlibrary{decorations.pathmorphing}
\usepackage{titling}
\usepackage{tocloft}
\usepackage{MnSymbol} 

\usepackage[colorlinks=true,citecolor=Green,linkcolor=NavyBlue,urlcolor=Magenta]{hyperref}
	\hypersetup{breaklinks=true}
\usepackage[capitalise,noabbrev]{cleveref}
	\crefname{equation}{equation}{equations}

\theoremstyle{plain}
	\newtheorem{thm}{Theorem}
	\newtheorem{prop}[thm]{Proposition}
	\newtheorem{cor}[thm]{Corollary} 
	\newtheorem{lemma}[thm]{Lemma} 
	\numberwithin{thm}{section}
	\newtheorem*{thm*}{Theorem}
\theoremstyle{definition}
	\newtheorem{defn}[thm]{Definition} 
\theoremstyle{remark}
	\newtheorem{remark}[thm]{Remark} 

\DeclareMathOperator{\Aut}{Aut}
\DeclareMathOperator{\expl}{Expl}

\newcommand{\an}{{\mathbf a}}
\newcommand{\E}{\mathfrak{E}}
\newcommand{\ex}{\mathbf}
\newcommand{\N}{\mathfrak N}
\newcommand{\Prop}{P}
\newcommand{\rh}{{}^{r}\!H}
\newcommand{\tc}[1]{\check\rvert_{#1}}
\newcommand{\T}{\mathfrak{T}}
\newcommand{\Tlarge}{\mathfrak t_b}
\newcommand{\Tsmall}{\mathfrak t_s}
\newcommand{\W}{\mathcal W}

\providecommand{\abs}[1]{\left\lvert #1\right\rvert}
\providecommand{\lrb}[1]{\ensuremath{\left(#1\right)}}
\providecommand{\op}[2]{{}_{#2}#1}
\providecommand{\totb}[1]{\ensuremath{\underline{ #1}}}
\providecommand{\totl}[1]{\ensuremath{\lceil #1\rceil }}
\providecommand{\ip}[2]{\int #1\wedge #2}
\providecommand{\pp}[2]{\left\llangle #1, #2\right\rrangle }

\tikzset{box style/.style={red!10}}
\tikzset{toric style/.style={color=cyan!90, fill=cyan!10, very thick}}
\tikzset{wall style/.style={black, line width = 1mm}}
\tikzset{curve style/.style={red, very thick}}
\tikzset{wallcurve style/.style={red, very thick, decorate, decoration={snake, segment length=2.5mm, amplitude=0.25mm, post length=0mm, pre length=0mm}}}

\title{The tropological vertex} 
\author{Norman Do \and Brett Parker}

\begin{document}

\textbf{\Large \thetitle}

\textbf{\theauthor}

School of Mathematics, Monash University, VIC 3800 Australia \\
Email: \href{mailto:norm.do@monash.edu}{norm.do@monash.edu}

Mathematical Sciences Institute, The Australian National University, ACT 2601 Australia \\
Email: \href{mailto:brett.parker@anu.edu.au}{brett.parker@anu.edu.au}

{\em Abstract.} The theory of the topological vertex was originally proposed by Aganagic, Klemm, Mari\~{n}o and Vafa as a means to calculate open Gromov--Witten invariants of toric Calabi--Yau threefolds. In this paper, we place the topological vertex within the context of relative Gromov--Witten invariants of log Calabi--Yau manifolds and describe how these invariants can be effectively computed via a gluing formula for the enumeration of tropical curves in a singular integral affine space. This richer context allows us to prove that the topological vertex possesses certain tropical symmetries. These symmetries are captured by the action of a quantum torus Lie algebra that is related to a quantisation of the Lie algebra of the tropical vertex group of Gross, Pandharipande and Siebert. Finally, we demonstrate how this algebra of symmetries leads to an explicit description of the topological vertex and related Gromov--Witten invariants.



\emph{Acknowledgements.} The first author was supported by the Australian Research Council grant DP180103891. Both authors thank the mathematical research institute MATRIX in Australia where part of this research was performed. Both authors thank Jean-Emile Bourgine for directing us to the references~\cite{AFS,Sdual}. 

\vspace{12pt} \hrule \vspace{12pt}

\setcounter{tocdepth}{1}
\tableofcontents

\vspace{12pt} \hrule

\section{Introduction}

Aganagic, Klemm, Mari\~{n}o and Vafa introduced the theory of the topological vertex for the effective computation of open Gromov--Witten invariants, motivated by the duality between Chern--Simons theory and Gromov--Witten theory~\cite{AKMV,gop-vaf99}. Li, Liu, Liu and Zhou provided a rigorous mathematical construction of the topological vertex in terms of formal relative Gromov--Witten invariants~\cite{LLLZ}. We will relate this construction to Gromov--Witten invariants of three-dimensional log Calabi--Yau manifolds, which are computable using a calculus of tropical curves. This perspective provides an algebra of tropical symmetries of the topological vertex; hence, `the tropological vertex'.

The aforementioned tropical symmetries arise as a representation of a quantum torus Lie algebra, which is introduced in \cref{sec:quantumtorus} and comes in two forms. The small quantum torus Lie algebra $\Tsmall$ has generators $T_{(a,b)}$ for $(a,b)\in\mathbb Z^2 \setminus \{(0,0)\}$ that satisfy the commutation relations
\[
[T_{(a,b)},T_{(c,d)}]=[ad-bc]_q \, T_{(a+c,b+d)} \ ,
\] 
where we use the slightly unconventional notation $[n]_q := -i (q^{n/2}-q^{-n/2})=2\sin(n \hbar/2)$ and the relation $q^{1/2} = e^{i\hbar/2}$. This small quantum torus Lie algebra is associated to a real 2-dimensional torus with Lie algebra $\mathfrak t^2$ so that the $\mathbb Z^2$ used to index the generators corresponds naturally to the integral lattice $\mathbb Z^2=\mathfrak t^2_{\mathbb Z}\subset\mathfrak t^2$. The big quantum torus Lie algebra $\Tlarge$ is a central extension of $\Tsmall$ by $\mathfrak t^2$, with generators that satisfy the commutation relations
\[
[T_{(a,b)},T_{(c,d)}]=[ad-bc]_q \, T_{(a+c,b+d)} + \delta_{(a+c,b+d),(0,0)}(a,b) \ .
\]

Our quantum torus Lie algebra can be regarded as a quantisation of the Lie algebra of the tropical vertex group of Gross, Pandharipande and Siebert~\cite{tropicalvertex}, and also appears as the Lie algebra of the quantum tropical vertex group of Bousseau~\cite{qtropicalvertex}.

The main output of our analysis is the following, which appears subsequently with proof as \cref{tvs}.

\begin{thm*}
For any $v\in\mathfrak t^2_{\mathbb Z}$, we have
\[
\left( \W_{v,1^-} + \W_{v,2^-} + \W_{v,3^-} \right) \T = 0 \ .
\]
\end{thm*}

One can think of $\T$ as a partition function for the topological vertex, storing certain Gromov--Witten invariants of $\mathbb{C}^3$, whose toric graph comprises three legs adjacent to a single vertex. The operators $\W_{v,\ell}$ encode Gromov--Witten invariants that virtually count holomorphic curves with an extra constraint corresponding to $v \in \mathfrak t^2_{\mathbb Z}$ on the leg $\ell$ of the toric graph. Using a calculus of tropical curves, we prove that the map
\[
T_v \longmapsto \W_{v,1^-} + \W_{v,2^-} + \W_{v,3^-}
\]
produces a representation of the small quantum torus Lie algebra $\Tsmall$. Thus, the result above demonstrates that tropical symmetries of the topological vertex are captured by the action of a quantum torus Lie algebra. A connection between the closely related quantum $\mathcal{W}_{1+\infty}$ algebra and the refined topological vertex was shown in the algebraic approach of Awata, Feigin and Shiraishi~\cite{AFS}, and a similar symmetry of the topological vertex was used by Sasa,  Watanabe, and Matsuo in \cite{Sdual}.

The relevant definitions and full details required for the statement and proof of the theorem above appear in the remainder of the paper.

The geometric setup for our analysis and results appears in \cref{cytotoric,cdtc,es}.
\begin{itemize}
\item In \cref{cytotoric}, we discuss how to construct log Calabi--Yau manifolds from toric Calabi--Yau manifolds and relate the topological vertex to Gromov--Witten invariants of these log Calabi--Yau manifolds. The Lie algebra $\mathfrak t^2$ described above makes its first appearance here as an algebra of symmetries for such manifolds.
\item In \cref{cdtc}, we describe how holomorphic curves in these log Calabi--Yau manifolds can be studied via tropical curves in a singular integral affine space. 
\item In \cref{es}, we introduce the spaces used to encode Gromov--Witten invariants of these log Calabi--Yau manifolds. In this three-dimensional Calabi--Yau setting, it turns out that these evaluation spaces have a holomorphic symplectic structure and the image of the moduli space of holomorphic curves is holomorphic Lagrangian.
\end{itemize}

The algebraic setup for our analysis and results appears in \cref{rgw,nr,ss,css}.
\begin{itemize}
\item In \cref{rgw,nr}, we explain how to encode Gromov--Witten invariants using cohomology classes and package them in suitable generating functions. We then introduce a Novikov ring $\N$ in which such generating functions naturally reside. This Novikov ring  allows us to keep track of the symplectic area and Euler characteristic of the curves that we enumerate.
\item In \cref{ss,css}, we package constraints for the Gromov--Witten invariants of interest into an algebra $\mathcal H$, which plays the role of a bosonic Fock space of quantum states. We encode the full Gromov--Witten theory in a partition function $\exp\eta\in\mathcal H$ that can be determined tropically, but contains more information than the topological vertex. The topological vertex partition function $Z_{\bar X}$ is obtained by projecting $\mathcal H$ to a subalgebra $\mathcal H^\pm$, corresponding to counting holomorphic curves with particular constraints. This constrained state space $\mathcal H^\pm$ is naturally a tensor product of algebras $\mathcal H^\pm_\ell$ over the legs $\ell$ of the toric graph associated to our toric Calabi--Yau manifold.
\end{itemize}

The analysis that leads to a proof of our main result appears in the remaining \cref{gfs,sec:product,toriccase,sec:woperators,sec:quantumtorus,sec:framing,sec:tv}.
\begin{itemize}
\item In \cref{gfs}, we reformulate the gluing formula for topological vertex partition functions in the language of this paper. This formula can be derived from the tropical gluing formula appearing in previous work of the second author~\cite[Equation~(1)]{gfgw}. We briefly present this tropical gluing formula, which is also required for subsequent arguments in the paper.

\item In \cref{sec:product,toriccase}, we calculate Gromov--Witten invariants in the simplest cases by using the tropical correspondence formula for three-dimensional toric manifolds as a key input~\cite{3d}. This allows us to give a complete analysis of the cases $X = \mathbb C\times (\mathbb C^*)^2$ and $X = \mathbb C^2 \times \mathbb C^*$, which correspond to the empty tropical graph and the tropical graph with no vertices, respectively.

\item In \cref{sec:woperators,sec:quantumtorus}, we construct operators $\W_{v,\ell}$ on $\mathcal H_\ell^+$ that encode Gromov--Witten invariants counting holomorphic curves with an extra constraint corresponding to $v\in \mathfrak t^2_{\mathbb Z}\setminus \{0\}$. \cref{Wcommutation} states that these operators provide a representation of the big quantum torus Lie algebra $\Tlarge$, corresponding to a projective representation of $\Tsmall$. The proof involves the calculus of tropical curves, using diagrams such as that pictured in \cref{f0}.
We then prove that this is a highest weight representation with weights calculated in \cref{W01,W0}.  In \cref{translation} we observe that, after tensoring with $\mathbb C$, our representation agrees with one arising in the Gromov--Witten/Hurwitz correspondence of Okounkov and Pandharipande~\cite{OP1}.
\begin{figure}[ht!]
\centering
\begin{tikzpicture}
\begin{scope}
\fill [box style] (0,0) rectangle (4,2);
\draw[wall style] (0,1) -- (4,1);
\draw[curve style] (1.5,1) -- (0.5,0);
\draw[curve style] (2,0) -- (2,1);
\draw[wallcurve style] (0,1) -- (4,1);
\end{scope}
\node at (4.5,1) {$=$};
\begin{scope}[xshift=5cm]
\fill [box style] (0,0) rectangle (4,2);
\draw[wall style] (0,1) -- (4,1);
\draw[curve style] (2.5,1) -- (1.5,0);
\draw[curve style] (2,0) -- (2,1);
\draw[wallcurve style] (0,1) -- (4,1);
\end{scope}
\node at (9.5,1) {$+$};
\begin{scope}[xshift=10cm]
\fill [box style] (0,0) rectangle (4,2);
\draw[wall style] (0,1) -- (4,1);
\draw[curve style] (2,0.5) -- (1.5,0);
\draw[curve style] (2,0) -- (2,0.5);
\draw[curve style] (2.25,1) -- (2,0.5);
\draw[wallcurve style] (0,1) -- (4,1);
\end{scope}
\end{tikzpicture}
\caption{ $\W_{v,\ell}\W_{w,\ell}=\W_{w,\ell}\W_{v,\ell}+\left[{v\wedge w}\right]_q \, \W_{v+w,\ell}$}
\label{f0}
\end{figure}

\item In \cref{sec:framing}, we analyse the behaviour of the operators $\W_{v,\ell}$ under a change of framing. Recall that a choice of framing is required for each leg $\ell$ in order to define the topological vertex. We prove that the structure of $\mathcal H^+_\ell$ as a highest weight representation does not depend on the framing, so the framing change isomorphism $F_{\bar X}$ is an isomorphism of representations. Again, the proof involves the calculus of tropical curves, using diagrams such as that pictured in \cref{f01}.
\begin{figure}[ht!]
\centering
\begin{tikzpicture}
\begin{scope}
\fill [box style] (0,0) rectangle (4,2);
\draw[wall style] (0,1) -- (2,1) -- (3,0);
\draw[curve style] (2,0) -- (2.3333,0.6666);
\draw[wallcurve style] (0,1) -- (2,1) -- (3,0);
\end{scope}
\node at (4.5,1) {$=$};
\begin{scope}[xshift=5cm]
\fill [box style] (0,0) rectangle (4,2);
\draw[wall style] (0,1) -- (2,1) -- (3,0);
\draw[curve style] (0.5,0) -- (1,1);
\draw[wallcurve style] (0,1) -- (2,1) -- (3,0);
\end{scope}
\end{tikzpicture}
\caption{$F_{\bar X}\circ\W_{v,\ell_1}=\W_{v,\ell_2^-}\circ F_{\bar X}$.}
\label{f01}
\end{figure}
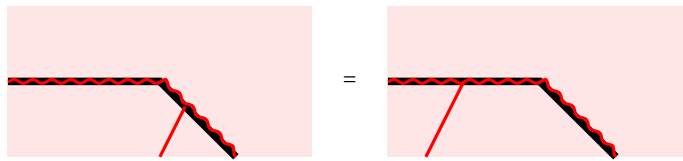

\item In \cref{sec:tv}, we consider the topological vertex partition function, corresponding to the toric Calabi--Yau manifold $X=\mathbb C^3$, whose toric graph comprises three legs adjacent to a single vertex. The partition function $Z_{\bar X}$ is an element of the algebra $\mathcal H^+_{1^-}\otimes\mathcal H^+_{2^-}\otimes\mathcal H^+_{3^-}$, which has a representation of the small quantum torus Lie algebra $\Tsmall$ induced from the representations of $\Tlarge$ on the three tensor factors. We state and prove our main results --- \cref{tvs} and \cref{Zsym} --- which assert that $Z_{\bar X}$ is invariant under this action, thereby giving rise to tropical symmetries of the topological vertex. Once again, the proof involves the calculus of tropical curves, using diagrams such as that pictured in \cref{f02}. We conclude the paper by showing how these symmetries allow one to effectively calculate the topological vertex in \cref{1leg,2leg,3leg} and \cref{tvcalc}.

\begin{figure}[ht!]
\centering
\begin{tikzpicture}
\begin{scope}
\fill [box style] (0,0) rectangle (4,2);
\draw[wall style] (1,2) -- (2,1) -- (4,1);
\draw[wall style] (2,0) -- (2,1);
\draw[curve style] (1,0) -- (2,0.5);
\draw[wallcurve style] (1,2) -- (2,1) -- (4,1);
\draw[wallcurve style] (2,0) -- (2,1);
\end{scope}
\node at (4.5,1) {$+$};
\begin{scope}[xshift=5cm]
\fill [box style] (0,0) rectangle (4,2);
\draw[wall style] (1,2) -- (2,1) -- (4,1);
\draw[wall style] (2,0) -- (2,1);
\draw[curve style] (1,0) -- (3,1);
\draw[wallcurve style] (1,2) -- (2,1) -- (4,1);
\draw[wallcurve style] (2,0) -- (2,1);
\end{scope}
\node at (9.5,1) {$=$};
\begin{scope}[xshift=10cm]
\fill [box style] (0,0) rectangle (4,2);
\draw[wall style] (1,2) -- (2,1) -- (4,1);
\draw[wall style] (2,0) -- (2,1);
\draw[curve style] (0,1) -- (1.3333,1.6666);
\draw[wallcurve style] (1,2) -- (2,1) -- (4,1);
\draw[wallcurve style] (2,0) -- (2,1);
\end{scope}
\end{tikzpicture}
\caption{Two ways of calculating a Gromov--Witten invariant to derive tropical symmetries of the topological vertex.}
\label{f02}
\end{figure}
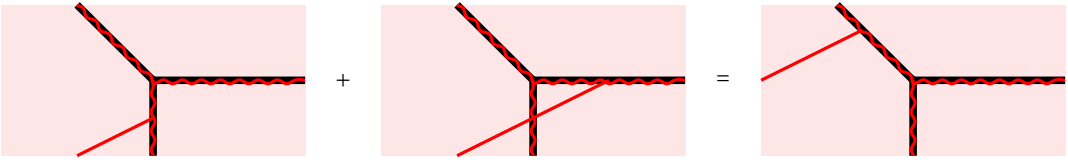
\end{itemize}

Our general approach relies on the theory of exploded manifolds developed by the second author in a series of previous works~\cite{iec,3d,gfgw,dre,vfc,scgp}. However, our main results only depend largely on two main outputs of that program --- the tropical gluing formula and the three-dimensional tropical correspondence formula.


\section{From toric Calabi--Yau to log Calabi--Yau} \label{cytotoric}

A toric Calabi--Yau threefold $X$ is a (non-compact) toric manifold with $\dim_{\mathbb C}X=3$ and a holomorphic volume form $\Omega$. In toric coordinates,
\[
\Omega_X = f \frac{\mathrm{d}z_1}{2\pi i z_1}\wedge \frac{\mathrm{d}z_2}{2\pi iz_2}\wedge \frac{\mathrm{d}z_3}{2\pi iz_3} \ ,
\]
where $f:X\longrightarrow \mathbb C$ is a globally defined holomorphic function, whose vanishing set is the toric boundary divisor of $X$. So on the interior of $X$, $\Omega_X$ is $f$ multiplied by the unique $(\mathbb C^*)^3$--invariant holomorphic $3$-form with integral equal to 1 on each each torus fibre. Assume that the toric fan\footnote{The toric fan of $X$ naturally embeds in the Lie algebra of $(\mathbb R^*)^3\subset (\mathbb C^*)^3$. Each stratum $N$ of the boundary divisor corresponds to the cone comprised of vectors $v$ in this Lie algebra such that the flow of $v$ fixes $N$, and flows generic points towards $N$. Identify the Lie algebra of $(\mathbb R^*)^3\subset(\mathbb C^*)^3$ with $i$ times the Lie algebra $\mathfrak t^3$ of $\mathbb T^3\subset(\mathbb C^*)^3$, so there is a natural notion of integral vectors. Each codimension $1$ component of the toric boundary of $X$ is the positive span of a primitive integral vector $v$. The monomial $f$ determines a $\mathbb C$--linear function $w_f$ on the Lie algebra $\mathfrak g$ of the algebraic torus $(\mathbb C^*)^3$ acting on $X$ such that the derivative of $f$ under the flow of $v\in \mathfrak g$ is $w_f(v)f$. The condition that the zero set of $f$ is the toric boundary divisor of $X$ is equivalent to $w_f(v)=-1$. } of $X$ is convex and that $f$ is a primitive monomial in the toric coordinates.\footnote{The condition that the toric fan of $X$ is convex implies that $f$ must be a toric monomial times the exponential of a global holomorphic function, so we can always deform $\Omega_X$ until $f$ is a toric monomial.} Let $\bar X$ be a smooth toric compactification of $X$ such that $f$ extends to a meromorphic function on $\bar X$, and vanishes only on the toric boundary strata of $X\subset \bar X$.\footnote{ We can restate the conditions on $\bar X$ in terms of toric fans. The condition that $f$ extends to a meromorphic function is that the toric fan of $\bar X$ includes cones whose union is the kernel of $w_f$, and the condition that $f$ not vanish on any of the new boundary components of $\bar X$ is that $w_f\geq 0$ on the rays corresponding to the new codimension $1$ toric boundary components of $X$. This is achievable because the toric fan of $X$ is convex.}
The form
\[
\Omega_{\bar X}:=\frac {\Omega_X}{1-f}
\]
extends to a meromorphic volume form on $\bar X$ with poles at the simple normal crossing divisor
\[
D := f^{-1}(1) \cup (\bar X \setminus X)
\]
and $(\bar X,D)$ is a log Calabi--Yau manifold with logarithmic holomorphic volume form $\Omega_{\bar X}$. We will study Gromov--Witten invariants of $\bar X$ relative to the simple normal crossing divisor $D$ and relate these to the topological vertex of Aganagic, Klemm, Mari\~{n}o and Vafa~\cite{AKMV}.

\subsection{Toric graphs}

Let $\mathbb T^2$ be the sub-torus preserving $f$, and let $\mathfrak t^2$ be the corresponding Lie algebra. This has a canonical integral lattice, $\mathfrak t^2_{\mathbb Z}\subset \mathfrak t^2$ comprised of the elements whose time 1 flow is the identity. Choose a $\mathbb T^2$--invariant symplectic form $\omega$ on $\bar X$ defining a K\"ahler structure.\footnote{This K\"ahler form $\omega$ is a smooth form on $\bar X$, and hence degenerate as a logarithmic 2-form. As such, it cannot be chosen compatible with the logarithmic holomorphic volume form $\Omega_{\bar{X}}$ to reduce the structure group to $SU(3)$; for example, it is impossible for $\Omega_{\bar X}\wedge\bar{\Omega}_{\bar X}$ to be proportional to $\omega\wedge\omega\wedge\omega$. } This can be a toric symplectic form, however the following discussion does not require that $\omega$ be $\mathbb T^3$ invariant.

The moment map of the $\mathbb T^2$-action is a map
\[
\pi: \bar X\longrightarrow (\mathfrak t^2)^*=\mathbb R^2\ ,
\]
defined up to translation by the condition that for $v\in \mathfrak t^2$ and $\widetilde v$ the corresponding vector field on $X$, $-i_{\widetilde v}\omega= \pi^*dv$, where we consider $v$ as giving a linear function on $(\mathfrak t^2)^*$. Within $X \subset \bar X$, the $\mathbb T^2$-action is free away from toric boundary strata of codimension at least $2$; the image of these strata is a graph called the {\bf toric graph} of $X$. The edges of this toric graph travel in integral directions: if the flow of $v\in \mathfrak t^2$ is constant on the stratum over an edge, then the linear function $v$ on $(\mathfrak t^2)^*$ is constant on that edge. For an edge $e$, a choice of primitive integral generator $v_e\in\mathfrak t^2_{\mathbb Z}$ for the sub-torus preserving this stratum then corresponds to a choice of coorientation of $e$. Call such a $v_e$ a {\bf normal vector} to the edge $e$. Vertices of this toric graph are trivalent, and with a cyclic choice of coorientation of the edges leaving a vertex, the corresponding normal vectors sum to $0$. Moreover, we can always choose $\mathbb Z$-affine coordinates $(x,y)$ centred on a vertex such that these cyclically oriented normal vectors are $(0,1)$, $(-1,0)$, and $(1,-1)$.

The edges of a toric graph can also end at the boundary of $\pi(\bar X)$, which is a compact, convex polytope. Such an end is called a {\bf leg}. In $\bar X$, this leg $\ell$ corresponds to a $0$-dimensional stratum of $\bar X$, the intersection of two codimension $1$ strata of $X$, and one extra codimension $1$ stratum $\bar X_\ell$ of $\bar X$. Our conditions on $\bar X$ ensure that this extra stratum is necessarily fixed by the flow of some primitive $w_\ell\in \mathfrak t^2_{\mathbb Z}$ such that the corresponding boundary face of the polytope $\pi(\bar X)$ is where the linear function $w_\ell$ on $(\mathfrak t^2)^*$ is constant and achieves its maximum. Moreover, $w_\ell$ and the normal vector to the edge form a basis for the lattice $\mathfrak t^2_{\mathbb Z}$, so there exist $\mathbb Z$-affine coordinates $(x,y)$ such that the edge is a positive ray in the $x$ direction, and $w_\ell=x$. Different choices of compactification $\bar X\supset X$ lead to different $w_\ell$, and a choice of such $w_\ell$ is called a $\bf{framing}$. See \cref{fig:toricgraph} for an example of a toric graph.

\begin{figure}[ht!]
\centering
\begin{tikzpicture}
\begin{scope}
\filldraw[toric style] (0,0) rectangle (4,2);
\draw[wall style] (0.3,0.5) -- (0.3,1.5) -- (2.5,1.5) -- (2.5,0.5) -- cycle;
\draw[wall style] (0,0.2) -- (0.3,0.5);
\draw[wall style] (0,1.8) -- (0.3,1.5);
\draw[wall style] (3,0) -- (2.5,0.5);
\draw[wall style] (3,2) -- (2.5,1.5);
\end{scope}
\end{tikzpicture}
\caption{A toric graph with four legs and framing vectors $(-1,0)$, $(-1,0)$, $(0,1)$, $(0,-1)$.}
\label{fig:toricgraph}
\end{figure}
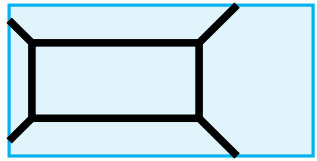

This stratum $\bar X_\ell \subset \bar X$ where the leg $\ell$ ends naturally has the structure of a $2$-dimensional log Calabi--Yau manifold. The divisor $D_\ell \subset \bar X_\ell$ is the intersection of $\bar X_\ell$ with the strata of $\bar X$ that don't contain $\bar X_\ell$. To obtain the holomorphic volume form $\Omega_{\bar X_\ell}$ on $\bar X_\ell$, denote by $\iota_{w_\ell}\Omega_{\bar X}$ the holomorphic $2$--form on $\bar X$ obtained by inserting the vector field generating the flow of $w_\ell$; so if $z$ is a primitive monomial vanishing on $\bar X_\ell \subset\bar X$, then $\Omega_{\bar X}=-\frac{\mathrm{d}z}{2\pi i z}\wedge \iota_{w_\ell}\Omega_{\bar X}$. Then $\Omega_{\bar X_\ell}$ is the restriction of $\iota_{w_\ell}\Omega_{\bar X}$ to $\bar X_\ell \subset \bar X$. We will also think of $\Omega_{\bar X_\ell}$ as a holomorphic symplectic form on $(\bar X_\ell, D_\ell)$.

\subsection{Lagrangian fibrations}

The holomorphic map 
\[
f:\bar X\longrightarrow \mathbb{CP}^1
\]
is a submersion away from $0$ and $\infty$, and the nondegenerate fibres are K\"ahler, with structure preserved by the $\mathbb T^2$-action. Parallel transport orthogonal to fibres defines a symplectic connection on this bundle. This connection preserves the symplectic structure, and also the $\mathbb T^2$-action and its moment map $\pi$. We can use this connection to define various Lagrangians in $X$ and $\bar X$. Because the moment map $\pi$ is preserved, parallel transport of a $\mathbb T^2$-orbit around any smooth closed curve $\gamma$ in $\mathbb{CP}^1\setminus\{0,\infty \}$ closes up to give a Lagrangian torus in $\bar X$. Moreover, this Lagrangian is $f^{-1}(\gamma)\cap\pi^{-1}(p)$ for some point $p\in(\mathfrak t^2)^*$. For example, the toric $\mathbb T^3$-orbits are the parallel transport of $\mathbb T^2$-orbits around the loops where $\abs f$ is constant. Parallel transport around the loops where $\abs {f-1}$ is constant defines a nice singular Lagrangian torus fibration on $\bar X$ such that the behaviour of this fibration around the divisor $D\subset\bar X$ is modelled on  the behaviour of a toric fibration near the toric boundary divisor. Moreover, $\Omega_{\bar{X}}$ is purely real on these Lagrangian fibres, so this is a special Lagrangian fibration. The geometric significance of this is that, if $\omega$ defined a K\"ahler metric in which $\Omega_{\bar X}$ was covariantly constant, then these special Lagrangian fibres would have minimal volume, because their volume coincides with the integral of the closed form $\Omega_{\bar X}$. In $f^{-1}(0)$, there are singularities of this fibration where $\mathrm{d}f=0$, but as $\mathrm{d}f \neq 0$ away from codimension 2 strata of $X$, the fibration is smooth away from such strata. So we have singular Lagrangian torus fibration 
\[
(\ln\abs{f-1},\pi):\bar X\longrightarrow [-\infty,\infty]\times (\mathfrak t^2)^* \ ,
\]
which is smooth away from the embedding of the toric graph in $\{0\}\times (\mathfrak t^2)^*$, which we will call the {\bf singular locus}. We will also refer to the inverse image of this graph in $\bar X$ as the singular locus --- this inverse image in $\bar X$ is actually a union of holomorphic spheres, one over each edge of the toric graph; it is also the locus where $f=0$ and $\mathrm{d}f = 0$. Over any point in an edge of the singular locus, the fibre  degenerates to an immersed Lagrangian $S^1\times S^2$ which intersects itself in a circle where $f=0$. These singular fibres degenerate further to have a more complicated singularity at $f=0$ over a vertex.

On the complement of the singular locus, the smooth Lagrangian torus fibration induces a $\mathbb Z$-affine structure on the base $[-\infty,\infty]\times (\mathfrak t^2)^*$. This structure is such that if $(x_1,x_2,x_3)$ are local $\mathbb Z$-affine coordinates on the base, they locally generate a free Hamiltonian $\mathbb T^3$-action on $\bar X$ with orbits the Lagrangian torus fibres. If $w\in\mathfrak t^2$ is integral, the corresponding linear function on $(\mathfrak t^2)^*$ is a $\mathbb Z$-affine function, however there is monodromy in the $\mathbb Z$-affine structure around the singular locus, so  global $\mathbb Z$-affine coordinates do not exist. Note that the real part of $\Omega_{\bar X}$ gives an orientation form on fibres, and hence induces an orientation on the base --- we choose the convention that if $(x_1,x_2,x_3)$ are oriented $\mathbb Z$-affine coordinates, then the Hamiltonian vector fields generated by $-x_1$, $-x_2$ and $-x_3$ provide an oriented basis for the tangent space of the fibres.

This fibration also restricts to a singular special Lagrangian torus fibration on the toric boundary stratum $\bar X_\ell$ at the end of a leg $\ell$. Here, the singular locus consists of a single point for each leg ending at this stratum. Both the $\mathbb Z$-affine stucture from the Lagrangian torus fibration and the orientation from $\Omega_{\bar X_\ell}$ coincide with the corresponding structure induced on the boundary of $\bar X$.

There is another interesting singular Lagrangian fibration
\[ 
\lrb{\frac{f-1}{\abs{f-1}},\pi}:\bar X \setminus f^{-1}(\{1,\infty \}) \longrightarrow S^1 \times (\mathfrak t^2)^* 
\] 
whose fibres are now non-compact Lagrangian manifolds which are special in the sense that $\Omega_{\bar X}$ restricts to be purely imaginary on them. On $(X,\Omega_X)$, there is a related non-compact special Lagrangian fibration given by the map $((f-\bar f),\pi)$. Of particular interest are the Aganagic--Vafa branes, given by $f^{-1}(-\infty,0]$ intersected with the inverse image of a point on an edge. These are diffeomorphic to $S^1\times \mathbb R^2$, intersecting $f^{-1}(0)$ in the circle $S^1\times\{0\}$; and are also special Lagrangians in $\bar X$ with the logarithmic holomorphic volume form $\Omega_{\bar X}$. These Aganagic--Vafa branes in $\bar X$ have  boundary a $\mathbb T^2$--orbit in $f^{-1}(\infty)$.

The above singular special Lagrangian fibration on $\bar X$ also restricts to $\bar X_\ell \subset \bar X$, and $\Omega_{\bar X_\ell}$ is also purely imaginary restricted to fibres. As there is a unique singular point corresponding to $\ell$ in the fibration of $\bar X_\ell$, there is a unique Aganagic--Vafa brane $A_\ell \subset \bar X_\ell$ corresponding to $\ell$ over $f^{-1}(-\infty,0]$. We orient $A_\ell$ so that $i\Omega_{\bar X_\ell}$ is a volume form on $A_\ell$. Topologically, $A_\ell \subset \bar X_\ell$ is a disk with boundary on $f^{-1}(\infty)$, so $A_\ell$ defines a class $[A_\ell] \in H_2(\bar X_\ell, f^{-1}(\infty))$.

\subsection{The topological vertex} 

The topological vertex, introduced by Aganagic, Klemm, Mari\~{n}o and Vafa, involves a virtual count of holomorphic curves in $X=\mathbb C^3$ with boundary on three chosen Aganagic--Vafa branes over the three legs of the corresponding toric graph~\cite{AKMV}. In this case, $f=z_1z_2z_3$ and the singular locus consists of the points where two coordinate functions vanish.

To define the topological vertex, the extra information of a framing is required --- in \cite{AKMV}, a framing is given by a choice of 1-dimensional sub-torus of $\mathbb T^2$ acting freely on this Aganagic--Vafa brane. Li, Liu, Liu, and Zhou provide a mathematical definition of the topological vertex using relative Gromov--Witten invariants~\cite{LLLZ}. In both \cite{LLLZ} and
 our setting, the framing on each leg is provided by a toric compactification $\bar X$ of $X$, so the new boundary component $\bar X_\ell$ at the end of each leg $\ell$ is invariant under the 1-dimensional sub-torus given by the framing. Symplectically, this $\bar X_\ell$ can be regarded as the quotient of a hypersurface by this 1-dimensional sub-torus, and $A_\ell \subset \bar X_\ell$ is the quotient of an Aganagic--Vafa brane by this 1-dimensional sub-torus.

The definition of the topological vertex in~\cite{LLLZ} involves a count of holomorphic curves in $\bar X$, touching the new parts of the boundary divisor only in $A_\ell \subset \bar X_\ell$. Heuristically, a point on a holomorphic curve sent to $A_\ell$ and tangent to $\bar X_\ell$ with order $k$ plays the role of a boundary of a holomorphic curve that wraps $k$ times around the $S^1$ direction of the Aganagic--Vafa brane. In both~\cite{AKMV} and~\cite{LLLZ}, the invariants only see curves around the singular locus of $\bar X$, where $f=0$ and $\mathrm{d}f=0$. Accordingly,~\cite{LLLZ} defines the required curve counts in terms of formal relative Gromov--Witten invariants, with `formal' meaning that the moduli stack of holomorphic curves is restricted to a formal neighbourhood of the curves in the singular locus. This restriction allows~\cite{LLLZ} to use the technology of Gromov--Witten invariants relative to smooth divisors from~\cite{Li}, avoiding the need for Gromov--Witten invariants relative to normal crossing divisors, which were only defined after~\cite{LLLZ} was published; see~\cite{GSlogGW,IonelGW,vfc}. In what follows, we instead use Gromov--Witten invariants of $\bar X$ relative to the normal crossing divisor $D$, so that we can apply the tropical gluing formula from~\cite{gfgw} to derive tropical symmetries of the topological vertex.

\subsection{Substituting holomorphic constraints for Lagrangian constraints} 

In general, the moduli stack of holomorphic curves in $\bar X$ constrained using $A_\ell \subset \bar X_\ell$ has nonempty boundary, so it is more convenient for us to use different constraints. When no other legs have the same framing as $\ell$, $f^{-1}(0)\subset \bar X_\ell$ is the union of two embedded holomorphic spheres $E^+_\ell$ and $E^-_\ell$ so that 
\begin{equation} \label{epm}
[E^\pm_\ell] = \pm[A_\ell] \text { in } H_2(\bar X_\ell, f^{-1}(\infty)) \ .
\end{equation} 
We can make a minor perturbation of $A_\ell$ for different legs $\ell$, so that their image under $f$ intersects only at $0$. Then, the constraint of requiring holomorphic curves in $\bar X$ to only touch the divisor in these $A_\ell$ is satisfied only by curves in $f^{-1}(0)$. We can therefore get the same counts by using $E^+_\ell$ or $-E^-_\ell$ as a constraint in place of $A_\ell$.

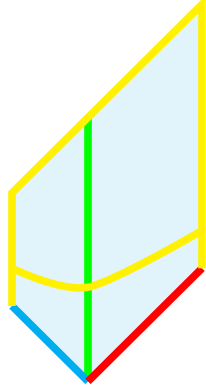
\begin{figure}[ht!]
\centering
\begin{tikzpicture}
\fill[toric style] (-1,1) -- (0,0) -- (1.5,1.5) -- (1.5,5) -- (-1,2.5) -- cycle;
\draw[green, line width = 1mm] (0,0) -- (0,3.5);
\draw[cyan, line width = 1mm] (0,0) -- (-1,1);
\draw[red, line width = 1mm] (0,0) -- (1.5,1.5);
\draw[yellow, line width = 1mm] (1.5,1.5) -- (1.5,5) -- (-1,2.5) -- (-1,1);
\draw[yellow, line width = 1mm] plot [smooth] coordinates{(-1,1.5) (0,1.25) (1.5,2)};
\end{tikzpicture}
\caption{A toric picture of $\bar X_\ell$, with divisor $D_\ell$ shown in yellow, the Lagrangian submanifold $A_\ell$ shown in green, and the holomorphic spheres $E^+_\ell$ shown in red, and $E^-_\ell$ shown in blue.} 
\label{xlpic}
\end{figure}

As classes in $H_2(\bar X)$, we can readily calculate the intersections between $E^{\pm}_\ell$.
\begin{equation} \label{eint}
[E^+_\ell]\cdot [E^-_\ell]=1 \qquad [E^+_\ell]\cdot[E^+_\ell]=-1 \qquad [E^-_\ell]\cdot[E^-_\ell]=-1
\end{equation}
These spheres $E^\pm_\ell$ are related to the two choices of normal vectors to the leg $\ell$. In particular, there is a canonical choice\footnote{In this paper, there are a series of choices, choosing an orientation on $A_\ell$, distinguishing $E^+$ and $E^-$, and $\pm n_\ell$, and further choices on signs of operators $\W_{v,\ell}$. These choices are inconsequential, so long as the choices for different legs are compatible. } of normal vector $n_\ell\in\mathfrak t^2_{\mathbb Z}$ such that, for some choice of moment map, $\bar X\longrightarrow \mathfrak (\mathfrak t^2)^*$ (recalling that the moment map is only canonical up to translation) the hamiltonian function corresponding to $n_\ell$ is positive on $E^+$ and negative on $E^-$. Another way of describing this canonical choice is as follows: removing their intersection with the divisor, $E^\pm$ is a complex plane, and  the action of $n_\ell$ on $E^+$ has weight $1$ and the action of $n_\ell$ on $E^-$ has weight $-1$.

When $k$ legs have the same framing, $f^{-1}(0)\subset \bar X_\ell$ consists of $k+1$ spheres. In this case define $E^\pm_\ell$ to be a union of these spheres satisfying \cref{epm}. In each case, $f^{-1}(0)=E^+_\ell\cup E^-_{\ell}$, with both $E_\ell^+$ and $E^-_\ell$ connected, and intersecting only at the point in $\bar X_\ell$ corresponding to $\ell$. \Cref{eint} still holds. Moreover, for two legs $\ell$ and $\ell'$ with the same framing, $[E^{\pm}_\ell]\cdot [E^{\pm}_{\ell'}]=0$.

There exist toric degenerations of $X$ which break the original toric graph at some internal edges into matched pairs of legs $\ell^+, \ell^-$ with the opposite framings $w_{\ell^-}=-w_{\ell^+}$. There is a natural identification of $\bar X_{\ell^+}$ with $\bar X_{\ell^-}$, but the induced holomorphic volume forms are opposite $\Omega_{\bar X_{\ell^-}}=-\Omega_{\bar X_{\ell^+}}$, so $E_{\ell^-}^+=E_{\ell^+}^-$. Holomorphic curves within $f^{-1}(0)$ in such situations can be analysed using tropical curves in a $\mathbb Z$-affine space modelled on $\mathfrak t^2$, but with a 1-dimensional piecewise linear singular locus. The broken internal edge of the toric graph corresponds to an edge in this singular locus travelling in the direction $w_{\ell^+}$, and there are tropical curves corresponding to $E_{\ell^+}^\pm$ travelling out from this singular edge in the directions $\pm n_{\ell^+}$. Our holomorphic volume form gives a canonical orientation of $\mathfrak t^2_{\mathbb Z}$ such that for each leg, $(w_\ell, n_\ell)$ is an oriented basis for $\mathfrak t^2_{\mathbb Z}$. See \cref{fig:toricdegeneration} for an example of a toric degeneration.

\begin{figure}[ht!]
\centering
\begin{tikzpicture}
\begin{scope}
\filldraw[toric style] (0,0) rectangle (4,2);
\draw[wall style] (0.3,0.5) -- (0.3,1.5) -- (2.5,1.5) -- (2.5,0.5) -- cycle;
\draw[wall style] (0,0.2) -- (0.3,0.5);
\draw[wall style] (0,1.8) -- (0.3,1.5);
\draw[wall style] (3,0) -- (2.5,0.5);
\draw[wall style] (3,2) -- (2.5,1.5);
\draw[cyan, very thick] (0,0.9) -- (0.8,0.9) -- (1.1,1.2) -- (1.1,2);
\draw[cyan, very thick] (0.8,0) -- (0.8,0.9) -- (1.1,1.2) -- (4,1.2);
\end{scope}
\begin{scope}[xshift=5cm]
\fill[box style] (0,0) rectangle (4,2);
\draw[curve style] (0.25,1.5) -- (0.75,2);
\draw[curve style] (1.5,0) -- (1.5,0.5);
\draw[curve style] (2,1.75) -- (2.75,1) -- (4,1);
\draw[curve style] (2.75,1) -- (2.75,0);
\draw[wall style] (2,0) -- (2,2);
\draw[wall style] (0,0.5) -- (2,0.5);
\draw[wall style] (0,1.5) -- (2,1.5);
\draw[wall style] (0.5,0.5) -- (0.5,1.5);
\end{scope}
\end{tikzpicture}
\caption{A degeneration of $\bar X$ into four toric manifolds, and three red tropical curves ending on the black singular locus of the corresponding $\mathbb Z$-affine space.}
\label{fig:toricdegeneration}
\end{figure}
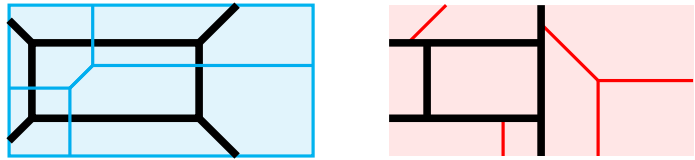

Using Poincar\'{e} duals, we define the cohomology classes
\[
\alpha_\ell=\alpha^+_\ell:=PD(E^+_\ell) \qquad \text{and} \qquad \alpha^-_\ell:=PD(E^-_\ell) \ .
\] 
Considering these cohomology classes as differential forms, we have 
\[ 
\int_{\bar X_\ell}\alpha^+_\ell \wedge \alpha^-_\ell=1 \qquad \text{and} \qquad \int_{\bar X_\ell}\alpha^\pm_\ell\wedge\alpha^\pm_\ell=-1 \ . 
\] 
Inside the cohomology of $\bar X_\ell$ relative $f^{-1}(\infty)$, we have $\alpha^-_\ell=-\alpha^+_\ell$, and making this identification uncovers some beautiful structure in our Gromov--Witten invariants.

To explain the full structure of relative Gromov--Witten invariants of $(\bar X,D)$, it is convenient to use exploded manifolds, so instead of the complex manifold $\bar X$ with the normal crossing divisor, we instead use its explosion $\expl (\bar X,D)$, and define invariants using exploded holomorphic curves in $\expl(\bar X,D)$; see~\cite{iec,scgp}. For the present discussion, it is enough to think of the moduli space of holomorphic curves in $\expl (\bar X,D)$ as a suitable compactification of the moduli space of holomorphic curves in $\bar X$, which are smooth and not contained in $D$. When we apply the explosion functor to such a smooth holomorphic curve, we obtain an exploded holomorphic curve in $\expl(\bar X,D)$, however the moduli space of exploded curves keeps track of more structure when such curves sink into the divisor. In particular, there is an evaluation map from the moduli stack of exploded curves analogous to the evaluation at a point of contact with $\bar X_\ell$, except now this map has codomain the exploded manifold $\expl (\bar X_\ell, D_\ell)$. We introduce the notation 
\[
\ex X_{w_\ell}:=\expl(\bar X_\ell, D_\ell) \ .
\]

Evaluation at a point of contact order $k$ to $\bar X_\ell$ determines a map to the quotient stack\footnote{This stack is constructed in Section 3 of~\cite{gfgw}. In more general situations, it is not naturally a quotient stack, but we can identify it with a quotient stack in this case by choosing a $k$th root $L_k$ of the normal bundle to $\bar X_\ell \setminus D_\ell \subset \bar X$, constructed using the trivialisation from the primitive monomial $z^\beta$ that vanishes on $\bar X_\ell$, but is a non-vanishing holomorphic function on $f^{-1}(0)$ near $\bar X_\ell$. The evaluation map to $X_\ell / \mathbb Z_k\subset \bar X_\ell / \mathbb Z_k$ is then given by the usual evaluation map to $X_\ell$, and the $\mathbb Z_k$--bundle whose fibres are isomorphisms of the tangent space at the marked point with the pullback of $L_k$ compatible with the natural isomorphism of the $k$th tensor power of this tangent space with the pullback of the normal bundle. For a more precise explanation, either log geometry or exploded manifolds can be used to describe what happens in the boundary of the moduli space, when curves sink into the divisor.} $\bar X_\ell / \mathbb Z_k$, where we take the quotient by the trivial $\mathbb Z_k$-action. The exploded manifold analogue is given by
\[
\ex X_{kw_\ell}/G_{kw_\ell}, \quad \text{where } \ex X_{kw_\ell}:= \expl(\bar X_\ell,D_\ell) \quad \text{and} \quad G_{kw_\ell}:=\mathbb Z_k \ .
\]

As $E^\pm\longrightarrow \bar X_\ell$ is holomorphic, and not contained in $D_\ell$, we can apply the explosion functor to this map, and, using~\cite{dre}, define the Poincar\'{e} dual to this as the class $\alpha^\pm_{w_\ell}\in H^2(\ex X_{w_\ell})$. Pushing this class forward to $\ex X_{kw_\ell}/G_{kw_\ell}$ gives a class $\alpha^{\pm}_{k,\ell}\in H^2(\ex X_{kw_\ell}/G_{kw_\ell})$, whose pullback to $\ex X_{w_\ell}$ is $k\alpha^\pm_{\ell}$. With these definitions,  we have 
\[
\int_{\ex X_{kw_\ell}/G_{kw_\ell}}\alpha^+_{k,\ell}\wedge \alpha^-_{k,\ell}:=\frac 1k\int_{\ex X_{w_\ell}}k\alpha^+_\ell\wedge k\alpha^-_\ell=k \qquad \text{and} \qquad \int_{\ex X_{kw_\ell}/G_{kw_\ell}}\alpha^\pm_{k,\ell}\wedge \alpha^\pm_{k,\ell}=-k.
\]

\section{Contact data and tropical curves} \label{cdtc} 

The exploded manifold $\expl (\bar X,D)$ has a functorial projection to a singular $\mathbb Z$-affine space, called its {\bf tropical part} $\totb{\expl (\bar X,D)}$~\cite{iec}. This tropical part has a stratification into $\mathbb Z$-affine cones, each isomorphic to $[0,\infty)^3$, $[0,\infty)^2$, $[0,\infty)$, or a point, where each $k$-dimensional cone corresponds to the intersection of $k$ components of $D$. See \cref{fig:momentpolytope} for an example of the tropical part of an exploded manifold.

\begin{figure}[ht!]
\centering \includegraphics[height=5cm]{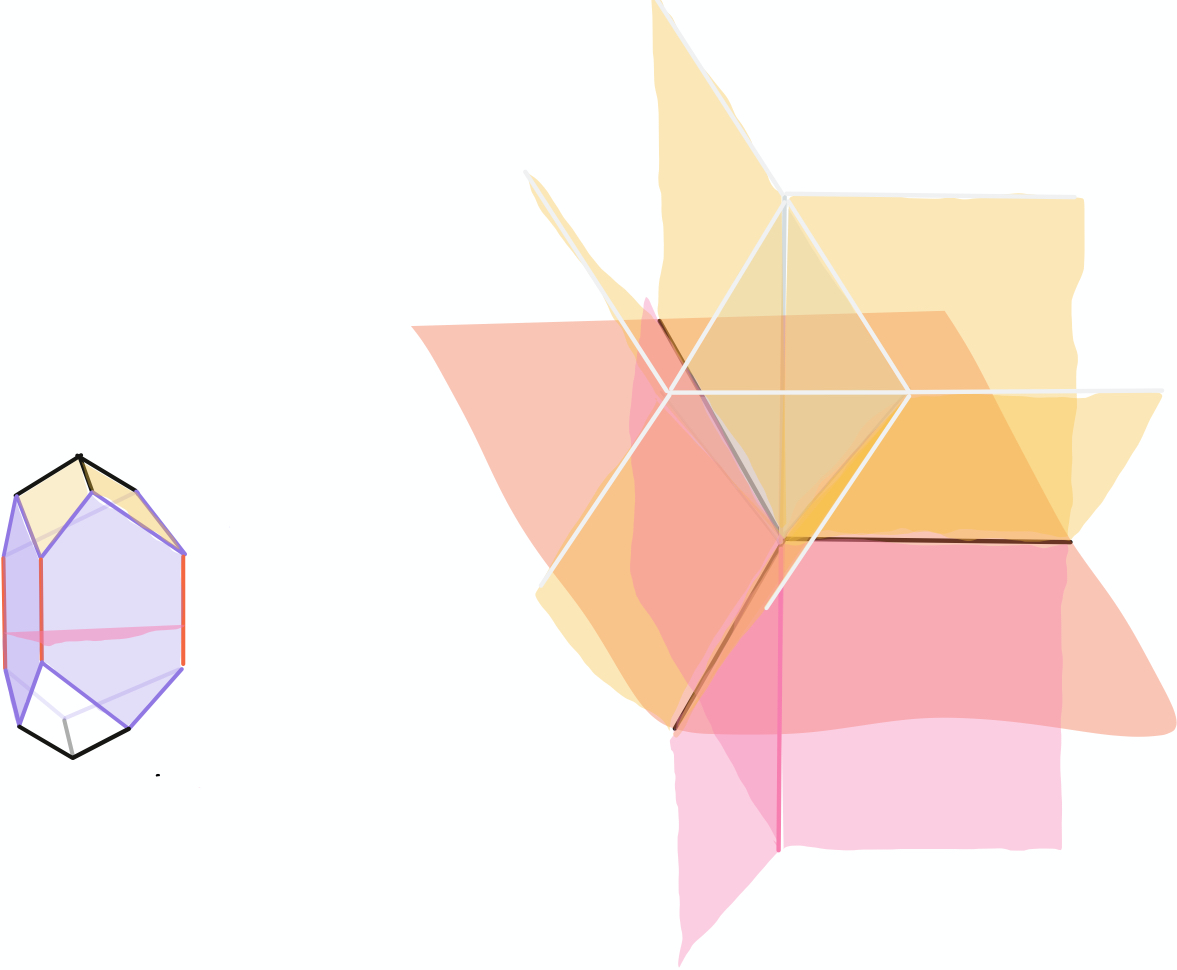}
\caption{A toric moment polytope of $\bar X$ in the case $X=\mathbb C^3$ on the left and its tropical part $\totb{\expl (\bar X,D)}$ on the right.}
\label{fig:momentpolytope}
\end{figure}

In general, the $\mathbb Z$-affine structure on the tropical part of an exploded manifold does not extend across strata. In this Calabi--Yau case however, we can extend the $\mathbb Z$-affine structure to a global singular $\mathbb Z$-affine structure, with singular rays corresponding to the legs of our toric graph. Divide $\totb{\expl (\bar X,D)}$ into two nonsingular half spaces, glued over a wall isomorphic to $\mathfrak t^2$, to create a singular $\mathbb Z$-affine space with singularities along the rays spanned by framing vectors $w_\ell$. The top half of $\totb{\expl(\bar X,D)}$ is the tropical part of $\expl \left((\bar X,D)\setminus f^{-1}(1)\right)$, which is naturally identified with the toric fan of $\bar X\setminus f^{-1}(0)$. As such, it has a natural global $\mathbb Z$-affine structure as a closed half space within the Lie algebra of the torus $\mathbb T^3$.  The wall is the tropical part of $\expl\left((\bar X,D)\setminus f^{-1}(\{1,\infty\})\right)$, which is naturally isomorphic to $\mathfrak t^2$, subdivided by the toric fan of $\bar X\setminus f^{-1}(\{0,\infty\})$. The bottom half is the tropical part of $\expl (\bar X,D)\setminus f^{-1}(\infty)$. This is naturally isomorphic to $\mathfrak t^2\times[0,\infty)$, with the stratification induced from the product of the stratification of the wall $\mathfrak t^2$, and $[0,\infty)$. The distinguished stratum $\{0\} \times [0,\infty)$ in this bottom half corresponds to the component $f^{-1}(1)$ of the divisor $D$.

We extend the $\mathbb Z$-affine structure of $\totb{\expl (\bar X,D)}$ by gluing the bottom half to the top half as follows. Let $v_f$ be the primitive integral vector in the $[0,\infty)$ direction, corresponding to the component $f^{-1}(1)$ of the divisor $D$. Using the given $\mathbb Z$-affine structure, we can transport $v_f$ anywhere in the bottom half, and we now specify how to transport $v_f$ into the top half over the interior of a 2-dimensional cone $\sigma$ on the wall. This cone corresponds to the intersection of two toric strata of $\bar X$ intersecting $f^{-1}(0)$ in exactly one stratum $S_\sigma$. Let $v_{S_\sigma}$ be the corresponding primitive integral vector in the toric fan of $\bar X$. This then can be considered as a constant vector field on the top half, which is naturally identified with half of the toric fan of $\bar X$. We extend the $\mathbb Z$-affine structure by gluing the top and bottom halves over $\sigma$ so $v_{f}$ is sent to $v_{S_\sigma}$ when we parallel transport $v_f$ over the cone $\sigma$. Note that this depends on the cone $\sigma$. If $\sigma$ and $\sigma'$ intersect along the singular ray spanned by $w_\ell$, and no other legs have the same framing, then the difference between $v_{S_\sigma}$ and $v_{S_{\sigma'}}$ is the normal vector $n_\ell$.

The importance of this global $\mathbb Z$-affine structure is due to the following. Each exploded holomorphic curve in $\expl (\bar X,D)$ has a tropical part consisting of a tropical curve in $\totb{\expl(\bar X, D)}$. These tropical curves satisfy the usual tropical balancing condition at vertices unless these vertices are on the singular locus. This balancing condition follows from the balancing condition for holomorphic curves in the explosion of toric manifolds relative to toric boundary divisors, because apart from the strata corresponding to legs, each stratum of the boundary divisor $D$ has a neighbourhood isomorphic to a toric boundary stratum in a toric manifold. See \cref{fig:tropicalcurves} for a diagrammatic representation of tropical curves in the tropical part of an exploded manifold.

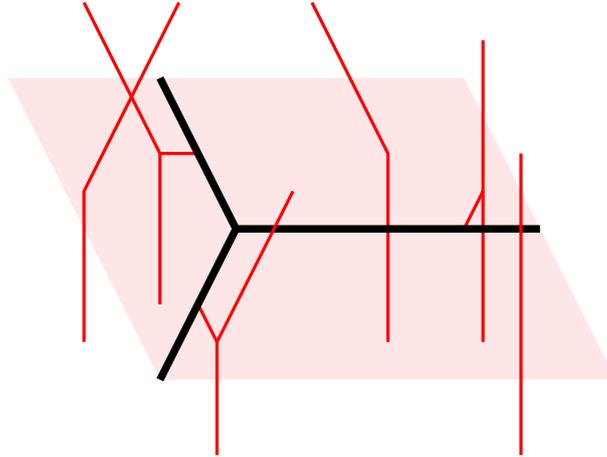
\begin{figure}[ht!]
\centering
\begin{tikzpicture}
\fill[box style] (-1,2) -- (1,-2) -- (7,-2) -- (5,2) -- cycle;

\draw[curve style] (3,3) -- (4,1) -- (4,-1.5);

\draw[curve style] (5,0) -- (5.25,0.5);
\draw[curve style] (5.25,-1.5) -- (5.25,2.5);

\draw[wall style] (2,0) -- (6,0);

\draw[curve style] (1.5,-1) -- (1.75,-1.5) -- (2.75,0.5);
\draw[curve style] (1.75,-1.5) -- (1.75,-3);

\draw[curve style] (1.5,1) -- (1,1) -- (0,3);
\draw[curve style] (1,1) -- (1,-1);

\draw[curve style] (1.25,3) -- (0,0.5) -- (0,-1.5);

\draw[curve style] (5.75,-3) -- (5.75,1);

\draw[wall style] (1,-2) -- (2,0) -- (1,2);
\end{tikzpicture}
\caption{A diagrammatic representation of tropical curves in $\totb{\expl(\bar X,D)}$, with the wall shown in pink, the singular locus shown in black, and tropical curves in red. Note that how tropical curves bend when passing through the wall depends on what cone they pass through. Through the interior of each cone $\sigma$ in the above picture, one of the above tropical curves is a `straight line' coming down to the wall in the direction $v_{S_\sigma}$, and passing through the wall in the direction $v_f$. These tropical curves are the tropical part of holomorphic spheres which have sunk into the component of the divisor corresponding to $\sigma$. The corresponding family of holomorphic spheres close to this component of the divisor consists of the orbits of the $\mathbb C^*$-action generated by $v_{S_\sigma}$.}
\label{fig:tropicalcurves}
\end{figure}

Relative Gromov--Witten invariants of $(\bar X,D)$ count holomorphic curves with specified contact with the divisor $D$, or more accurately, count the corresponding curves in $\expl(\bar X,D)$. In what follows, we formalise how we encode contact data. Given a smooth holomorphic curve in $\bar X$ intersecting the divisor in a collection of points, the explosion of this holomorphic curve has tropical part a tropical curve in $\totb{\expl(\bar X,D)}$ consisting of a single vertex sent to $0$, connected to an edge $[0,\infty)\longrightarrow \totb{\expl(\bar X,D)}$ for each point of contact with $D$. Let $M_{\mathbb Z}$ be the set of non-constant $\mathbb Z$-affine maps $[0,\infty)\longrightarrow \expl(\bar X,D)$ sending $0$ to $0$. The derivative of such a map is a nonzero integral vector $v$ in some cone $\sigma$ of $\expl(\bar X,D)$, and we use the notation $v\in M_{\mathbb Z}$.

The relationship between $v\in M_{\mathbb Z}$ and contact with the divisor is as follows: If $v$ is within the 1-dimensional cone corresponding to a component $S$ of the divisor, $v=kv_S$, where $v_S$ is the corresponding primitive integral vector and $k$ is a positive integer. Then $v$ indicates contact of order $k$ with the interior of $S$. More generally, if $v$ is in the interior of a cone spanned by $v_{S_i}$, we have $v=\sum_ik_i v_{S_i}$, and $v$ indicates contact with $\bigcap_i S_i$, of order $k_i$ with $S_i$.

We encode contact data $\mathbf p$ for a curve as a map
\[
\mathbf p: M_{\mathbb Z}\longrightarrow \{0, 1, 2, \ldots\} \ ,
\]
so that a curve with contact data $\mathbf p$ has exactly $\mathbf p(v)$ contact points of type $v\in M_{\mathbb Z}$. For curves with contact data $\mathbf p$ to have finite energy, $\mathbf p(v)$ must be zero for all but finitely many $v\in M_{\mathbb Z}$, so all contact data $\mathbf p$ is automatically assumed to have this property. This contact data continues to make sense for exploded curves that have sunk into the divisor. An exploded curve with contact data $\mathbf p$ is one whose tropical part continuously deforms to a tropical curve with all vertices at $0$, and with exactly $\mathbf p(v)$ infinite rays of type $v$ emanating from $0$.

The curves most relevant to this paper are contained in $f^{-1}(0)$, and thus never intersect $f^{-1}(1)$ or $f^{-1}(\infty)$. As such they have contact data $\mathbf p$ supported on 
\[
\mathfrak t^2_{\mathbb Z}\setminus \{0\} \subset M_{\mathbb Z} \ .
\]

Specialising further, we will be counting holomorphic curves in $\bar X$ only intersecting $D$ in the strata $\bar X_\ell$ at the end of each leg, and constrained to $E_\ell^+$. To specify the contact data of such curves, we need a sequence of natural numbers ${\mathbf p^\ell}:=(p^\ell_1,p^\ell_2, \ldots)$ for each leg $\ell$; this encodes contact data $\mathbf p^\ell: M_{\mathbb Z} \longrightarrow \{0, 1, 2, \ldots\}$ such that $\mathbf p^\ell(kw_\ell)= p^\ell_k$, and $\mathbf p^\ell(v)=0$ for any $v\in M_{\mathbb Z}$ that is not a positive multiple of the framing vector $w_\ell$. After specifying such contact data $\mathbf p^\ell$ for each leg $\ell$, we obtain contact data $\mathbf p=\sum_\ell \mathbf p^\ell$. Note however, that in the case that multiple legs have the same framing, we can not recover $\mathbf p^\ell$ from $\mathbf p$.

\section{The evaluation spaces \texorpdfstring{$\ex X_v$}{Xv} and \texorpdfstring{$\ex X^{\mathbf p}$}{Xp}} \label{es}

Suppose $v\in M_{\mathbb Z}$ is a primitive integral vector corresponding to a component $\bar X_v$ of the divisor $D$. The intersection $D_v$ of this component with the rest of the divisor is a normal crossing divisor in $\bar X_v$. Evaluation at a contact point of type $v$ determines a map to the exploded manifold
\[
\ex X_v:=\expl (\bar X_v,D_v)\ .
\]

More generally, to each primitive integral vector $w\in M_{\mathbb Z}$, we can associate an exploded manifold $\ex X_w$, constructed in~\cite[Section 3]{gfgw}. This exploded manifox $\ex X_w$ has a refinement,~\cite[Section 10]{iec}, which can be constructed as follows. We can perform blowups of $(\bar X,D)$, locally modelled on toric blowups, until $w$ corresponds to a component $(\bar X'_w,D_w)$ of the blown up divisor. These blowups correspond to refinements of $\expl(\bar X, D)$, which induce refinements of $\ex X_w$. In particular, $\expl(\bar X'_w,D_w)$ is the induced refinement of $\ex X_w$. In the case of a smooth curve with a contact point of type $w$, the evaluation map at this point can be understood as follows: performing these blowups gives a curve with simple contact with the interior of $\bar X'_w$, so evaluation at this point gives a map to $\bar X'_w$.

Using the language of exploded manifolds, we can describe $\ex X_w$ more explicitly in coordinates. Let $\widetilde {\ex X}_w$ be the subset of $\expl(\bar X,D)$ over the stratum corresponding to $w$; there is a canonical projection $\widetilde {\ex X}_w\longrightarrow \ex X_w$, that can be roughly thought of as the quotient by an action corresponding to $w$. We describe this more precisely below.

If $w$ is in the interior of the upper half of $M_{\mathbb Z}$, we can choose a basis $(z^{\alpha_1}, z^{\alpha_2}, z^{\beta})$ for toric monomials such that the flow induced by $w$ acts with weight 0 on $z^{\alpha_i}$ and weight $-1$ on $z^\beta$. These monomials do not generally extend to smooth functions on $\bar X$, however they each extend to smooth maps, $\widetilde z^{\alpha_i}$ and $\widetilde z^\beta$, from $\expl (\bar X,D)\setminus f^{-1}(0)$ to the exploded manifold $\ex T$, so they define exploded coordinate functions; see~\cite[Section 3]{iec}. These exploded coordinates $(\widetilde z^{\alpha_1},\widetilde z^{\alpha_2},\widetilde z^{\beta})$ give global coordinates on $\widetilde{\ex X}_w$. The exploded manifold $\widetilde {\ex X}_w$ is defined using its smooth part $\totl{\widetilde{\ex X}_w}$, which is the stratum of $D$ corresponding to $w$, and the sheaf of exploded functions generated by $\widetilde z^{\alpha_i}$ and $\widetilde z^{\beta}$. The exploded manifold $\ex X_w$ has the same smooth part, but its sheaf of exploded functions is generated by only $\widetilde z^{\alpha_1}$ and $\widetilde z^{\alpha_2}$; so the projection $\widetilde {\ex X}_w\longrightarrow \ex X_w$ simply forgets the coordinate $\widetilde z^\beta$.

If $w$ is in the lower half of $M_{\mathbb Z}$, we can similarly choose maps $\widetilde z^{\alpha_i}$ and $\widetilde z^{\beta}$ from $\widetilde {\ex X}_w$ to $\ex T$ such that the projection to $\ex X_w$ is given by $\widetilde z^{\alpha_i}$, however these exploded functions are obtained using monomials with $1-f$ replacing the role of the primitive monomial $f$. In particular, choose a basis $(f,u_1,u_2)$ for toric monomials on $X$, such that $u_i$ extend to holomorphic functions on $f^{-1}(0)$ near the stratum of $\bar X$ corresponding to $w$. Consider monomials in $1-f$, $u_1$ and $u_2$. Each of these monomials extends to a map to $\ex T$ near the stratum corresponding to $w$, and, so long as $w$ is not a framing vector, these exploded functions provide global coordinates on $\widetilde {\ex X}_w$. (When $w$ is a framing vector, these functions only fail to provide global coordinates because the derivative of $1-f$ vanishes when $f=0$.) Suppose that $w=av_f+w'$ with $w'\in\mathfrak t^2_{\mathbb Z}$ and $a\geq 0$, and suppose that $w'$ acts with weight $-b$ on $u_1$ and $-c$ on $u_2$. The vector $w\in M_{\mathbb Z}$ is primitive when $(a,b,c)$ is a primitive vector in $\mathbb Z^3$. The description of $\ex X_w$ is as above in the toric case, except the exploded coordinates $\widetilde z^{\alpha_i}$ are the extensions of $(1-f)^{x_1}u_2^{x_2}u_3^{x_3}$ where $(x_1,x_2,x_3)\in \mathbb Z^3$ satisfy $x_1a+x_2b+x_3c=0$, and $\widetilde z^\beta$ is the extension of some monomial in this form such that $x_1a+x_2b+x_3c=1$.

More generally, evaluation at contact points of type $kw\in M_{\mathbb Z}$ with $w$ primitive, determines a map to
\[
\ex X_{kw}/G_{kw}, \quad \text{where } \ex X_{kw}:=\ex X_w \quad \text{and} \quad G_{kw}:=\mathbb Z_k \ .
\]
As in the case when $w$ is a framing vector, the codomain of this evaluation map is a stack which is not naturally a quotient stack, however we can identify it with the quotient of $\ex X_{kw}$ by the trivial $\mathbb Z_k$-action by choosing a $k$th root of the coordinate function $\widetilde z^{\beta}$ on $\widetilde{\ex X}_w$.

Each $\ex X_w$ has a natural holomorphic volume form $\Omega_{\ex X_w}$ induced from $\Omega_{\bar X}$. On $\widetilde {\ex X}_w$, 
\[
\Omega_{\bar X}=h \, \frac{\mathrm{d}\widetilde z^{\alpha_1}}{2\pi i\widetilde z^{\alpha_1}}\wedge \frac{\mathrm{d}\widetilde z^{\alpha_2}}{2\pi i\widetilde z^{\alpha_2}}\wedge \frac{\mathrm{d}\widetilde z^{\beta}}{2\pi i\widetilde z^{\beta}} \ ,
\] 
where $h$ is holomorphic and $\mathbb C^*$-valued. When $w$ is primitive, $\Omega_{\ex X_w}$ is then defined by 
\[
\Omega_{\ex X_w} = -h \, \frac{\mathrm{d}\widetilde z^{\alpha_1}}{2\pi i\widetilde z^{\alpha_1}}\wedge \frac{\mathrm{d}\widetilde z^{\alpha_2}}{2\pi i\widetilde z^{\alpha_2}} \ ,
\]
and for a positive multiple $kw$ of $w$, define $\Omega_{\ex X_{kw}}$ by
\[
\Omega_{\ex X_{kw}}=-kh \, \frac{\mathrm{d}\widetilde z^{\alpha_1}}{2\pi i\widetilde z^{\alpha_1}}\wedge \frac{\mathrm{d}\widetilde z^{\alpha_2}}{2\pi i\widetilde z^{\alpha_2}}\ .
\]

Similarly, for any holomorphic volume form $\Omega$ on an exploded manifold $\ex X$, there is an analogous holomorphic volume form $\Omega_{\ex X_v}$ induced on the evaluation space $\ex X_v$. In general, $\widetilde{\ex X}_v\subset \ex X$ has an integral vector field $v'$ such that the projection $\widetilde {\ex X}_v\longrightarrow \ex X_v$ can be regarded as a quotient by an action generated by $v'$, and $\Omega_{\ex X_v}$ is defined so that its pullback to $\widetilde {\ex X}_v$ is $2\pi i \iota_{v'}\Omega$.

Given contact data $\mathbf p$, define 
\[
\ex X^{\mathbf p}:=\prod_{v\in M_{\mathbb Z}} \ex X_v^{\mathbf p(v)}\ .
\]
As $\Omega_{\ex X_v}$ is a holomorphic symplectic form on $\ex X_v$, the sum of the pullbacks of these forms defines a holomorphic symplectic form $\Omega_{\ex X^{\mathbf p}}$ on $\ex X^{\mathbf p}$. The significance of this holomorphic symplectic form is that the image of the evaluation map is a holomorphic Lagrangian; see \cref{holomorphic Lagrangian} below.

On $\ex X^{\mathbf p}$, we have an action of the symmetric group $S_{\mathbf p(v)}$ permuting all the factors of $\ex X_v$, and therefore an action of
\[
\Aut \mathbf p:=\prod_v S_{\mathbf p(v)} \ .
\]
There is also an action of $\Aut \mathbf p$ on $\prod_v{G_v}^{\mathbf p(v)}$ permuting the different factors. Let $G_{\mathbf p}$ be the corresponding semi-direct product of $\prod_v S_{\mathbf p(v)}$ with $\Aut\mathbf p$. 
\[
\prod_v G_v^{\mathbf p(v)}\hookrightarrow G_{\mathbf p}\longrightarrow \Aut \mathbf p
\]
There is a natural action of $G_{\mathbf p}$ on $\ex X^\mathbf p$ factoring through the permutation action of $\Aut\mathbf p$. The codomain of the evaluation map from the moduli stack of holomorphic curves with contact data $\mathbf p$ is the quotient stack
\[
\ex X^{\mathbf p}/G_{\mathbf p}\ .
\]
For a vector $v\in M_{\mathbb Z} $, let $\abs v$ denote the positive integer such that $v/\abs v$ is a primitive integral vector, so $G_v=\mathbb Z_{\abs v}$. We have
\[
\abs{G_{\mathbf p}}=\prod_v \abs v^{\mathbf p(v)}\mathbf p(v)! \ .
\]

The moduli stack of holomorphic curves in $\expl(\bar X, D)$ with contact data $\mathbf p$ has a natural evaluation map $ev$ with codomain $\ex X^{\mathbf p}/G_\mathbf p$. Taking the fibre product\footnote{More explicitly, $\mathcal M_{\mathbf p}$ is the moduli stack of holomorphic curves in $\expl (\bar X, D)$ with contact data $\mathbf p$, and with labelled asymptotic markers at each contact point. In the language of exploded manifolds, each contact point corresponds to an end of the holomorphic curve. An asymptotic marker at a contact point of order $k$ consists of a choice of coordinate $\widetilde z$ on this end such that the pullback of the function $\widetilde z^\beta$ is $\widetilde z^k$.} with $\ex X^{\mathbf p}\longrightarrow \ex X^{\mathbf p}/G_{\mathbf p}$ gives a $|G_{\mathbf p} |$--fold cover $\mathcal M_{\mathbf p}$ of this moduli stack with a natural evaluation map 
\[
ev:\mathcal M_{\mathbf p}\longrightarrow \ex X^{\mathbf p} \ .
\]
The complex virtual dimension of $\mathcal M_{\mathbf p}$ is $\abs{\mathbf p}:=\sum_v\mathbf p(v)$, which is half the dimension of $\ex X^{\mathbf p}$.

\begin{remark} \label{holomorphic Lagrangian}
The image of $ev$ is a holomorphic Lagrangian subset of $(\ex X^{\mathbf p}, \Omega_{\ex X^{\mathbf p}})$, so $ev^*\Omega_{\ex X^{\mathbf p}}=0$. In particular, this implies that the complex dimension of the image of $ev$ is at most $\abs{\mathbf p}$. An analogous result holds for all log Calabi--Yau threefolds; see \cite{lag} for full details. To see that $ev^*\Omega_{\ex X^{\mathbf p}}$ vanishes, pull back $\Omega_{\bar X}$ to the total space of a family of holomorphic curves parametrised by some manifold $M$, using $\Omega'_{\bar X}$ to indicate this pullback. Let $v$ and $w$ be two vector fields on $M$, with lifts $\widetilde v$ and $\widetilde w$ to vector fields on the total space of the family. Because $\Omega_{\bar X}$ is holomorphic and fibres are holomorphic, the restriction of $\iota_{\widetilde v}\iota_{\widetilde w}\Omega'_{\bar X}$ to fibres does not depend on the choice of lift, and a local calculation implies that on fibres, $\iota_{\widetilde v}\iota_{\widetilde w}\Omega'_{\bar X}$ is a holomorphic $1$-form, or a meromorphic 1-form with simple poles at each contact point, when we use the usual cotangent space on fibres instead of the logarithmic cotangent space. Then, $ev^*\Omega_{\ex X^{\mathbf p}}(v,w)$ is the sum of the residues of this meromorphic form, which is zero by Stokes' theorem.

If we break our contact data into $\mathbf p+\mathbf q$, with $\mathbf p$ considered incoming contact data and $\mathbf q$ as outgoing contact data, we can think of the image of $ev$ in $\ex X^{\mathbf p+\mathbf q}=\ex X^{\mathbf p}\times \ex X^{\mathbf q}$ as a holomorphic Lagrangian correspondence between $(\ex X^{\mathbf p}, -\Omega_{\ex X^{\mathbf p}})$ and $(\ex X^{\mathbf q},\Omega_{\ex X^{\mathbf q}})$. In particular, applying a holomorphic Lagrangian constraint $\mathcal L_p\subset \ex X^{\mathbf p}$ to holomorphic curves in $\mathcal M_{\mathbf p+\mathbf q}$, the image of the constrained moduli stack in $\ex X^q$ is also a holomorphic Lagrangian subset. For this reason, it is natural to use holomorphic Lagrangian constraints on our curves.

If we instead used special Lagrangian constraints, such as Aganagic--Vafa branes --- with Lagrangian now referring to the restriction of the ordinary symplectic form instead of $\Omega_{\ex X^{\mathbf p}}$, and special meaning that the imaginary part of $\Omega_{\ex X^{\mathbf p}}$ vanishes --- we would instead get the weaker result that the imaginary part of $\Omega_{\ex X^{\mathbf q}}$ vanishes on the image of the constrained curves, and the image of this evaluation map might no longer be holomorphic.
\end{remark}

\section{Relative Gromov--Witten invariants} \label{rgw}

Given a cohomology class in $H^{2\abs {\mathbf p}}(\ex X^{\mathbf p};\mathbb R)$ represented by a differential form $\alpha$, we can pull back $\alpha$ and integrate over the moduli stack of holomorphic curves to define a relative Gromov--Witten invariant. More generally, we can take $\alpha$ to be a differential form on a refinement of $\ex X^{\mathbf p}$, defining a class in the refined cohomology $\rh^{2\abs {\mathbf p}}(\ex X^{\mathbf p})$ ~\cite[Section 9]{dre}. Given a non-negative integer $g$ and a homology class $\beta\in H_2(\bar X;\mathbb Z)$, let $\mathcal M_{g,\mathbf p,\beta}\subset \mathcal M_{\mathbf p}$ denote the moduli stack of connected holomorphic curves with genus $g$, labelled contact data $\mathbf p$ and homology class $\beta$. Define the Gromov--Witten invariant
\[
\langle\alpha\rangle_{g,\mathbf p,\beta}:=\int_{[\mathcal M_{g,\mathbf p,\beta}]}ev^*\alpha\in \mathbb R \ ,
\]
where we can use the formalism from~\cite{vfc} to define the virtual fundamental class $[\mathcal M_{g,\mathbf p,\beta}]$ and integrate the differential form $ev^*\alpha$ over it. Also using the formalism from~\cite{vfc}, the Poincar\'{e} dual of the pushforward of $[\mathcal M_{g,\mathbf p,\beta}]$ is the cohomology class
\[
\eta_{g,\mathbf p,\beta}:=ev_! (1) \in \rh^{2\abs{\mathbf p}}(\ex X^{\mathbf p}) \ ,
\]
so it follows that
\[
\langle\alpha\rangle_{g,\mathbf p,\beta}=\int_{\ex X^{\mathbf p}}\alpha \wedge \eta_{g,\mathbf p,\beta} \ .
\]

Without using refined cohomology, the corresponding invariant would not capture the full relative Gromov--Witten invariants. For manifolds with normal crossing divisors, capturing the full relative Gromov--Witten invariants requires performing blowups on the boundary divisors until the image of $ev$ intersects boundary strata transversely. In this paper, we mainly consider moduli spaces whose image already intersects boundary strata transversely.

It is convenient to introduce formal parameters $\hbar$ and $t$ so that these invariants can be packaged in the formal sum
\begin{equation} \label{etadef}
\eta:=\sum_{\mathbf p} \eta_{\mathbf p} \ , \quad \text{where} \quad \eta_{\mathbf p}:=\sum_{g,\beta} \hbar^{2g-2+\abs{\mathbf p}} \, t^{\int_\beta\omega} \, \eta_{g,\mathbf p,\beta} \ .
\end{equation}
Here, we set $\abs{\mathbf p}:=\sum_v\mathbf p(v)$ to be the number of contact points, so removing these points gives a curve with Euler characteristic $-(2g-2+\abs{\mathbf p})$.

\section{The Novikov ring \texorpdfstring{$\N$}{N}} \label{nr}

To circumvent issues of convergence in infinite sums such as those appearing in \cref{etadef}, we work over a suitably chosen Novikov ring. Let $\N$ be the Novikov ring of formal sums
\[
\sum_{\chi\in \mathbb Z}\sum_{y\in [0,\infty)}c_{-\chi,y} \, \hbar^{-\chi} \, t^y
\]
with coefficients $c_{-\chi, y}\in\mathbb R$, such that for any $M\in\mathbb R$, there are only finitely many nonzero coefficients $c_{-\chi,y}$ with both $-\chi$ and $y$ bounded above by $M$. This ring $\N$ has a $\mathbb Z \times [0,\infty)$-grading by $(-\chi, y)$. Multiplication in $\N$ is well-defined without any notion of limits because the calculation of any coefficient in a product only involves a finite sum.

We introduce the following natural terminology for $\N$-modules.

\begin{defn} Let $\N$ be the Novikov ring introduced above.
\begin{itemize}
\item Define a {\bf graded $\N$-module} $A$ to be an $\N$-module with a $\mathbb Z\times[0,\infty)$-grading that is compatible with the grading of $\N$. For $v\in A$, use the notation $v_{-\chi,y}\in A_{-\chi,y}\subset A$ for the part of $v$ with grading $(-\chi,y)$, and use $v_{\leq M}\subset A_{\leq M}$ for the part of $v$ with grading $(-\chi,y)$ such that $-\chi \leq M$ and $y \leq M$. 

\item Define a {\bf bounded $\N$-module} $A$ to be a graded $\N$-module such that for all $v\in A$ and $M\in \mathbb R$, $v_{-\chi,y}$ is $0$ for all but finitely many $(-\chi,y)$ with $-\chi \leq M$ and $y \leq M$. Call $S \subset A$ a basis for a bounded $\N$-module if for all $v\in A$, $v_{\leq M}$ can be written uniquely as a finite $\N$--linear combination of elements of $S$.

\item Define the {\bf completion} of a bounded $\N$-module $A$ to be the limit of the $\mathbb R$-modules $A_{\leq M}$ and observe that the completion is itself a bounded $\N$-module. 

\item An $\N$-module homomorphism $\phi:A\longrightarrow B$ is {\bf graded} if it preserves the grading. It is {\bf bounded} if for all $M \in \mathbb R$, there exists an $M'$ such that $\phi^{-1}(B_{\leq M}) \subset A_{\leq M'}$.

\item Define the {\bf tensor product} of two bounded $\N$-modules $A$ and $B$ to be the completion of the tensor product of $A$ and $B$ as modules over the ring $\N$.

\item Define a {\bf bounded $\N$-algebra} $A$ to be an algebra over $\N$ such that $A$ is a bounded $\N$-module, the inclusion $\N\longrightarrow A$ is a graded homomorphism, and the multiplication $A \otimes A \longrightarrow A$ is a graded homomorphism.
\end{itemize}
\end{defn}

The completion of a bounded $\N$-module $A$ effectively allows infinite sums, as long as only finitely many terms are in $A_{\leq M}$ for each $M$. The grading on $A\otimes B$ is such that 
\[
(v\otimes w)_{-\chi,y}=\sum_{\substack{\chi_1+\chi_2=\chi \\ y_1+y_2=y} }v_{-\chi_1,y_1}\otimes w_{-\chi_2,y_2} \ , 
\]
which is a finite sum because $A$ and $B$ are assumed to be bounded $\N$-modules.

\section{The state space \texorpdfstring{$\mathcal H$}{H}} \label{ss} 

Let $\mathcal H$ be the completion of the bounded $\N$-module
\[
\bigoplus_{\mathbf p}\rh^{2\abs{\mathbf p}}(\ex X^{\mathbf p}/G_{\mathbf p};\N) := \bigoplus_{\mathbf p}\rh^{2\abs{\mathbf p}}(\ex X^{\mathbf p}/G_{\mathbf p}) \otimes_{\mathbb R} \N \ ,
\]
where we identify $\rh^*(\ex X^{\mathbf p}/G_{\mathbf p})$ with the refined cohomology defined using $(\Aut\mathbf p)$--invariant differential forms on refinements of $\ex X^{\mathbf p}$. The bounded $\N$-module $\mathcal H$ will play the role of a bosonic Fock space of states, and should be thought of as the space of possible constraints for our Gromov--Witten invariants. We shall see below that the integration pairing induces a nondegenerate bilinear pairing on $\mathcal H$ and there is a natural commutative multiplication on $\mathcal H$, where the constraint $\alpha \beta \in \mathcal H$ corresponds to applying both the constraints $\alpha\in\mathcal H$ and $\beta \in \mathcal H$. 

Since $\eta_{\mathbf p}$ is invariant under the action of $G_{\mathbf p}$ on $\ex X^{\mathbf p}$, we can consider the Gromov--Witten invariant $\eta$ from \cref{etadef} as an element of $\mathcal H$. Gromov compactness implies that $\eta_{\leq M}$ is a finite sum, so $\eta$ certainly lies in $\mathcal H$.

For $\alpha\in \mathcal H$, use the notation $\alpha_{\mathbf p}$ for the pullback of $\alpha$ to $\rh^{2\abs{\mathbf p}}(\ex X^{\mathbf p};\N)$. A special case of this is given by $\alpha_0\in H^0(\ex X^0;\N)=\N$. The $\N$-module $\mathcal H$ should be thought of as a generalisation of a commutative Fock space on the completion of $\bigoplus_{v\in M_{\mathbb Z}} \rh^{2}(\bar X_v/G_v; \N)$. Define the $\N$-bilinear integration pairing $\mathcal H\times\mathcal H\longrightarrow \N$ by 

\[\ip{\alpha}{\beta} := \int_{\coprod_{\mathbf p}\ex X^{\mathbf p}/G_ {\mathbf p}} \alpha \wedge \beta :=\sum_{\mathbf p}\frac 1{\abs {G_{\mathbf p}}}\int_{\ex X^{\mathbf p}} \alpha_{\mathbf p} \wedge \beta_{\mathbf p} \ .
\]
This integration pairing defines a graded $\N$-module homomorphism $\mathcal H\otimes\mathcal H\longrightarrow\N$.

The bounded $\N$-module $\mathcal H$ is also a commutative $\N$-algebra, with product induced from 
\[
\alpha\beta:=\frac 1{\abs{G_{\mathbf p}} \abs {G_{\mathbf q}}} \sum_{\sigma \in G_{\mathbf p+\mathbf q}}\sigma^*(\pi_1^*\alpha\wedge \pi_{2}^*\beta)\in \rh^{2\abs{\mathbf p+\mathbf q}}(\ex X^{\mathbf p+\mathbf q};\N) \ ,
\]
where $\alpha\in \rh^{2\abs{\mathbf p}}(\ex X^{\mathbf p};\N)$, $\beta \in \rh^{2\abs {\mathbf q}}(\ex X^{\mathbf q};\N)$, and $\pi_{1}$ and $\pi_{2}$ are the projections onto the factors of $\ex X^{\mathbf p+\mathbf q}=\ex X^{\mathbf p}\times \ex X^{\mathbf q}$. As this formula sends symmetric forms to symmetric forms, it induces a map
\[
\rh^*(\ex X^{\mathbf p}/G_{\mathbf p};\N) \times \rh^*(\ex X^{\mathbf q}/G_{\mathbf q};\N) \longrightarrow \rh^*(\ex X^{\mathbf p+\mathbf q}/G_{\mathbf p+\mathbf q};\N) \ .
\]
We can extend this map to a multiplication that is compatible with grading, thus endowing $\mathcal H$ with the structure of a bounded commutative $\N$-algebra.
\[
(\alpha\beta)_{\mathbf r}:=\sum_{\mathbf p+\mathbf q=\mathbf r}\frac 1{\abs{G_{\mathbf p}}\abs {G_{\mathbf q}}}\sum_{\sigma \in G_{\mathbf p+\mathbf q}}\sigma^*(\pi_1^*\alpha_{\mathbf p}\wedge \pi_{2}^*\beta_{\mathbf q})
\]

Since $\eta_{\leq 0}=0$, we can define
\[
\exp \eta:= 1+\eta+\frac 1{2!}\eta^2+\frac 1{3!}\eta^3+\dotsb \in\mathcal H \ ,
\]
which can be thought of as a partition function encoding Gromov--Witten invariants that count possibly disconnected curves, with $1\in \rh^0(\ex X^0;\N)=\N$ representing the empty holomorphic curve.

Note that multiplication by $\beta\in \mathcal H$ is a bounded $\N$-module homomorphism $\mathcal H\longrightarrow \mathcal H$. We can define a bounded $\N$-module homomorphism $\an_\beta:\mathcal H\longrightarrow \mathcal H$ that is adjoint\footnote{Note that $\an_\beta$ is adjoint using the integration pairing, which is not positive definite. In \cref{ppdef} below, we twist this integration pairing on a constrained state space to obtain a positive definite inner product $\pp{\cdot}{\cdot}$. } to multiplication by $\beta$ via
\begin{equation}\label{adef}
\left( \an_{\beta}(\alpha) \right)_{\mathbf q}:=\sum_{\mathbf p}\frac1{\abs{G_{ \mathbf p}}}(\pi_2)_!\lrb{ \alpha_{\mathbf p+\mathbf q}\wedge \pi_1^*\beta_{\mathbf p}} \ .
\end{equation}
This map is bounded and sends symmetric forms to symmetric forms, so it induces the required bounded map $\an_\beta:\mathcal H\longrightarrow\mathcal H$. One can think of multiplication by $\beta \in \mathcal H$ as analogous to a creation operator on a Fock space, while $\an_\beta$ is analogous to an annihilation operator.

The following calculation serves as a check that $\an_\beta$ is indeed adjoint to multiplication by $\beta$. For $\gamma\in \mathcal H$,
\begin{align*}
\ip{ \alpha}{\beta\gamma} &= \sum_{\mathbf p+\mathbf q}\frac 1{\abs{G_{\mathbf p+\mathbf q}}}\int_{\ex X^{\mathbf p+\mathbf q}}\alpha_{\mathbf p+\mathbf q}\wedge \frac 1{\abs{G_{ \mathbf p}}\abs{G_{ \mathbf q}}}\sum_{\sigma\in G_{\mathbf p+\mathbf q}}\sigma^*(\pi_1^*\beta_{\mathbf p}\wedge\pi_2^*\gamma_{\mathbf q})
\\
& = \sum_{\mathbf p+\mathbf q}\frac 1{\abs{G_{\mathbf p}}\abs{G_{\mathbf q}}}\int_{\ex X^{\mathbf p+\mathbf q}}\alpha_{\mathbf p+\mathbf q}\wedge \pi_1^*\beta_{\mathbf p}\wedge\pi_2^*\gamma_{\mathbf q} \\
& = \sum_{\mathbf p+\mathbf q}\frac 1{\abs{G_{\mathbf q}}}\int_{\ex X^{\mathbf q}}\frac 1{\abs{G_{\mathbf p}}}(\pi_2)_!(\alpha_{\mathbf p+\mathbf q}\wedge \pi_1^*\beta_{\mathbf p})\wedge \gamma_{\mathbf q} \\
&=\ip {\an_{\beta}(\alpha)}{\gamma} \ .
\end{align*}

Note that for $\alpha, \beta, \eta \in \mathcal H$, we have the equations
\[
\an_{\alpha\beta}=\an_\alpha \an_\beta, \quad \an_{\alpha+\beta}=\an_\alpha+\an_\beta, \quad \text{and} \quad \an_{\exp \eta}=\exp \an_{\eta}\ .
\]

\section{The constrained state space \texorpdfstring{$\mathcal H^\pm$}{H+ / H-}} \label{css}

Consider curves with contact data $\mathbf p$ supported on $\mathfrak t^2_{\mathbb Z}\setminus \{0\}$. Each connected component of such a curve is contained in a level set of $f$. If $\sum_v\mathbf p(v)v\neq 0$, then the only connected holomorphic curves in $\bar X$ with contact data $\mathbf p$ are contained in $f^{-1}(0)\subset\bar X$, because the other fibres of $f$ are toric manifolds and holomorphic curves in toric manifolds have balanced contact data that sums to $0$. Within $\bar X_v$, $f^{-1}(0)$ has complex dimension $1$, consisting of $E^+$ and $E^-$ if $v$ is a positive multiple of some leg framing $w_\ell$, and consisting of a single sphere otherwise. Accordingly, $\prod f^{-1}(0)\subset \ex X^{\mathbf p}$ is a $\abs{\mathbf p}$-dimensional holomorphic Lagrangian subvariety of $(\ex X^{\mathbf p},\Omega_{\ex X^{\mathbf p}})$, with different irreducible components depending on whether $E^+$ or $E^-$ is used within each $\bar X_\ell$. Moreover, the pushforward of $[\mathcal M_{g,\mathbf p,\beta}]$ represents a rational sum of the homology classes represented by these connected components.

Similarly, if we constrain at least one contact point to  $f^{-1}(0)$, the image of the evaluation map at the other contact points will be contained in $f^{-1}(0)$.

For a leg $\ell$, define $\mathcal H_\ell^{+}\subset\mathcal H$ and $\mathcal H^-_\ell \subset \mathcal H$ to be the completion of the $\N$--subalgebra generated by the Poincar\'{e} duals $\alpha^\pm_{k,\ell}$ to the maps $\expl E^{\pm}\longrightarrow \ex X_{kw_\ell}/G_{kw_\ell}$. There is an orthogonal basis for $\mathcal H_\ell^\pm$ defined by
\[
\alpha_\ell^{\pm\mathbf p}:=\prod_{k=1}^\infty(\alpha^\pm_{k,\ell})^{p_k} \ .
\] 
We have
\[
\ip{\alpha_\ell^{+\mathbf p}} {\alpha_\ell^{-\mathbf p}} =\abs{G_{\mathbf p}}=\prod_k k^{p_k}p_k! \qquad \text{and} \qquad \ip{\alpha_\ell^{\pm\mathbf p}}{\alpha_\ell^{\pm\mathbf p}}=\prod_k (-k)^{p_k}p_k! \ .
\]

There exists a canonical orthogonal projection to $\mathcal H_\ell^{\pm}$ given by
\begin{align*}
\mathcal H & \longrightarrow \mathcal H_\ell^{\pm} \\
\beta & \longmapsto (\beta)^\pm_\ell:= \sum_{\mathbf p}\lrb{\ip{ \beta} {\alpha_\ell^{\mp\mathbf p}} }\frac 1{\abs G_{\mathbf p}} \alpha_\ell^{\pm\mathbf p} \ .
\end{align*}
The kernel of this projection consists of $\beta$ such that $\an_{\beta}$ vanishes on $\mathcal H_\ell^\pm$, and is hence an ideal. It follows that this projection is a graded $\N$-algebra homomorphism. Note that restricted to $\mathcal H_\ell^\mp$, this projection defines an isomorphism from $\mathcal H_\ell^\mp$ to $\mathcal H^\pm_\ell$ sending $\alpha^\mp_{k,\ell}$ to $-\alpha^{\pm}_{k,\ell}$.

Similarly, define $\mathcal H^\pm\subset \mathcal H$ to be the completion of the subalgebra generated by $\mathcal H^\pm_\ell$ for all legs $\ell$. In particular, 
\[
\mathcal H^\pm=\bigotimes_\ell \mathcal H^{\pm}_\ell\ . 
\]

There is also a canonical projection $\mathcal H\longrightarrow\mathcal H^\pm$ which is a graded $\N$-algebra homomorphism given by 
\begin{equation} \label{-projdef}
\beta\mapsto (\beta)^\pm:=\sum_{\{\mathbf p^\ell\}}\lrb{ \ip{\beta} {\prod_\ell \alpha^{\mp\mathbf p^\ell}_\ell }} \prod_\ell \frac 1{\abs {G_{\mathbf p^\ell}}}\alpha^{\pm\mathbf p^\ell}_\ell
\end{equation}
and a similarly defined projection $\mathcal H \longrightarrow\bigotimes_{\ell \in I} \mathcal H^\pm_\ell$ given by
\begin{equation}\label{projdef2}
\beta \longmapsto (\beta)^\pm_I \ .
\end{equation}
Here, $I$ is any subset of legs, with $(\beta)^\pm_I$ defined using \cref{-projdef} with the sum restricted to only using $\mathbf p^\ell$ for $\ell \in I$.

As well as the orthogonal projection from $\mathcal H^\pm$ to $\mathcal H^\mp $, there is another natural graded $\N$-algebra isomorphism
\[
\text{Conj}:\mathcal H^\pm\longrightarrow \mathcal H^\mp
\]
such that 
\[
\text{Conj}(\alpha^{\pm\bf p})=\alpha^{\mp\bf p} .
\]
Composing this isomorphism with the projection $\mathcal H^\mp\longrightarrow \mathcal H^{\pm}$ gives an involution that acts a little like complex conjugation. In fact, \cref{translation} gives an isomorphism between $\mathcal H^+_\ell\otimes \mathbb C$ and the infinite wedge space. Under this isomorphism, our involution corresponds to the anti-complex involution preserving the standard orthonormal basis for the infinite wedge space. Twisting the integration product by this involution gives an important positive definite metric on $\mathcal H^{\pm}$ 
\begin{equation} \label{ppdef}
\pp{\alpha}{\beta}:=\ip{\text{Conj}(\alpha)}{\beta} \ .
\end{equation}
This pairing is symmetric and positive definite. In particular, $\alpha_\ell^{\pm\mathbf p}$ form an orthogonal basis for $\mathcal H^{\pm}_\ell$, with
\[
\pp{\alpha_\ell^{\pm\mathbf p}} {\alpha_\ell^{\pm\mathbf p}} =\abs{G_{\mathbf p}}=\prod_k k^{p_k}p_k! \ .
\]

We encode Gromov--Witten invariants counting possibly disconnected curves constrained to $E^+_\ell$ in a partition function $Z_{\bar X}\in \mathcal H^-$, by projecting $\exp \eta$ to $\mathcal H^-\subset\mathcal H$.
\begin{equation} \label{Zdef}
Z_{\bar X}:=(\exp\eta)^-=\exp\lrb{\sum_{\{\mathbf p^\ell\},g,y}n_{\{\mathbf p^\ell\},g,y}\hbar^{2g-2} t^y\prod_\ell\hbar^{\abs{\mathbf p^\ell}}\alpha^{-\mathbf p^\ell}_{\ell}}
\end{equation}

Where, in the above, $n_{\{\mathbf p^\ell\},g,y}\in\mathbb Q$ is the Gromov--Witten invariant counting connected holomorphic curves with genus $g$, $\omega$-energy $y$, and contact data $\{\mathbf p^\ell\}$, always constrained to $E^+_\ell$ in $\bar X_\ell$, and not contacting any other component of the divisor.

Given $\beta\in \mathcal H$, define $Z_{\bar X,\beta}\in \mathcal H^-$ as the projection of $\an_\beta\exp\eta$, where $\an_\beta$ is the integration-pairing adjoint to multiplication by $\beta$ defined in \cref{adef}.
 \begin{equation} \label{Zbetadef}
Z_{\bar X,\beta}:=(\an_\beta\exp \eta)^-=\sum_{\{\mathbf p^\ell\}}\lrb{\ip{ \exp\eta}{\beta \prod_\ell \alpha_\ell^{+\mathbf p^\ell}}}\prod_\ell \frac 1{\abs{G_{\mathbf p^\ell}}}\alpha^{-\mathbf p^\ell}_\ell
\end{equation} 
This is a kind of partition function counting holomorphic curves constrained to $E^+_\ell$ at some contact points, and constrained using $\beta$ at the remaining contact points. The map $\beta \mapsto Z_{\bar X,\beta}$ is a bounded $\N$-module homomorphism $\mathcal H\longrightarrow\mathcal H^-$.

Given any $\alpha\in \mathcal H$, and a subset $I$ of legs, $\alpha$ induces a bounded $\N$-module homomorphism
\begin{equation} \label{opdef}
\op \alpha{I}:\mathcal H\longrightarrow \bigotimes_{\ell \in I} \mathcal H^-_\ell
\end{equation}
defined by 
\[
\op\alpha I (\beta)=(\an_\beta \alpha)^-_I \ .
\]

We will use this to interpret various Gromov--Witten invariants as operators; for example, if we divide legs into outgoing legs in $I$, and incoming legs in $I'$, restricting $\op {Z_{\bar X}} I$ to $\bigotimes_{\ell'\in I'} \mathcal H^+_{\ell'}$ defines a bounded $\N$-module homomorphism
 \[ \op {Z_{\bar X}} I:\bigotimes_{\ell'\in I'} \mathcal H^+_{\ell'}\longrightarrow \bigotimes_{\ell\in I} \mathcal H^-_{\ell}\]

\section{The gluing formula for the partition function \texorpdfstring{$Z_{\bar X}$}{Z}} \label{gfs}

Given a toric degeneration of $\bar X$ into $\bigcup _i\bar X_i$, there is a simple gluing formula for $Z_{\bar X}$ in terms of $Z_{\bar X_i}$, reformulating the gluing formula from~\cite{LLLZ,AKMV}. In particular, the legs of the toric graph of $\coprod_i \bar X_i$ consist of legs from $\bar X$, and matched pairs of new legs, $k^+$ and $k^-$ where some $\bar X_i$ meets some $\bar X_j$ in a stratum $(\bar X_i)_{k^+}=(\bar X_j)_{k^-}$. At such a matched pair of legs, the framings are opposite: $w_{k^-}=-w_{k^+}$. Hence, the homology classes we use to define $\mathcal H^\pm_{k^{\pm}}$ are opposite: $[E_{k^-}]=-[E_{k^+}]$ and $\alpha^+_{k^-}=\alpha^-_{k^+}$. With this identification of $\mathcal H^\mp_{k^\pm}$, the integration pairing gives a natural graded $\N$-module homomorphism
\[\mathcal H^-_{k^+}\otimes \mathcal H^-_{k^-}\longrightarrow \N \] 
\[\alpha^{-\bf p}_{k^+}\otimes \alpha^{-\bf q}_{k^-}\mapsto \ip {\alpha^{-\bf p}_{k^+}}{\alpha^{+\bf q}_{k^+}} =\pp {\alpha^{-\bf p}_{k^+}}{\alpha^{- \bf q}_{k^+}} \]
and these homomorphisms then induce a graded $\N$-module homomorphism
\[
\pi: \bigotimes_\ell \mathcal H^-_\ell \bigotimes_k \Big( \mathcal H^-_{k^+} \otimes \mathcal H^-_{k^-} \Big) \longrightarrow \bigotimes_\ell \mathcal H^-_\ell \ .
\] 
The gluing formula for $Z_{\bar X}$ is then simply
\begin{equation} \label{glue}
Z_{\bar X} = \pi \bigg( \bigotimes_iZ_{\bar X_i} \bigg) \ .
\end{equation}

The proof of \cref{glue} can be found in~\cite{LLLZ}. One can also see~\cite{Ranganathan, acgsdegeneration, Tehrani, kim-lho-ruddat, mandel-ruddat} for different approaches to such a gluing formula that use different formalisms to define relative Gromov--Witten invariants. However, given our geometric setup, \cref{glue} can also be deduced from the tropical gluing formula of~\cite[Equation~(1)]{gfgw}. We reproduce the formula here for the convenience of the reader, since we will invoke it at various times below.
\begin{equation} \label{eq:tgf}
\eta \tc \gamma = \frac{k_\gamma}{|\Aut \gamma|} \iota_!^{[\gamma]} \Delta^* \prod_v \eta^{[\gamma_v]}
\end{equation}
Here, $\eta$ represents a Gromov--Witten invariant and the notation $\tc \gamma$ indicates the contribution of a tropical curve $\gamma$ to this invariant. The term $\eta^{[\gamma_v]}$ represents a relative Gromov--Witten invariant associated to the vertex $v$ of $\gamma$. The right side takes the form of a `pull-push formula', which one can think of as elementary instructions for gluing together these relative invariants. The prefactor on the right side is essentially combinatorial in nature, taking into account edge multiplicities and symmetries of the tropical curve $\gamma$. Rather than describe the tropical gluing formula in full detail and generality, we explicitly identify all elements of the formula required for our purposes below, particularly in the proof of \cref{zc}. As a word of warning, we flag the fact that \cref{eq:tgf} cannot be used verbatim in our context. A minor adjustment to the combinatorial factor needs to be made because the evaluation stacks in the present paper are quotients of the evaluation spaces used in \cite{gfgw}.

The following is a brief sketch of how the gluing formula of \cref{glue} for the partition function $Z_{\bar X}$ can be deduced from the tropical gluing formula of \cref{eq:tgf}. Exploding the toric degeneration provides a smooth family of exploded manifolds containing $\expl (\bar X,D)$ and an exploded manifold $\ex X$ with smooth part the union of the $\bar X_i$. The tropical gluing formula of \cref{eq:tgf} provides a gluing formula for $\eta$ as a sum over tropical curves in the tropical part $\totb{ \ex X}$ of $\ex X$, and this tropical gluing formula implies a slightly simpler gluing formula for $\exp\eta$, where it is not necessary to keep track of how tropical curves are connected together. Analogously to the tropical part of $\expl (\bar X,D)$, we can put a global $\mathbb Z$-affine structure on $\totb{\ex X}$ with singular locus a graph with a vertex for each $\bar X_i$, an internal edge in direction $w_{k^+}$ for each matched pair of new legs $k^\pm$, and singular rays in the directions $w_\ell$, as in the tropical part of $\expl(\bar X,D)$. For computing the projection of $\exp\eta$ to $\mathcal H^-$ and hence $Z_{\bar X}$, constrain our curves to (the equivalent of) $E^+_\ell$. The only tropical curves with nonzero contribution consist of tropical curves with image contained in the singular locus, and all vertices at vertices of the singular locus, and the formula for the contribution of all such curves is analogous to \cref{glue}, except the pairing between $\mathcal H^-_{k^+}$ and $\mathcal H^-_{k^-}$ is replaced by integration over the relevant evaluation space without first projecting to $\mathcal H^-$. This still gives rise to the same formula as \cref{glue}, because once we have constrained our curves to $E^{+}_\ell$, the cohomology classes we have to integrate are contained in $\mathcal H_{k^\pm}^-\otimes\mathcal H_{k^\pm}^+$, and orthogonal projection to $\mathcal H^-$ then does not affect the integration pairing.

\section{Analysis of the empty tropical graph} \label{toriccase}

In the case that $X=\mathbb C\times (\mathbb C^*)^2$, we may take the function $f$ to be the first coordinate, and choose $\bar X$ as $(\mathbb{CP}^1)^3$, or the product of $\mathbb{CP}^1$ with any toric compactification of $(\mathbb C^*)^2$. In this case, the toric graph is empty. As there is a toric structure on $\mathbb{CP}^1$ such that $f-1$ is a primitive toric monomial, $((\mathbb{CP}^1)^3,D)$ is isomorphic to $(\mathbb{CP}^1)^3$ with its toric boundary divisor. Relative Gromov--Witten invariants of such three-dimensional toric manifolds are calculated using a tropical gluing fomula in~\cite{3d}.

For $v\in M_{\mathbb Z}=\mathfrak t^3_{\mathbb Z}\setminus \{0\}$, and $\theta\in(\mathfrak t^3_{\mathbb Z})^*$ such that $\theta(v)=0$, there is a distinguished class $h_{v,\theta}\in \mathcal H$ defined as follows: First, blow up $(\mathbb{CP}^1)^3$ using toric blowups until $v/{\abs v}$ corresponds to a codimension $1$ toric boundary stratum $S$, and the monomial $z^\theta$ extends to a holomorphic map to $\mathbb{CP}^1$. Then $h_{v/\abs v,\theta}$ is the pullback of the Poincar\'{e} dual to a point in $\mathbb{CP}^1$ using $z^\theta\rvert_{S}:S\longrightarrow \mathbb{CP}^1$. As $\expl S$ is a refinement of $\ex X_v$, we have that $h_{v/\abs{v},\theta}$ defines a class in the refined cohomology of $\ex X_v$. Define $h_{v,\theta}\in\mathcal H$ to be the pushforward of this class to $\ex X_v/G_v$. Note that $h_{v,k\theta}=\abs kh_{v,\theta}$.

For $\alpha\in\mathcal H$ a product of these classes $h_{v,\theta}$, we can write some examples of $Z_{(\mathbb{CP}^1)^3,\alpha}$. Apart from the exponent of $t$ that records $\omega$-energy, these Gromov--Witten invariants do not depend on which three-dimensional toric manifold is used.  Moreover, the exponent of $t$ has a particularly simple dependence on the contact data in the case of $(\mathbb{CP}^1)^3$. Suppose that the $\omega$-area of the $k$th copy of $\mathbb{CP}^1$ is $x_k$ and, for $v \in \mathfrak t^3_\mathbb Z=\mathbb Z^3$, define $2x(v)=x_1\abs v_1+x_2\abs v_2+x_3\abs v_3$. Then the $\omega$-energy of a curve with contact data $\mathbf p$ is $\sum_v\mathbf p(v)x(v)$.

We have the equations
\[
Z_{(\mathbb{CP}^1)^3}=1\in\N \qquad \text{and} \qquad Z_{(\mathbb{CP}^1)^3,h_{v,\theta}}=0 \ ,
\]
which reflect that the virtual count of curves with empty or unbalanced contact data is $0$. One can also observe that the moduli space of genus zero curves with exactly two contact points consists of a compactification of the space of monomial maps $\mathbb C^*\longrightarrow (\mathbb C^*)^3$, and that the corresponding virtual moduli space of higher genus curves vanishes. This gives the equation
\begin{equation} \label{toric2leg}
Z_{(\mathbb{CP}^1)^3,h_{v,\theta}h_{w,\gamma}}=\delta_{v+w}\abs{\theta\wedge\gamma}t^{x(v)+x(w)} \ ,
\end{equation}
where $\abs{\theta\wedge\gamma}$ indicates the smallest nonnegative integer $k$ such that the integral vector $\theta\wedge\gamma$ is $k$ times a primitive integral vector.

The virtual moduli space of curves with three contact points does contain interesting contributions from higher genus curves. The calculation of these contributions appears in \cite[Theorem~1.1]{3d} and leads to the equation
\[
Z_{(\mathbb{CP}^1)^3,h_{u,\beta}h_{v,\theta}h_{w,\gamma}}=\delta_{u+v+w}\abs{\beta\wedge\theta\wedge\gamma}t^{x(u)+x(v)+x(w)}2\sin\lrb{\abs{ u\wedge v}\hbar/2} \ .
\]
The result of \cite[Theorem~1.1]{3d} furthermore implies that
\begin{equation} \label{3dvertex}
\an_{h_{v,\theta}h_{w,\gamma}} \eta = t^{x(v)+x(w)+x(-v-w)} 2\sin(\abs{v\wedge w} \hbar/2) \, h_{-v-w,\iota_{v+w}(\theta\wedge\gamma)} + \dotsb \ ,
\end{equation}
where the missing terms count curves with at least four contact points.

Observe that the two previous equations involve factors of
\[
2\sin(n \hbar/2):=n\hbar- \frac {n^3}{2^2 \, 3!}\hbar^3+\frac{n^5}{2^4 \, 5!}\hbar^5-\dotsb \in \N \ ,
\]
which will occur regularly in this work. So we make the slightly unconventional definition\footnote{The notation $[n]_q$ is often used for `quantum integers', for which there are various definitions. The appearance of the factor $-i$ makes our definition unconventional, although convenient for the current setting. The choice of $q^{1/2}$ here differs from the $q^{1/2}$ appearing in the Gromov--Witten/Donaldson--Thomas correspondence by a factor of $i$~\cite{gwdt}. We expect that there is a parallel story involving the relative Donaldson--Thomas invariants defined by Maulik and Ranganathan in place of relative Gromov--Witten invariants~\cite{ldt}.}
\begin{equation} \label{eq:qinteger}
q^{1/2}:=e^{i\hbar/2}\in \N \qquad \text{and} \qquad [n]_q := -i( q^{n/2}-q^{-n/2}) =2\sin(n \hbar/2)\in \N \ .
\end{equation}

\section{Analysis of the tropical graph with no vertices} \label{sec:product}

In the case that $X=\mathbb C^2\times \mathbb C^*$, we can take the function $f$ to be $z_1z_2$. In this case, the toric graph has two legs $\ell_1$ and $\ell_2$, and consists of a single edge with no vertices. The simplest case is when the two legs have opposite framing. Then, we can take $\bar X$ to be the product of a toric compactification of $\mathbb C^2$ with $\mathbb{CP}^1$. Note that not any toric compactification of $\mathbb C^2$ will do, because $z_1z_2$ must extend to a meromorphic function on this compactification. A concrete example of such a compactification is the blowup $Y$ of $(\mathbb{CP}^1)^2$ at the points $(0,\infty)$ and $(\infty,0)$; another example is drawn in \cref{xlpic}.

Consider a connected holomorphic curve in $Y \times \mathbb{CP}^1$ that only touches the boundary divisor in $Y\times\{0,\infty\}$. All such curves have image in a $\mathbb{CP}^1$ fibre. Accordingly, $Z_{Y\times\mathbb{CP}^1}$ is simple to compute and we obtain
\[
Z_{Y\times \mathbb{CP}^1}=\exp \sum_{k=1}^{\infty} \frac {t^{kx}}{k}\alpha^-_{k,\ell_1}\alpha^-_{k,\ell_2}=\sum_{\mathbf p} \frac{t^{x(\mathbf p)}}{\abs{G_{\mathbf p}}}\alpha_{\ell_1}^{-\mathbf p}\alpha_{\ell_2}^{-\mathbf p} \ ,
\]
where $x$ is the $\omega$-area of $\mathbb{CP}^1$ and $x(\mathbf p) = \sum_k xp_k$. 

The corresponding Gromov--Witten invariants counting curves with constraints $\bf p$ at $\ell_1$ and $\bf q$ at $\ell_2$ are given by 
\[
\ip {Z_{Y\times \mathbb CP^1}} {\alpha_{\ell_1}^{+\bf p} \alpha_{\ell_2}^{+\bf q}}= t^{x(\bf p)}\abs{G_{\bf p}}\delta_{\bf p,\bf q} =t^{x(\bf p)}\pp {\alpha_{\ell_1}^{+\bf p}}{\alpha_{\ell_1}^{+\bf q}} \ .
\]

 By thinking of $\ell_1$ as incoming and $\ell_2$ as outgoing, we get a bounded $\N$-module homomorphism 
\[
\Prop_x:\mathcal H^+_{\ell_1}\longrightarrow H^-_{\ell_2} \ ,
\]
defined as the restriction of $\op{Z_{Y\times\mathbb{CP}^1}}{\ell_2}$ to $\mathcal H^+_{\ell_1}$; where \cref{opdef} is used to define $\op{Z_{Y\times\mathbb{CP}^1}}{\ell_2}$ so, for $\alpha\in \mathcal H_{\ell_1}^+$, $\Prop_x\alpha$ is the projection of $Z_{Y\times \mathbb{CP}^1,\alpha}$ to $\mathcal H^-_{\ell_2}$. In particular, 
\[
\Prop_x\alpha_{\ell_1}^{+\mathbf p}=t^{x(\mathbf p)}\alpha_{\ell_2}^{-\mathbf p} \ .
\]

There is a canonical identification of $\mathcal H^{\pm}_{\ell_1}$ with $\mathcal H^{\mp}_{\ell_2}$ from identifying $Y\times 0$ with $Y\times \infty$, so we can also think of $\Prop_x$ as a bounded $\N$-module automorphism of $\mathcal H^+_{\ell_1}$. In particular, this defines a bounded $\N$-module automorphism 
\begin{align*}
\Prop_{x,\ell}:\mathcal H^+_\ell &\longrightarrow \mathcal H^+_\ell \\
\alpha^{+\mathbf p}_\ell &\longmapsto t^{\sum_{k=1}^\infty p_k kx}\alpha^{+\mathbf p}_\ell \ .
\end{align*}
This operator $\Prop_{x,\ell}$ can be thought of as a kind of propagation operator. The equation $\Prop_{x,\ell} \circ \Prop_{y,\ell} = \Prop_{x+y,\ell}$ is an immediate calculation, but also follows from the gluing formula of \cref{glue}.

\section{The algebra of operators \texorpdfstring{$\W_{v,\ell}$}{W}} \label{sec:woperators}

Suppose that $v\in\mathfrak t^2_{\mathbb Z}$ is a primitive integral vector, and let $h_v$ be some class in $\rh^2(\ex X_v)$ such that 
\begin{equation} \label{hvdef}
\int_{f^{-1}(0)}h_v=1 \text{ and, if $v=w_\ell$, } \ip{\alpha_\ell^+}{ h_v}=0\ .
\end{equation}
Similarly, if $v$ is not a primitive vector, define $h_v$ to be the pushforward of such a class $h_{v/\abs v}$ to $\ex X_v/G_v$. The classes $h_{v,\theta}$ in the case $X=\mathbb C \times (\mathbb C^*)^2$ from \cref{toriccase} are examples of such a class when $v\in\mathfrak t^2_{\mathbb Z}$ and $\theta\in (\mathfrak t^3)^*$ is primitive and vanishes on $v$ and the vector $v_f$ corresponding to the component $f^{-1}(1)$ of the divisor. 

In calculations below, we will use that, when $-v=w_{\ell_2}=-w_{\ell_1}$,

\[ \ip{\alpha^+_{k,\ell_1}}{h_{-kv}}=k \ \ \ \text{and } \ip {\alpha^+_{k,\ell_2}}{h_{-kv}}=0 \ . \]

Note that $Z_{\bar X,h_v}$ does not depend on the particular choice of such an $h_v$, because holomorphic curves otherwise constrained to $E^+_{\ell_i}$ with one further contact point corresponding to $v$ are contained in $f^{-1}(0)$, so the cohomology class representing the pushforward of the constrained moduli space to $\ex X_v$ is some linear combination of the Poincar\'{e} dual to $f^{-1}(0)$ and $\alpha_\ell^+$.

Recall the notation from \cref{adef,Zbetadef,opdef}. For $\bar X=Y\times \mathbb{CP}^1$, the operator $\op{Z_{Y\times\mathbb{CP}^1,h_{-v}}}{\ell_2}$ defines a bounded map of $\N$-modules
\begin{align*} 
\widetilde W_{v}:\mathcal H^+_{\ell_1} & \longrightarrow \mathcal H^{-}_{\ell_2} \\
\alpha & \longmapsto (\an_{\alpha}Z_{Y\times\mathbb{CP}^1,h_{-v}})_{\ell_2}^- \ .
\end{align*}
Note the use of $-v$ to indicate that this contact point is thought of as incoming. This map $\widetilde W_v$ depends on the symplectic form chosen on $Y\times\mathbb{CP}^1$, however this dependence is straightforward to calculate. Setting the symplectic form to $0$, and identifying $\mathcal H^-_{\ell_2}$ with $\mathcal H^+_{\ell_1}$ by identifying $Y\times 0$ with $Y\times \infty$, we get a canonical bounded $\N$-module homomorphism
\[
W_{v}:\mathcal H^+_{\ell_1}\longrightarrow \mathcal H^+_{\ell_1} \ . 
\]

To see that $W_{v}$ is bounded, note that 
\[
W_{v}\alpha_{\ell_1}^{+\mathbf p}=\sum_{\mathbf q,\chi}\frac{n_{\mathbf p,\mathbf q,-v,\chi}}{\abs{G_{\mathbf q}}}\hbar^{-\chi}\alpha_{\ell_1}^{+\mathbf q} \ ,
\]
where $n_{\mathbf p,\mathbf q,-v,\chi}\in\mathbb Q$ counts the number of possibly disconnected holomorphic curves in $Y\times\mathbb{CP}^1$ with Euler characteristic $\chi$ and contact data $\mathbf p$ with $Y\times 0$, contact data $\mathbf q$ at $Y\times \infty$, and one extra contact point determined by $-v$, suitably constrained. Given a positive symplectic form $\omega$, the $\omega$-energy of such curves is entirely determined by $\mathbf p$ and $v$, so Gromov--compactness implies that this is a finite sum once $\chi$ is fixed. Moreover, $\chi$ is bounded above by $1$, because, in the above count, the only possible connected stable holomorphic curve with positive Euler characteristic is a sphere with a unique contact point given by $v$. Accordingly, $W_v$ is a bounded $\N$-module automorphism.

In the following, we define an operator
\[
\W_{v,\ell}:\mathcal H^+_\ell \longrightarrow \mathcal H^+_\ell
\]
for each $v\in \mathfrak t^2_{\mathbb Z} \setminus \{0\}$. Depending on the arrangement of $v$ relative to the leg $\ell$, we change the sign of $W_v$ to define $\W_{v,\ell}$. We do this to ensure that the operators $\W_{v,\ell}$ satisfy the commutation relation from \cref{Wcommutation}, and so that framing changes are compatible with the action of these operators, especially in the case depicted in  \cref{f14}.

Associated to the leg $\ell$, we have have the framing $w_\ell \in \mathfrak t^2_{\mathbb Z}$, and also a canonical normal vector $n_\ell \in \mathfrak t^2_{\mathbb Z}$ such that the leg travels in the direction annihilated by $n_\ell$, and $(w_\ell, n_\ell)$ forms an oriented basis for $\mathfrak t^2_{\mathbb Z}$. We can write $v$ in this basis, using the notation
\[
v=\frac{v\wedge n_\ell}{w_\ell \wedge n_\ell}w_\ell + \frac{w_\ell \wedge v}{w_\ell \wedge n_\ell}n_\ell := \lrb{v\wedge n_\ell}w_\ell + \lrb{w_\ell \wedge v} n_\ell \ .
\]
Then define $\W_{v,\ell}$ as follows.
\begin{equation} \label{wdef}
\W_{v,\ell}:=\begin{cases}W_{v}, & \text{ if }{w_\ell \wedge v}>0 \text{ or } v= -kw_\ell \text { with }k>0
\\ -W_{v}, & \text{ if }{w_\ell \wedge v}<0 \text{ or } v= kw_\ell \text { with }k>0\end{cases}
\end{equation}

\begin{remark}\label{W identification} Given two legs $\ell_1$ and $\ell_2$ and an integral matrix $A$ such that
\[Aw_{\ell _1}=w_{\ell_2}\ \ \ \ \text{ and } An_{\ell_1}=n_{\ell_2}\]
There is a natural isomorphism 
\[
I_A: \mathcal H^+_{\ell_1}\longrightarrow \mathcal H^+_{\ell_2}
\] 
such that
\[
I_A(\alpha^{+\bf p}_{\ell_1})=\alpha^{+\bf p}_{\ell_2} \ .
\]
This isomorphism respects all the structure of $\mathcal H^+_{\ell}$, including the action of $\W_{v,\ell}$ 
\[
I_A\circ \W_{v,\ell_1}=\W_{Av, \ell_2} \circ I_A \ .
\]
Note that in the case that $\ell_1$ and $\ell_2$ are the two legs from $Y\times\mathbb CP^1$, this natural identification is different from the identification used above, which instead identifies $\mathcal H^+_{\ell_1}$ with $\mathcal H^-_{\ell_2}$ because the two copies of $Y$ over $0$ and $\infty$ have opposite holomorphic volume forms.
\end{remark}

\begin{lemma} \label{zc}
Suppose that $v_1$ and $v_2$ satisfy
\[
{w_{\ell_1}\wedge v_1}>0, \qquad {w_{\ell_1}\wedge v_2}>0, \qquad v_1 \wedge v_{2} \geq 0.
\]
Then
\[
\op{Z_{Y\times\mathbb{CP}^1,h_{-v_1}h_{-v_2}}}{\ell_2}\rvert_{\mathcal H^{+}_{\ell_1}}=\Prop_x\circ t^{c_1}\W_{v_1,\ell_1}\circ t^{c_2}\W_{v_2,\ell_1},
\]
where $x$, $c_1$ and $c_2$ are given by the formulas
\begin{align*}
x &= \int_{\mathbb{CP}^1}\omega \\
c_i &= \bigg( \int_{E^-_{\ell_1}}\omega \bigg) {w_{\ell_1}\wedge v_i}+ \bigg( \int_{\mathbb{CP}^1} \omega \bigg) \max\lrb{0,{v_i\wedge n_{\ell_1}}}.
\end{align*}
So, for $\alpha$ in $\mathcal H^+_{\ell_1}$ and $\beta\in\mathcal H^+_{\ell_2}$,
\[
\ip{ Z_{Y\times\mathbb{CP}^1,h_{-v_1}h_{-v_2}}}{ \beta \alpha }=t^{c_1+c_2} \ip{ Z_{Y\times\mathbb{CP}^1}}{ \beta \W_{v_1,\ell_1} \W_{v_2,\ell_1} \alpha} \ .
\]
\end{lemma}

\begin{proof}
The exponent of $t$ is straightforward to verify, as it is determined topologically. The remainder of the proof involves calculating Gromov--Witten invariants using the tropical gluing formula.

For $\alpha\in\mathcal H^+_{\ell_1}$, note that
\[
\op{Z_{Y\times\mathbb{CP}^1,h_{-v_1}h_{-v_2}}}{\ell_2}(\alpha)=\op{\exp \eta}{\ell_2}(h_{v_1}h_{v_2}\alpha) \ ,
\]
because all the holomorphic curves contributing to this count are contained in $f^{-1}(0)$. The tropical gluing formula of \cref{eq:tgf} then allows us to compute this Gromov--Witten invariant as a sum over the contributions of tropical curves.

We are free to choose representatives of $h_{-v_i}$ to constrain these tropical curves to appear as in \cref{zcalc}. In particular, the tropical part of $\ex X_{-v_i}$ parametrises the space of infinite rays in the tropical part of $\expl (Y\times\mathbb{CP}^1,D)$ travelling in the direction $-v_i$, with two rays identified if they eventually coincide. We can choose refined forms representing $h_{-v_i}\in \rh(\ex X_{-v_i})$ so that the tropical part of the support of $h_{-v_i}$ is a 1-dimensional linear subspace in this space of rays, intersecting the space of rays emanating from the singular locus at a single point, which we can choose where we like. Our conditions on the vectors $v_i$ ensure that both these rays are on one side of the singular locus, and that we can choose constraints so that these rays do not intersect, and the second ray intersects the singular locus closer to $\ell_1$ than the first ray.

\begin{figure}[ht!]
\centering
\begin{tikzpicture}
\begin{scope}
\fill [box style] (0,0) rectangle (4,2);
\draw[wall style] (0,1) -- (4,1);
\draw[curve style] (2,0) -- (2,1);
\draw[curve style] (0.5,0) -- (1.5,1);
\draw[wallcurve style] (0,1) -- (4,1);
\node[below, xshift=7.0710] at (0.5,0) {$v_1$};
\node[below, xshift=5] at (2,0) {$v_2$};
\node[above left, xshift=2, yshift=5] at (4,1) {$n_{\ell_1}$};
\node[below right] at (4,1) {$w_{\ell_1}$};
\node[above left] at (0,1) {$w_{\ell_2}$};
\node[below right, xshift=-2, yshift=-5] at (0,1) {$n_{\ell_2}$};
\draw[->, thick, xshift=5] (4,1) -- (4.5,1);
\draw[->, thick, yshift=5] (4,1) -- (4,1.5);
\draw[->, thick, xshift=-5] (0,1) -- (-0.5,1);
\draw[->, thick, yshift=-5] (0,1) -- (0,0.5);
\draw[->, thick, xshift=7.0710] (0.5,0) -- (0.5+0.3535,0.3535);
\draw[->, thick, xshift=5] (2,0) -- (2,0.5);
\end{scope}
\end{tikzpicture}
\caption{Tropical curves contributing to $Z_{Y\times\mathbb{CP}^1,h_{-v_1}h_{-v_2}}$. The wavy line along the singular line indicates some unspecified combination of edges travelling along the singular line.}
\label{zcalc}
\end{figure}
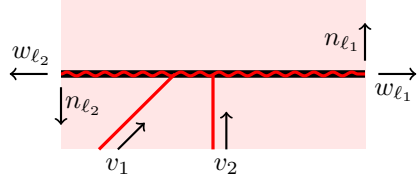

Using such a representative for $h_{-v_i}$, the tropical curves contributing to our calculation must have vertices only\footnote{Actually, there are some more vertices, which we are ignoring because they are irrelevant: The tropical part of $\expl(Y\times\mathbb{CP}^1)$ is a subdivision of the space pictured in \cref{zcalc}. Tropical curves are forced to have vertices where edges intersect this suppressed subdivison.} where these rays intersect the singular locus, and edges only along these rays or along the singular locus. This is because other tropical curves satisfying the required balancing condition are either not rigid after being constrained, and hence do not contribute to the count, or contain a vertex connected by two different paths of edges to the singular locus, or contain a vertex attached to the singular locus, and also attached to an edge constrained to $f^{-1}(0)$. Such curves also do not contribute, as can be seen from the gluing formula combined with the observation that the self intersection of $f^{-1}(0)$ vanishes. See \cref{fig:noncontributors} for examples of tropical curves that do not contribute to our calculation.

\begin{figure}[ht!]
\centering
\begin{tikzpicture}
\begin{scope}
\fill [box style] (0,0) rectangle (4,2);
\draw[wall style] (0,1) -- (4,1);
\draw[curve style] (2,0) -- (2,0.5);
\draw[curve style] (2,0.5) -- (1.5,1) -- (0.5,0);
\draw[curve style] (2,0.5) -- (2.5,1);
\draw[wallcurve style] (0,1) -- (4,1);
\end{scope}
\begin{scope}[xshift=5cm]
\fill [box style] (0,0) rectangle (4,2);
\draw[wall style] (0,1) -- (4,1);
\draw[curve style] (0.25,0) -- (1,0.75) -- (0,0.75);
\draw[curve style] (1,0.75) -- (1.5,1);
\draw[curve style] (2,0) -- (2,1);
\draw[wallcurve style] (0,1) -- (4,1);
\end{scope}
\end{tikzpicture}
\caption{Some tropical curves that are the tropical part of holomorphic curves, but do not contribute to $Z_{Y\times\mathbb{CP}^1,h_{-v_1}h_{-v_2}}$.}
\label{fig:noncontributors}
\end{figure}
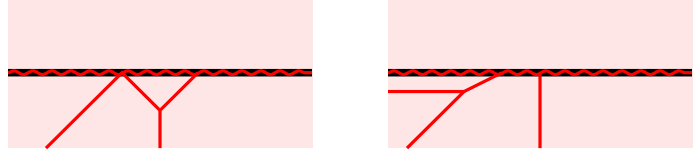

Once we have that all contributing tropical curves are in the form of \cref{zcalc}, the tropical gluing formula of \cref{eq:tgf} simplifies to a formula gluing together two relative invariants: one for each of the vertices in \cref{zcalc}. Let $p_1$ and $p_2$ be the corresponding points in \cref{zcalc} and choose our constraints such that $p_1$ and $p_2$ are both contained in the ray corresponding to the leg $\ell_1$.

To simplify the discussion, we can refine\footnote{See~\cite[Section 9]{iec}. Such refinements do not affect Gromov--Witten invariants, simply resulting in a refinement of the moduli space of holomorphic curves. } $\expl (Y\times\mathbb{CP}^1)$ as pictured in \cref{f21} to an exploded manifold $\ex X$ so that contributing tropical curves are forced to have vertices at these points $p_i$; otherwise, the discussion is complicated by contributing tropical curves with edges passing through these points. After this refinement our contributing tropical curves are forced to have vertices along the singular line at the points $0$, $p_1$, and $p_2$, and any point where one of our rays in the direction $-v_i$ passes through a lower dimensional stratum, like on the left in \cref{f21}. The tropical gluing formula involves Gromov--Witten invariants for the exploded manifolds $\ex X\tc 0:=\ex X_0$, $\ex X\tc{p_1}:=\ex X_1$ and $\ex X\tc {p_2}:=\ex X_2$, each of which is isomorphic to $\expl(Y\times \mathbb{CP}^1)$. The Gromov--Witten invariant from $\ex X_0$ is encapsulated in $\Prop_x$, as it counts possibly disconnected curves with contact data only along the singular locus, whereas the Gromov--Witten invariant from $\ex X_{i}$ is encapsulated in $\W_{v_i,\ell_1}$, (times $t$ to some exponent) as it counts possibly disconnected curves with a single contact point of type $v_i$ and all other contact points along the singular locus.

\begin{figure}[ht!]
\centering
\begin{tikzpicture}
\begin{scope}
\fill [box style] (0,0) rectangle (4,2);
\draw[wall style] (0,1) -- (4,1);
\draw[curve style] (1.5,1) -- (0.5,0);
\draw[curve style] (2,0) -- (2,1);
\draw[wallcurve style] (0,1) -- (4,1);
\draw[lightgray, thick] (1,0) -- (1,2);
\draw[lightgray, thick] (1.5,0) -- (1.5,2);
\draw[lightgray, thick] (2,0) -- (2,2);
\end{scope}
\end{tikzpicture}
\caption{The subdivision of the tropical part of $\expl(Y\times\mathbb{CP}^1)$ used to refine $\expl(Y\times\mathbb{CP}^1)$. Each of the three grey lines should be imagined as a plane extending in the unpictured direction. The line on the left passes through $0\subset\totb{\expl(Y\times\mathbb{CP}^1)}$, and was already present, but not pictured before.}
\label{f21}
\end{figure}
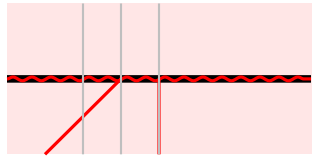

The tropical gluing formula of \cref{eq:tgf} is stated for the contribution of a single, connected tropical curve, however we can apply it for each connected component of a tropical curve $\gamma$ with multiple connected components to derive an analogous formula for disconnected tropical curves. It requires a combinatorial factor $k_\gamma$, the product of the multiplicities of the internal edges of $\gamma$, because it uses evaluation spaces analogous to $\ex X_{v}$, instead of the quotient stack $\ex X_{v}/G_v:=\ex X_{v}/\mathbb Z_{\abs v}$, and it requires dividing by $\abs{\Aut \gamma}$, because the relative invariants use labelled contact data so that we can specify which edges get attached, and the evaluation spaces for multiple edges are the products of the evaluation spaces for individual edges, instead of the quotient of this by symmetries. Summing over all contributing tropical curves, and allowing disconnected curves means that we glue edges in all possible ways, so allows us to use evaluation spaces analogous to our $\ex X^{\mathbf p}/G_{\mathbf p}$, without these extra combinatorial factors.

Let us apply the tropical gluing formula of \cref{eq:tgf} to calculate
\begin{equation} \label{cal1}
\ip{ \beta} {\op{Z_{Y\times\mathbb{CP}^1,h_{-v_1}h_{-v_2}}}{\ell_2}\alpha} =\lrb{\an_{\beta\alpha h_{-v_1}h_{-v_2}}\exp \eta}_{0}
\end{equation}
in the notationally simpler case that the only vertices of our contributing tropical curves $\gamma$ are at the points $p_0$, $p_1$ and $p_2$. In this calculation $\eta$ is a form on $\coprod_{\mathbf p}\ex X^{\mathbf p}/G_{\mathbf p}$, whereas in~\cite{gfgw}, the notation $\eta$ indicates a form on $\coprod_{\mathbf p}\ex X^{\mathbf p}$, which on each component is the pullback $\eta_{\mathbf p}$ of our $\eta$ divided by $\prod_{v}\abs v^{\mathbf p(v)}=\abs{G_{\mathbf p}}/\abs{\Aut \mathbf p}$. It follows that, using our present notation, the combinatorial factor in~\cite[Equation~(1)]{gfgw} becomes $\frac 1{k_\gamma\abs{\Aut \gamma}}$ instead of $\frac {k_\gamma}{\abs{\Aut \gamma}}$ where now $k_\gamma$ means the product of multiplicities of all edges of $\gamma$ instead of just the internal edges, and $\Aut\gamma$ means the full automorphism group of the tropical curve $\gamma$ instead of the group of automorphisms fixing the ends of $\gamma$. Our calculation of \cref{cal1} reduces to
\[
\sum_{\gamma}\frac{1}{k_{\gamma}\abs{\Aut \gamma}}\int{\beta h_{-v_1} h_{-v_2}}\iota^{[\gamma]}_!\Delta^*(\exp \eta)^{[\gamma_{p_0}]}(\exp \eta)^{[\gamma_{p_1}]}(\exp \eta)^{[\gamma_{p_2}]} \ ,
\]
where the sum is over contributing, possibly disconnected tropical curves $\gamma$. In the above, $\gamma_{p_i}$ indicates the tropical curve in $\totb{\ex X_{i}}$ with one connected component for each vertex of $\gamma$ at $p_i$, such that each component has a single vertex at $p_i$ and edges corresponding to the edges of $\gamma$ leaving this vertex. The term $(\exp \eta)^{[\gamma_{p_i}]}$ indicates the Gromov--Witten invariant counting curves in $\ex X_{i}$ with connected components labelled by the components of $\gamma_{p_i}$, and contact data determined by the edges of $\gamma_{p_i}$. Each such tropical curve $\gamma$ determines contact data $\mathbf p_i$, for curves in $\ex X_{i}$ but different tropical curves can determine the same contact data. When we sum over all tropical curves with the same contact data, we obtain the following expression for \cref{cal1}.
\begin{equation} \label{cal2}
\sum_{\mathbf p_0,\mathbf p_1,\mathbf p_2}\frac{1}{\abs {G_{\mathbf p_0}}\abs{G_{\mathbf p_1}}\abs{G_{\mathbf p_2}}}\int{\beta h_{-v_1} h_{-v_2}}\iota_!\Delta^*(\exp \eta_0)_{\mathbf p_0}(\exp \eta_1)_{\mathbf p_1}(\exp \eta_2)_{\mathbf p_2} \ .
\end{equation}
Here, $\eta_i$ now indicates the Gromov--Witten invariant from $\ex X_{i}$ and $(\exp \eta_i)_{\mathbf p_i}$ indicates the pullback of $\exp \eta_i$ to the evaluation space $(\ex X_{i})^{\mathbf p_i}$. The contact data in $\mathbf p_0$ and $\mathbf p_1$ corresponding to edges between $p_0$ and $p_1$ is matched. Denote this contact data by $\mathbf p^{1^-}$ and $\mathbf p^{1^+}$, respectively; so as lists of numbers, $\mathbf p^{1^-}=\mathbf p^{1^+}$. There is a canonical identification of $(\ex X_{0})^{\mathbf p^{1^-}}=(\ex X_{1})^{\mathbf p^{1^+}}$. Similarly, denoting the matched contact data corresponding to edges between $p_1$ and $p_2$ by $\mathbf p^{2^\pm}$ there is a canonical identification of $(\ex X_{1})^{\mathbf p^{2^-}}=(\ex X_{2})^{\mathbf p^{2^+}}$. In \cref{cal2}, the map $\Delta$ is the corresponding diagonal inclusion 
\[
\Delta: (\ex X_{1})^{\mathbf p^{1^+}}(\ex X_{2})^{\mathbf p^{2^+}}\longrightarrow (\ex X_{1})^{\mathbf p_1}(\ex X_{2})^{\mathbf p_2}(\ex X_{2})^{\mathbf p_2}
\]
and $\iota$ indicates the projection forgetting all factors corresponding to internal edges
\[
\iota: (\ex X_{1})^{\mathbf p_1}(\ex X_{2})^{\mathbf p_2}(\ex X_{2})^{\mathbf p_2}\longrightarrow (\ex X_{1})^{\mathbf p_1-\mathbf p^{1^-}}(\ex X_{2})^{\mathbf p_2-\mathbf p^{1^+}-\mathbf p^{2^-}}(\ex X_{2})^{\mathbf p_2-\mathbf p^{2^+}} \ .
\]
With this understood, we can rewrite \cref{cal2} as 
\begin{equation} \label{cal3}
\sum_{\mathbf p^{1^+},\mathbf p^{2^+}}\frac{1}{\abs {G_{\mathbf p^{1^+}}}\abs{ G_{\mathbf p^{2^+}}}}\int_{\ex X_1^{\mathbf p^{1^+}}\ex X_2^{\mathbf p^{2^+}}}\lrb{\an_{\beta}(\exp \eta_0)}_{\mathbf p^{1^-}}\wedge \lrb{\an_{h_{-v_1}}(\exp \eta_1)}_{\mathbf p^{2^-}+\mathbf p^{1^+}} \wedge\lrb{\an_{h_{-v_2}\alpha}(\exp \eta_2)}_{\mathbf p^{2^+}} .
\end{equation}
Note that $\lrb{\an_{\beta}(\exp \eta_0)}_{\mathbf p^{1^-}}$ is contained within $\mathcal H^+_{1^-}\otimes\mathcal H^{-}_{1^-}$, because it counts holomorphic curves which, once constrained by $\beta$, are contained in $f^{-1}(0)$. Similarly, $\lrb{\an_{h_{-v_2}\alpha}(\exp \eta_2)}_{\mathbf p^{2^+}}\in \mathcal H^+_{2^+}\otimes\mathcal H^{-}_{2^+}$. Moreover, given any $\theta\in \mathcal H^+_{2^-}\otimes\mathcal H^{-}_{2^-}$, we have $\lrb{\an_{\theta} \an_{h_{-v_1}}(\exp \eta_1)}_{\mathbf p^{1^+}}\in \mathcal H^+_{1^+}\otimes\mathcal H^{-}_{1^+}$, and the projection of $\an_{\theta} \an_{h_{-v_1}}(\exp \eta_1)$ to $\mathcal H^+_{1^+}$ depends only on the projection of $\theta$ to $\mathcal H^+_{2^-}$. Accordingly, we can rewrite \cref{cal3} as
\[
\int \lrb{\an_{\beta}(\exp \eta_0)}_{1^+}\wedge \lrb{\an_{h_{-v_1}}(\exp \eta_1)}_{1^+,2^-}\wedge \lrb{\an_{h_{-v_2}\alpha}(\exp \eta_2)}_{2^+}
\]
and then rewrite this as a composition of operators
\[
\ip{\beta}{ \Prop_{x}\circ t^{c_1}\W_{v_1,\ell_1}\circ t^{c_2}\W_{v_2,\ell_1}(\alpha)} \ ,
\]
where the constant $c_i$ is the $\omega$-energy of the corresponding curve at the vertex $p_i$, which is $\lrb{\int_{E^-_{\ell_1}}\omega} {w_{\ell_1}\wedge v_i}$.

This above suffices to prove \cref{zc} in the case that our tropical curves do not have vertices away from $p_0$, $p_1$, and $p_2$. In the case of these extra vertices, using the tropical gluing formula in conjunction with \cref{toric2leg}, we obtain the same result
\[
\op{Z_{Y\times\mathbb{CP}^1,h_{-v_1}h_{-v_2}}}{\ell_2}\rvert_{\mathcal H^{+}_{\ell_1}}=\Prop_x\circ t^{c_1}\W_{v_1,\ell_1}\circ t^{c_2}\W_{v_2,\ell_1} \ ,
\]
except now $c_i$ is $\lrb{\int_{E^-_{\ell_1}}\omega} {w_{\ell_1}\wedge v_i}$ plus the $\omega$-energy of the curve at the extra vertices on the $i$th ray, which is concentrated where this ray crosses the codimension 1 stratum passing through $0$; this only happens when $v_i\wedge n_{\ell_1}>0$, and the $\omega$-energy here is $(v_i\wedge n_{\ell_1})x$. So 
\[
c_i= \lrb{\int_{E^-_{\ell_1}}\omega} {w_{\ell_1}\wedge v_i}+ \lrb{\int_{\mathbb{CP}^1}\omega}\max\lrb{0,{v_i\wedge n_{\ell_1}}} \ . \qedhere
\]
\end{proof}

A remarkable fact about these operators $\W_{v,\ell}$ is that they obey the following commutation relations.

\begin{thm} \label{Wcommutation}
We have
\[
\W_{v,\ell} \W_{w,\ell}-\W_{w,\ell}\W_{v,\ell} = [{v\wedge w}]_q \, \W_{v+w,\ell}+ ( n_\ell\wedge v) \delta_{v+w} \ ,
\]
or equivalently,
\[
\left[\W_{aw_\ell+bn_\ell,\ell},\W_{cw_\ell+dn_\ell,\ell}\right]=[ad-bc]_q\W_{(a+c)w_\ell+(b+d)n_\ell,\ell}-a\delta_{a+c}\delta_{b+d} \ .
\]
\end{thm}

\begin{proof}
This commutation relation follows from computing $Z_{Y\times\mathbb{CP}^1,h_{-v}h_{-w}}$ using two different forms representing $h_{-v}$ and $h_{-w}$.

Consider the case that $v+w\neq 0$, in which the second term on the right side of the commutation relation does not contribute. This assumption ensures that all holomorphic curves contributing to the calculation of $Z_{Y\times\mathbb{CP}^1,h_{-v}h_{-w}}$ are contained in $f^{-1}(0)$. As in the proof of \cref{zc}, we are free to choose forms representing $h_{-v}$ and $h_{-w}$ so that the tropical curves contributing to the calculation of $Z_{Y\times \mathbb{CP}^1,h_{-v}h_{-w}}$ consist of tropical curves with an end in the directions $-v$ and $-w$ constrained to a chosen ray, and all other ends travelling out in the direction of the singular line. Moreover, each contributing tropical curve must be rigid once these ends are constrained, and each component of these tropical curves minus the singular line must consist of a tree with at most one edge attached to a vertex on the singular line.

First, consider the case that $v\wedge w$, $w_\ell \wedge v$ and $w_\ell \wedge w$ are positive. In this case, \cref{zc} computes $Z_{Y\times \mathbb{CP}^1,h_{-v}h_{-w}}$, and if we choose constraints as in the proof of \cref{zc}, contributing tropical curves are as pictured on the left in \cref{fig:comproof1}. However, if we choose our constraints so that the corresponding rays from the singular line in the direction $-v$ and $-w$ cross, we get contributing tropical curves of the two types depicted on the right. The first of these corresponds to applying $\W_{w,\ell}\W_{v,\ell}$, whereas \cref{3dvertex} and the tropical gluing formula imply that the second corresponds to applying $[v\wedge w]_q \, \W_{v+w,\ell}$.

\begin{figure}[ht!]
\centering
\begin{tikzpicture}
\begin{scope}
\fill [box style] (0,0) rectangle (4,2);
\draw[wall style] (0,1) -- (4,1);
\draw[curve style] (1.5,1) -- (0.5,0);
\draw[curve style] (2,0) -- (2,1);
\draw[wallcurve style] (0,1) -- (4,1);
\end{scope}
\node at (4.5,1) {$=$};
\begin{scope}[xshift=5cm]
\fill [box style] (0,0) rectangle (4,2);
\draw[wall style] (0,1) -- (4,1);
\draw[curve style] (2.5,1) -- (1.5,0);
\draw[curve style] (2,0) -- (2,1);
\draw[wallcurve style] (0,1) -- (4,1);
\end{scope}
\node at (9.5,1) {$+$};
\begin{scope}[xshift=10cm]
\fill [box style] (0,0) rectangle (4,2);
\draw[wall style] (0,1) -- (4,1);
\draw[curve style] (2,0.5) -- (1.5,0);
\draw[curve style] (2,0) -- (2,0.5);
\draw[curve style] (2.25,1) -- (2,0.5);
\draw[wallcurve style] (0,1) -- (4,1);
\end{scope}
\end{tikzpicture}
\caption{In the case $v\wedge w$, $w_\ell \wedge v$ and $w_\ell \wedge w$ are positive, this diagram represents two ways of calculating $Z_{Y\times\mathbb{CP}^1,h_{-v}h_{-w}}$, thus showing that $\W_{v,\ell} \W_{w,\ell}=\W_{w,\ell} \W_{v,\ell} + \left[{v\wedge w}\right]_q \, \W_{v+w,\ell}$.}
\label{fig:comproof1}
\end{figure}
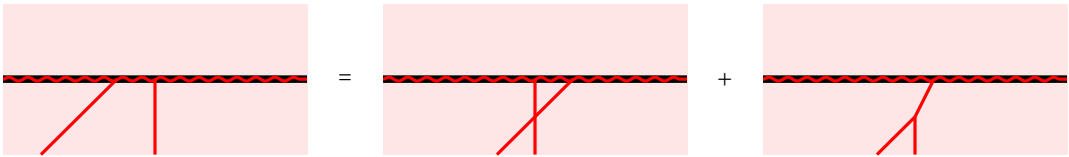

Similarly, calculating $Z_{Y\times \mathbb{CP}^1, h_{-v}h_{-w}}$ in two different ways gives our commutation relation in all other cases when $v+w\neq 0$, as illustrated in \cref{fig:comproof2,fig:comproof3,fig:comproof4,fig:comproof5,fig:comproof6}.

\begin{figure}[ht!]
\centering
\begin{tikzpicture}
\begin{scope}
\fill [box style] (0,0) rectangle (4,2);
\draw[wall style] (0,1) -- (4,1);
\draw[curve style] (1.5,1) -- (0.5,2);
\draw[curve style] (3.5,2) -- (2.5,1);
\draw[wallcurve style] (0,1) -- (4,1);
\end{scope}
\node at (4.5,1) {$=$};
\begin{scope}[xshift=5cm]
\fill [box style] (0,0) rectangle (4,2);
\draw[wall style] (0,1) -- (4,1);
\draw[curve style] (1.5,1) -- (2.5,2);
\draw[curve style] (1.5,2) -- (2.5,1);
\draw[wallcurve style] (0,1) -- (4,1);
\end{scope}
\node at (9.5,1) {$+$};
\begin{scope}[xshift=10cm]
\fill [box style] (0,0) rectangle (4,2);
\draw[wall style] (0,1) -- (4,1);
\draw[curve style] (2,1) -- (2,1.5);
\draw[curve style] (1.5,2) -- (2,1.5);
\draw[curve style] (2.5,2) -- (2,1.5);
\draw[wallcurve style] (0,1) -- (4,1);
\end{scope}
\end{tikzpicture}
\caption{In the case $v\wedge w$, $w_\ell \wedge v$ and $w_\ell \wedge w$ are negative, we obtain $(-\W_{v,\ell})(-\W_{w,\ell})=(-\W_{w,\ell})(-\W_{v,\ell})+\left[-{v\wedge w}\right]_q\, (-\W_{v+w,\ell})$.}
\label{fig:comproof2}
\end{figure}
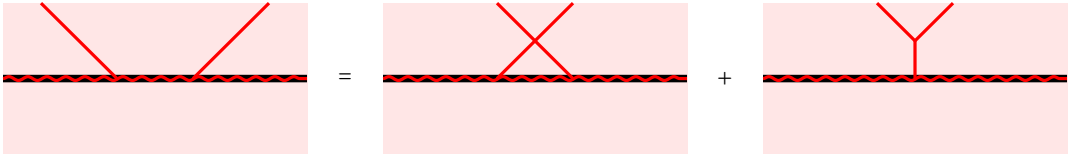

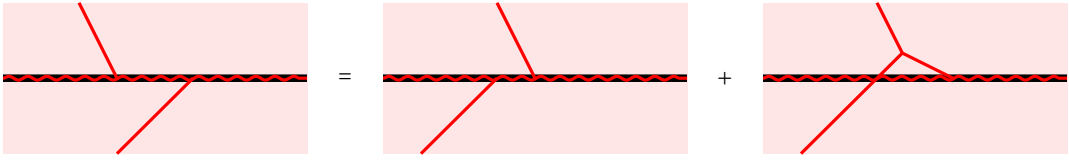
\begin{figure}[ht!]
\centering
\begin{tikzpicture}
\begin{scope}
\fill [box style] (0,0) rectangle (4,2);
\draw[wall style] (0,1) -- (4,1);
\draw[curve style] (1.5,1) -- (1,2);
\draw[curve style] (1.5,0) -- (2.5,1);
\draw[wallcurve style] (0,1) -- (4,1);
\end{scope}
\node at (4.5,1) {$=$};
\begin{scope}[xshift=5cm]
\fill [box style] (0,0) rectangle (4,2);
\draw[wall style] (0,1) -- (4,1);
\draw[curve style] (1.5,1) -- (0.5,0);
\draw[curve style] (1.5,2) -- (2,1);
\draw[wallcurve style] (0,1) -- (4,1);
\end{scope}
\node at (9.5,1) {$+$};
\begin{scope}[xshift=10cm]
\fill [box style] (0,0) rectangle (4,2);
\draw[wall style] (0,1) -- (4,1);
\draw[curve style] (1.8333,1.3333) -- (0.5,0);
\draw[curve style] (1.5,2) -- (1.8333,1.3333) -- (2.5,1);
\draw[wallcurve style] (0,1) -- (4,1);
\end{scope}
\end{tikzpicture}
\caption{In the case $w_\ell \wedge v$ and $w_\ell \wedge (v+w)$ are negative, but $v\wedge w$ and $w_\ell \wedge w$ are positive, we obtain $-\W_{v,\ell}\W_{w,\ell}=\W_{w,\ell}(-\W_{v,\ell})+\left[ {v\wedge w}\right]_q \, (-\W_{v+w,\ell})$.}
\label{fig:comproof3}
\end{figure}

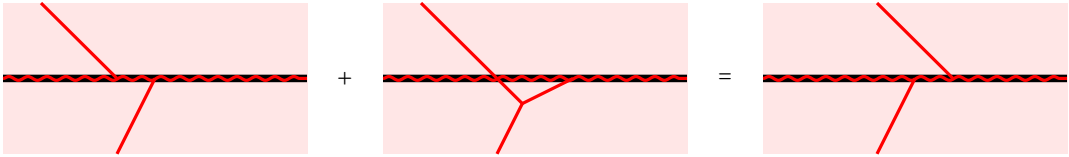
\begin{figure}[ht!]
\centering
\begin{tikzpicture}
\begin{scope}
\fill [box style] (0,0) rectangle (4,2);
\draw[wall style] (0,1) -- (4,1);
\draw[curve style] (1.5,1) -- (0.5,2);
\draw[curve style] (1.5,0) -- (2,1);
\draw[wallcurve style] (0,1) -- (4,1);
\end{scope}
\node at (4.5,1) {$+$};
\begin{scope}[xshift=5cm]
\fill [box style] (0,0) rectangle (4,2);
\draw[wall style] (0,1) -- (4,1);
\draw[curve style] (1.8333,0.6666) -- (0.5,2);
\draw[curve style] (1.5,0) -- (1.8333,0.6666) -- (2.5,1);
\draw[wallcurve style] (0,1) -- (4,1);
\end{scope}
\node at (9.5,1) {$=$};
\begin{scope}[xshift=10cm]
\fill [box style] (0,0) rectangle (4,2);
\draw[wall style] (0,1) -- (4,1);
\draw[curve style] (2.5,1) -- (1.5,2);
\draw[curve style] (1.5,0) -- (2,1);
\draw[wallcurve style] (0,1) -- (4,1);
\end{scope}
\end{tikzpicture}
\caption{In the case $w_\ell \wedge v$ is negative, but $v\wedge w$, $w_\ell \wedge (v+w)$ and $w_\ell \wedge w$ are positive, we obtain $-\W_{v,\ell}\W_{w,\ell}+\left[ {v\wedge w}\right]_q \, (\W_{v+w,\ell})=\W_{w,\ell}(-\W_{v,\ell})$.}
\label{fig:comproof4}
\end{figure}

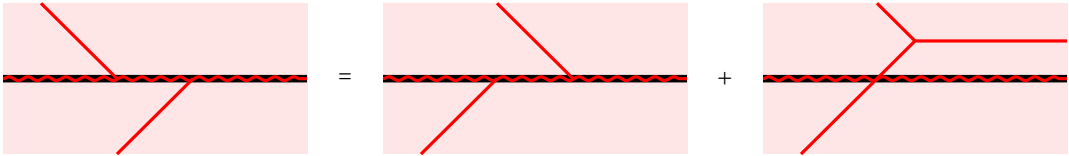
\begin{figure}[ht!]
\centering
\begin{tikzpicture}
\begin{scope}
\fill [box style] (0,0) rectangle (4,2);
\draw[wall style] (0,1) -- (4,1);
\draw[curve style] (1.5,1) -- (0.5,2);
\draw[curve style] (1.5,0) -- (2.5,1);
\draw[wallcurve style] (0,1) -- (4,1);
\end{scope}
\node at (4.5,1) {$=$};
\begin{scope}[xshift=5cm]
\fill [box style] (0,0) rectangle (4,2);
\draw[wall style] (0,1) -- (4,1);
\draw[curve style] (1.5,1) -- (0.5,0);
\draw[curve style] (1.5,2) -- (2.5,1);
\draw[wallcurve style] (0,1) -- (4,1);
\end{scope}
\node at (9.5,1) {$+$};
\begin{scope}[xshift=10cm]
\fill [box style] (0,0) rectangle (4,2);
\draw[wall style] (0,1) -- (4,1);
\draw[curve style] (2,1.5) -- (0.5,0);
\draw[curve style] (1.5,2) -- (2,1.5);
\draw[curve style] (2,1.5) -- (4,1.5);
\draw[wallcurve style] (0,1) -- (4,1);
\end{scope}
\end{tikzpicture}
\caption{In the case $v+w$ is a positive multiple of $w_\ell$, $w_\ell \wedge w$ is positive, but $v \wedge w$ and $w_\ell \wedge v$ are negative, we obtain $-\W_{v,\ell}\W_{w,\ell}=\W_{w,\ell}(-\W_{v,\ell})+\left[ {v\wedge w}\right]_q \, (-\W_{v+w,\ell})$.}
\label{fig:comproof5}
\end{figure}

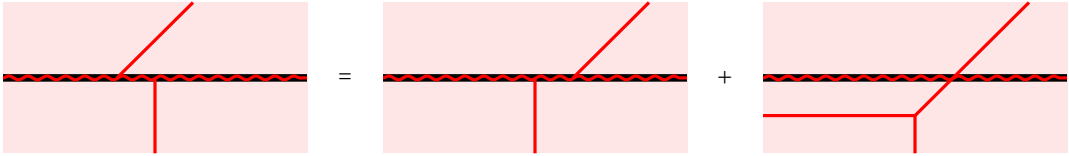
\begin{figure}[ht!]
\centering
\begin{tikzpicture}
\begin{scope}
\fill [box style] (0,0) rectangle (4,2);
\draw[wall style] (0,1) -- (4,1);
\draw[curve style] (1.5,1) -- (2.5,2);
\draw[curve style] (2,0) -- (2,1);
\draw[wallcurve style] (0,1) -- (4,1);
\end{scope}
\node at (4.5,1) {$=$};
\begin{scope}[xshift=5cm]
\fill [box style] (0,0) rectangle (4,2);
\draw[wall style] (0,1) -- (4,1);
\draw[curve style] (2.5,1) -- (3.5,2);
\draw[curve style] (2,0) -- (2,1);
\draw[wallcurve style] (0,1) -- (4,1);
\end{scope}
\node at (9.5,1) {$+$};
\begin{scope}[xshift=10cm]
\fill [box style] (0,0) rectangle (4,2);
\draw[wall style] (0,1) -- (4,1);
\draw[curve style] (2,0.5) -- (3.5,2);
\draw[curve style] (2,0) -- (2,0.5);
\draw[curve style] (0,0.5) -- (2,0.5);
\draw[wallcurve style] (0,1) -- (4,1);
\end{scope}
\end{tikzpicture}
\caption{In the case $v+w$ is a negative multiple of $w_\ell$, $w_\ell \wedge v$ is negative, but $v\wedge w$ and $w_\ell \wedge w$ positive, we obtain $-\W_{v,\ell}\W_{w,\ell}=\W_{w,\ell}(-\W_{v,\ell})+\left[- {v\wedge w}\right]_q \, \W_{v+w,\ell}$.}
\label{fig:comproof6}
\end{figure}

\newpage

The only important cases not drawn are when $v$ is proportional to $w$. When $w=kv$ with $k\neq -1$, and $v$ not proportional to $w_\ell$, all relevant holomorphic curves for calculating $Z_{Y\times\mathbb{CP}^1,h_{-v}h_{-w}}$ are contained in $f^{-1}(0)$, and $\W_{v,\ell}$ and $\W_{w,\ell}$ commute, as in \cref{fig:comproof7}. When $v=-w$, there are extra holomorphic curves which contribute, outside of $f^{-1}(0)$. These contributions can also be calculated tropically, but we instead calculate them using the Jacobi identity and the case when $v$ is a multiple of $w_\ell$.

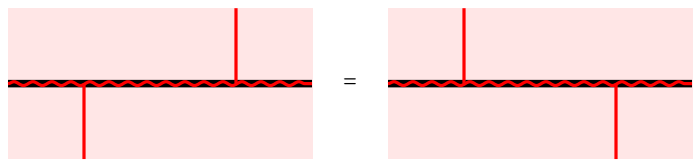
\begin{figure}[ht!]
\centering
\begin{tikzpicture}
\begin{scope}
\fill [box style] (0,0) rectangle (4,2);
\draw[wall style] (0,1) -- (4,1);
\draw[curve style] (1,0) -- (1,1);
\draw[curve style] (3,1) -- (3,2);
\draw[wallcurve style] (0,1) -- (4,1);
\end{scope}
\node at (4.5,1) {$=$};
\begin{scope}[xshift=5cm]
\fill [box style] (0,0) rectangle (4,2);
\draw[wall style] (0,1) -- (4,1);
\draw[curve style] (3,0) -- (3,1);
\draw[curve style] (1,1) -- (1,2);
\draw[wallcurve style] (0,1) -- (4,1);
\end{scope}
\end{tikzpicture}
\caption{$\W_{n_\ell,\ell}\W_{-n_\ell,\ell}=\W_{-n_\ell,\ell}\W_{n_\ell,\ell}$.}
\label{fig:comproof7}
\end{figure}

When $v$ is a multiple of $w_\ell$, $\W_{v,\ell}$ is straightforward to calculate, because all holomorphic curves involved are $k$-fold covers of the fibers of $Y\times\mathbb CP^1$. For $v=kw_\ell$ with $k$ positive, we get
\[W_v(\alpha^{+}_{k,\ell})=\ip {h_{-kw_\ell}}{\alpha^+_{k,\ell}} =k \ ,\]
and as $\W_{kw_\ell,\ell}=-W_{kw_\ell}$ (see \cref{wdef}) we get that $\W_{kw_\ell,\ell}\alpha^+_{k,\ell}=-k$.  More generally, for $k>0$, 
\begin{equation} \label{W+}
\W_{kw_\ell,\ell}=\an_{\alpha^+_{k,\ell}}:\mathcal H^+_\ell \longrightarrow \mathcal H^+_\ell \ .
\end{equation}
Similarly, as $\ip {h_{kw_{\ell_1}}}{\alpha^+_{k,\ell_2}} =k$ we get that $\widetilde W_{-kw_{\ell_1}} \alpha^{+\bf p}_{\ell_1}=\alpha^-_{k,\ell_2}\alpha^{-\bf p}_{\ell_2}$, so 
\begin{equation} \label{W-}
\W_{-kw_\ell,\ell}(\beta)=\alpha^+_{k,\ell}\beta \ .
\end{equation}

So our commutation relation holds in this case. 
\[
\W_{kw_\ell,\ell}\W_{k'w_\ell,\ell}-\W_{k'w_\ell,\ell}\W_{kw_\ell,\ell} = -k \delta_{k+k'}
\]
We now have that our commutation relation holds whenever $v+w\neq 0$, and whenever $v$ is proportional to $w_\ell$. It is straightforward to verify that our commutation relation satisfies the Jacobi identity, so it follows that all our operators obey this commutation relation. In particular, when $v+w= 0$, and neither is proportional to $w_\ell$, write $\W_{w,\ell}$ in terms of the commutator of $\W_{w_\ell,\ell}$ and $\W_{w-w_\ell,\ell}$, then apply the Jacobi identity to verify our commutation relation for $\W_{v,\ell}$ and $\W_{w,\ell}$.
\end{proof}

\Cref{W+,W-} give that $\W_{kw_\ell,\ell}$ is adjoint to $\W_{-kw_\ell,\ell}$, using the integration pairing. Equivalently, $\W_{kw_\ell,\ell}$ is adjoint to $-\W_{-kw_\ell,\ell}$ when using our positive definite inner product from \cref{ppdef}. 
\[
\pp{\W_{kw_\ell,\ell}\alpha}{\beta}=\pp{\alpha}{-\W_{-kw_\ell,\ell}\beta}
\]

More generally, using the positive definite inner product,  the adjoint of $\W_{v,\ell}$ is $-\W_{-v,\ell}$ in the sense that
\begin{equation} \label{Wadjoint}
\pp{\alpha}{\W_{v,\ell}\beta}=\pp{-\W_{-v,\ell} \alpha}{ \beta } \ .
\end{equation}

Conceptually, this is because the integration-pairing adjoint of $\widetilde W_v:\mathcal H_{\ell_1}^+\longrightarrow\mathcal H_{\ell_2}^-$ can be constructed in the same way as $\widetilde W_v$, but reversing the roles of the legs $\ell_1$ and $\ell_2$, which are also reversed by the symmetry sending $v$ to $-v$.
\[
\begin{split} \ip{\alpha_{\ell_2}^{+\bf q}}{\widetilde W_{v}\alpha_{\ell_1}^{+\bf p}}&=\ip{\alpha_{\ell_2}^{+\bf q} \alpha_{\ell_1}^{+\bf p} h_{-v} }{\exp \eta}
\\ &=\ip{\alpha_{\ell_1}^{+\bf q} \alpha_{\ell_2}^{+\bf p} h_{v} }{\exp \eta}
\\ &= \ip{\alpha_{\ell_2}^{+\bf p}}{\widetilde W_{-v}\alpha_{\ell_1}^{+\bf q}}
\end{split}
\]
Identifying $\mathcal H_{\ell_2}^\mp$ with $\mathcal H_{\ell_1}^\pm$, to define $W_v$ from $\widetilde W_v$, we then get
\[\pp{\alpha_{\ell_1}^{+\bf q}}{W_v\alpha_{\ell_1}^{+\bf p}}=\ip{\alpha_{\ell_1}^{-\bf q}}{W_v\alpha_{\ell_1}^{+\bf p}}=\ip{\alpha_{\ell_2}^{+\bf q}}{\widetilde W_{v}\alpha_{\ell_1}^{+\bf p}}= \ip{\alpha_{\ell_2}^{+\bf p}}{\widetilde W_{-v}\alpha_{\ell_1}^{+\bf q}}=\pp{W_{-v}\alpha_{\ell_1}^{+\bf q}}{\alpha_{\ell_1}^{+\bf p}}\ .\]
When $\W_{v,\ell_1}=\pm W_{v}$, we have $\W_{-v,\ell_1}=\mp W_{-v}$, so \cref{Wadjoint} holds and the adjoint of $\W_{v,\ell}$ using the positive definite pairing is $-\W_{-v,\ell}$. Note that it is not true in general that $\W_{-v,\ell}$ is the integration-pairing adjoint of $\W_{v,\ell}$, even though this holds in the special case that $v$ is a multiple of $w_\ell$.

There are also the following two vanishing results for $\W_{v,\ell}$ that follow from topological consideration of the possible contact data for holomorphic curves in $Y\times\mathbb{CP}^1$.
\begin{align} 
\W_{v,\ell}\alpha_\ell^{+\mathbf p} &= 0, && \text{ if }\sum_{k=1}^\infty kp_k< {v\wedge n_\ell} \ , \label{Wvanishing} \\
\pp{ \alpha_\ell^{+\mathbf q}}{\W_{v,\ell}\alpha_\ell^{+\mathbf p}} &= 0, && \text{ unless }\sum_{k=1}^\infty kp_k= {v\wedge n_\ell}+\sum_{k=1}^\infty kq_k \ . \label{Wgrading}
\end{align}

So if we assign $\mathcal H^+_\ell$ the $\mathbb Z$-grading in which the degree of $\alpha_\ell^{+\mathbf p}$ is $\sum_k k p_k$, the operator $\W_{v,\ell}$ has degree $-v\wedge n_\ell$. Note that this implies that the propagation operator $\Prop_{x,\ell}$ and $\W_{v,\ell}$ almost commute in the following sense.
\begin{equation} \label{PWc}
\Prop_{x,\ell} \circ \W_{v,\ell}=t^{-(v\wedge n_\ell)x}\W_{v,\ell} \circ \Prop_{x,\ell}
\end{equation}

The degree $0$ operators $\W_{kn_\ell,\ell}$ act as scalars on the degree zero subspace of $\mathcal H^+_\ell$. We compute these weights in the lemmas below.

\begin{lemma} \label{W01}
\[
\W_{n_\ell,\ell}(1)=\frac 1{[1]_q}
\]
\end{lemma}

\begin{proof}
Take $Y$ to be the toric blowup of $(\mathbb{CP}^1)^2$ at $(0,\infty)$ and $(\infty,0)$, so we have coordinates $z_1,z_2$ on $Y$ and $z_3$ on $\mathbb{CP}^1$. Consider the family of embedded holomorphic spheres defined by setting $z_2$ and $z_3$ to be constant and nonzero. The evaluation space for the corresponding moduli space of curves is $\ex X_{v_f}\times \ex X_{(1,0,0)}$. Let $\beta\in\mathcal H$ represent the Poincar\'{e} dual to $\ex X_{v_f}\times \{p\}$ where $p$ is a chosen point in $\ex X_{(1,0,0)}$. Note that $\ex X_{(1,0,0)}$ is the explosion of a toric manifold with toric coordinates $(z_2,z_3)$, and the evaluation map from the interior of our moduli space simply reads off these coordinates. Choosing $p$ so that $(z_2(p),z_3(p))\in (\mathbb C^*)^2$, we get
\[
\ip{ \eta}{\beta} = t^x
\]
where $x$ is the symplectic area of these spheres. The calculation of this Gromov--Witten invariant is analogous when we choose $p\in\ex X_{(1,0,0)}$ in the divisor where $z_2$ is infinite. There, the description of the moduli space is as above, with the coordinates $z_2$ and $z_3$ extended to exploded coordinate functions. The corresponding holomorphic curves have tropical part a straight line in the $v_f$ and $(1,0,0)$ directions, translated in the $(0,1,0)$ direction from the singular line some distance, depending on the image of $p$ in the tropical part of $\ex X_{(1,0,0)}$.

If instead, we choose $p$ in the divisor where $z_2=0$, the calculation becomes interesting. In this case, the tropical part of the relevant curves have a ray in the $(1,0,0)$ direction, translated in the $(0,-1,0)$ direction from the singular line. Over this side of the singular line, the parallel transport of $v_f$ is (0,-1,0) instead of $(-1,0,0)$ so in order for this ray to continue down in the $v_f$ direction and satisfy the balancing condition, this tropical curve must also have an edge in the $(1,-1,0)$ direction, travelling from the singular line to our rays in the $(1,0,0)$ and $v_f$ directions. So this tropical curve has a monovalent vertex on the singular line, joined to a trivalent vertex by an edge in the $(1,-1,0)$ direction, with rays leaving this vertex in the directions $(1,0,0)$ and $v_f$. The only other tropical curves satisfying the balancing condition and our constraints have extra vertices in the interior of the edges of this curve, and such tropical curves never contribute to our Gromov--Witten invariants. See \cref{fig:YxCP1} for a diagram of the tropical curves in the tropical part of $\expl(Y\times \mathbb{CP}^1)$.

\begin{figure}[ht!]
\centering
\begin{tikzpicture}
\begin{scope} [scale = 0.5, yshift = -0.5cm]
\fill[toric style] (0,0) -- (0,1) -- (2,3) -- (4,1) -- (4,0) -- (2,-2) -- cycle;
\draw[yellow, line width = 0.5mm] (0,0) -- (0,1) -- (2,3) -- (4,1) -- (4,0);
\draw[cyan, line width = 0.5mm] (0,0) -- (2,-2) -- (4,0);
\draw[red, line width = 0.5mm] plot [smooth] coordinates{(0.5,-0.5) (1.25,1) (2.5,2.5)};
\draw[yellow, line width = 0.5mm] plot [smooth] coordinates{(0,0.25) (2,0) (4,0.25)};
\end{scope}
\begin{scope} [xshift = 3.5cm, scale = 0.5]
\draw[box style, line width = 1mm] (0,0) -- (6,0);
\draw[curve style] (1.5,-2) -- (1.5,2);
\draw[curve style] (4.5,-2) -- (4.5,2); 
\draw[curve style] (3,0) -- (4.5,0);
\node at (3,0) [circle,fill,inner sep=0.5mm]{};
\end{scope}
\begin{scope} [xshift = 8cm, scale = 0.5]
\fill[box style] (-1,1) -- (1,-1) -- (7,-1) -- (5,1) -- cycle;
\draw[curve style] (2.5,-1.5) -- (2.5,2.5);
\draw[curve style] (3,0) -- (2.5,0.5);
\draw[wall style] (0,0) -- (6,0);
\draw[curve style] (3.5,-2.5) -- (3.5,1.5);
\end{scope}
\end{tikzpicture}
\caption{On the left, a toric picture of $Y$, showing the divisor in yellow and a holomorphic curve in red. On the right, a picture of two tropical curves in the tropical part of $\expl(Y\times \mathbb{CP}^1)$, each the tropical part of a holomorphic curve in the same homology class. In the middle, a side-on view of these two curves, showing the projection to the tropical part of $\expl Y$.}
\label{fig:YxCP1}
\end{figure}
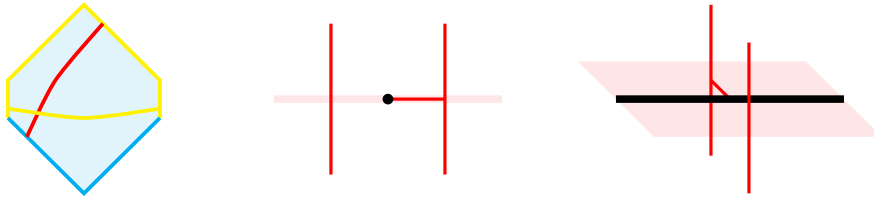

Choose our leg $\ell$ such that $n_\ell = -(1,-1,0)$. The balancing condition for curves along the singular line implies that $\W_{n_\ell,\ell}(1) \in \N$. Moreover, we can identify $W_{n_\ell,\ell}(1)$ as the Gromov--Witten invariant
\[
\W_{n_\ell,\ell}(1) = \ip {\eta}{h_{-n_\ell}} \ ,
\]
because it only counts connected curves --- the virtual moduli space of stable holomorphic curves with zero contact data is empty in this manifold because of symmetry considerations.

The tropical gluing formula of \cref{eq:tgf} and \cref{3dvertex} together give
\[
\ip {\eta}{ \beta } = \W_{n_\ell,\ell}(1)[1]_q \, t^x \ ,
\]
which implies our desired result when compared to $\ip{ \eta}{\beta} = t^x$.
\end{proof}

\cref{W01} together with \cref{Wadjoint}, the commutation relations from \cref{Wcommutation}, and the vanishing relations of \cref{Wgrading}, suffice to determine the operators $\W_{v,\ell}$ completely. In the following lemma, we instead use the tropical gluing formula to calculate the weights $\W_{kn_\ell,\ell}(1)$.

\begin{lemma} \label{W0}
\[
\W_{kn_\ell,\ell}(1)=\frac {(-1)^{k+1}}{[k]_q} \text{ for } k>1
\]
\end{lemma}

\begin{proof}
Use the same notation as in the proof of \cref{W01}. As before, we have
\[
\W_{kn_\ell,\ell}(1) =\ip{ \eta}{ h_{-kn_\ell}} \in\N \text{ for }k>0\ .
\]
Fix a non-negative integer $m$ and consider connected holomorphic curves with contact data $(1,0,0)$ and $v_f -(m-1)n_\ell$. Let $\beta_m$ be the Poincar\'{e} dual to $\ex X_{v_f-(m-1)n_\ell}\times \{p\} \subset \ex X_{v_f-(m-1)n_\ell}\times \ex X_{(1,0,0)}$. Let us compute the Gromov--Witten invariant $\ip \eta {\beta_m}$. The case $\ip\eta{\beta_1} =1$ was computed in the proof of \cref{W01}. For $m \geq 2$, there are no curves with this contact data in the interior of $Y\times \mathbb{CP}^1$, so choosing $p$ such that $(z_2(p),z_3(p))\in(\mathbb C^*)^2$ gives that $\ip\eta{\beta_m }=0$. Similarly, choosing $p$ so that $z_2(p)$ is infinite, there are no tropical curves with the required contact data satisfying the balancing condition, and we come to the same conclusion that these Gromov--Witten invariants vanish.

However, if we choose $p \in \ex X_{(1,0,0)}$ so that $z_2(p)=0$, there are tropical curves satisfying the required conditions. The relevant tropical curves are labelled by partitions of $m$. These tropical curves have a distinguished vertex, from which emanates the rays in direction $(1,0,0)$ and $v_f-(m-1) n_\ell$, and several monovalent vertices on the singular line, each of which is connected to the distinguished vertex by an edge with derivative $-kn_\ell$. Suppose that there are $\mu_k$ such edges with derivative $-kn_\ell$ for each positive integer~$k$. Then the partition of $m$ has $\mu_k$ parts equal to $k$ for each positive integer~$k$. Such tropical curves have $\prod_k\mu_k!$ automorphisms, which must be accounted for when applying the tropical gluing formula of \cref{eq:tgf}. Applying this tropical gluing formula, along with \cite[Theorem~1.1]{3d} and \cite[Equation~(5)]{3d} gives the following formula for our Gromov--Witten invariants. 
\[
\delta_{m,1}t^x=\ip\eta{\beta_m} =\sum_{\sum k\mu_k=m}\prod_k \frac 1{\mu_k!}\lrb{\W_{kn_\ell,\ell}(1)\frac{[k]_q}{k}}^{\mu_k}
\]

These equations determine $\W_{kn_\ell,\ell}$ completely, so to complete the proof, it suffices to substitute $\W_{kn_\ell,\ell}(1)=\frac {(-1)^{k+1}}{[k]_q}$ into the above equation and to check the resulting identity
\[
\sum_{\sum k\mu_k=m} \frac {(-1)^{\sum (k+1)\mu_k}}{\prod \mu_k! \, k^{\mu_k}}=\delta_{m,1} \ .
\]
To see why this holds, note that the sign in the numerator is the parity of a permutation in the symmetric group $S_m$ with $\mu_k$ cycles of length $k$ for each positive integer $k$. The denominator arises in the formula $m!/ (\prod \mu_k! \, k^{\mu_k})$ for the number of such permutations. So the left side of the equation is simply $\frac{1}{m!}$ multiplied by the sum of the signs of the permutations in $S_m$. The above identity then follows from the observation that for $m \geq 2$, the number of even permutations is equal to the number of odd permutations in $S_m$. 
\end{proof}

\section{The quantum torus Lie algebra} \label{sec:quantumtorus}

The commutation relations from \cref{Wcommutation} identify the Lie algebra generated by the operators $\W_{v,\ell}$ as a subalgebra of the sine Lie algebra, also known as a quantum 2-torus Lie algebra. Below, we first discuss the non-quantum case, then discuss its quantum deformation.

Consider the complexification of the Poisson algebra of a real  2-dimensional symplectic torus $\mathbb R^2/(2\pi \mathbb Z)^2$ with symplectic form $\mathrm{d}\theta_1\wedge \mathrm{d}\theta_2$. This algebra has a dense subalgebra generated by the functions $T_v:=e^{i(v_1\theta_1+v_2\theta_2)}$, with the Poisson bracket given by
\[
\{T_v,T_w\}=(v\wedge w) \, T_{v+w} \ ,
\]
where $v\wedge w:=v_1w_2-v_2w_1$ for $v = (v_1, v_2)$ and $w = (w_1, w_2)$. These relations also describe a dense subalgebra of the Poisson algebra of the holomorphic symplectic manifold $(\mathbb C^*)^2$ with holomorphic symplectic form $\frac{\mathrm{d}z_1}{z_1}\wedge \frac{\mathrm{d}z_2}{z_2}$, and corresponding Poisson bivector $z_1\partial_{z_1}\wedge z_2\partial_{z_2}$, where now $T_v$ corresponds to $z_1^{v_1}z_2^{v_2}$.

The above Poisson algebra is related to the Lie algebra of the topological vertex group of Gross, Pandharipande and Siebert~\cite{tropicalvertex}, used to compute genus $0$ Gromov--Witten invariants of 2-dimensional log Calabi--Yau manifolds. It is interesting that a quantisation of this algebra occurs in our study of arbitrary genus Gromov--Witten invariants of three-dimensional log Calabi--Yau manifolds. Indeed, such a quantisation was suggested in the work of Kontsevich and Soibelman~\cite{kontsevicha}, and has already appeared in the work of Bousseau on higher genus Gromov--Witten invariants of 2-dimensional log Calabi--Yau manifolds~\cite{bou,qtropicalvertex}. Bousseau's work can be interpreted as providing a calculation of Gromov--Witten invariants counting curves contained in 2-dimensional Calabi--Yau submanifolds such as our calculation in \cref{W01,W0} of curves in $Y\times \{1\} \subset Y\times\mathbb{CP}^1$. This is different to our current project studying curves contained in $f^{-1}(0)$ and accordingly, Bousseau's quantum torus is a quantisation of a different torus. These two theories may potentially be connected through the web of correspondences introduced by Bousseau, Brini and van Garrel~\cite{bousseau-brini-vangarrel}.

Discarding $T_0$ by restricting to complex-valued functions on the 2-torus whose integral is $0$, we can obtain a quantum 2-torus Lie algebra\footnote{This Lie algebra is also the commutator on the non-commutative torus $C^*$-algebra with generators $U_1$, $U_2$ satisfying the relation $qU_1U_2= U_2 U_1$. Then, setting $T_v=i q^{v_1v_2/2}U_1^{v_1}U_2^{v_2}$, we obtain the above commutation relations for $T_v$. This non-commutative torus algebra is different from the algebra generated by our operators.} as a quantum deformation\footnote{If we followed standard practice for quantisation, we should be taking $-i$ times these generators, resulting in a factor of $i$ in the commutation relation so that $[T_v,T_w]=i\hbar\{T_v,T_w\} + \cdots$.} of this Poisson algebra, with generators $T_v$ for $v\in\mathbb Z^2\setminus \{0\}$ satisfying the commutation relations
\begin{equation} \label{smallqt}
[T_v,T_w]=[v\wedge w]_q \, T_{v+w} \ ,
\end{equation}
where the quantum integer $[n]_q$ is defined as in \cref{eq:qinteger}.

Given a vector $n_\ell \in \mathbb Z^2$, there is a central extension of this Lie algebra with generators $T_v$ for $v\in\mathbb Z^2\setminus \{0\}$ and a central element $C$, satisfying the commutation relations
\[
[T_v,T_w]=[v\wedge w]_q \, T_{v+w}+\delta_{v+w}( n_\ell\wedge v)C \ .
\]
These are the commutation relations defining a sine Lie algebra, also known as the quantum 2-torus Lie algebra. More canonically, there is a central extension of this algebra with generators $T_v$ for $v\in\mathbb Z^2\setminus \{0\}$ and central elements $C_1$ and $C_2$, satisfying the commutation relation
\begin{equation} \label{bigqt}
[T_v,T_w] = [v\wedge w]_q \, T_{v+w}+\delta_{v+w}(v_1C_1+v_2C_2)\ .
\end{equation}

Define the small quantum torus Lie algebra $\Tsmall$ to be the completion of the Lie algebra over $\N$ generated by $T_v$ for $v\in\mathfrak t^2_{\mathbb Z}\setminus \{0\}$, satisfying the commutation relations of \cref{smallqt}. Furthermore, define the big quantum torus Lie algebra $\Tlarge$ as the central extension of $\Tsmall$ by $\mathfrak t^2$, defined by the commutation relations
\[
[T_{v},T_w]=[v\wedge w]_q \, T_{v+w} + \delta_{v+w}v \ .
\]

So on $\mathcal H^+_\ell$, \cref{Wcommutation} gives that the operators $\W_{v,\ell}$ provide a projective representation of $\Tsmall$, and a representation of $\Tlarge$ so that $v\in\mathfrak t^2_{\mathbb Z}\subset \Tlarge$ acts by multiplying by the constant $n_\ell \wedge v$. Using the $\mathbb Z$-grading in which $\W_{v,\ell}$ has degree $n_\ell \wedge v$, this representation is quasi-finite, with basis for the degree $n$ subspace given by 
\begin{equation}\label{Walphap}\alpha_\ell^{+\mathbf p}=\lrb{\prod_k W_{-kw_\ell,\ell}^{ p_k}}(1)\end{equation}  with $\sum_k k p_k=n$; so $\mathbf p$ represents a partition of $n$ in which there are $p_k$ parts equal to $k$. Moreover, the vanishing result of \cref{Wgrading} and the weights from \cref{W0} identify this as a highest weight representation.

These conditions determine the operators $\W_{v,\ell}$. In light of \cref{Walphap}, it suffices to determine the product of such operators acting on $1$. Given $v_1, v_2, \ldots, v_n$, the vanishing result of \cref{Wgrading} implies that $\prod_i \W_{v_i,\ell}(1)$ has degree $-\sum_i v_i\wedge n_\ell$, and therefore 
\begin{equation}\label{Wprod1}
\prod_{i}\W_{v_i,\ell}(1)=\sum_{\mathbf p}c_{\mathbf p}\alpha_\ell^{+\mathbf p} \ ,
\end{equation} 
where the sum is over partitions $\mathbf p$ of $-\sum_i v_i\wedge n_\ell$. The coefficient $c_{\mathbf p}$ is then determined by applying the operator $\frac 1{\abs {G_{\mathbf p}}}\prod_k \W_{kn_\ell,\ell}^{ p_k}$, and we can compute 
\begin{equation}\label{Wprod2}
c_{\mathbf p}=\frac 1{\abs {G_{\mathbf p}}}\prod_i \W_{kn_\ell,\ell}^{ p_k}\prod_i\W_{v_i,\ell}(1) \ ,
\end{equation}
by commuting the negative degree operators to the right using \cref{Wcommutation}, and discarding terms in which negative degree operators act on $1$ until we are left with an expression in degree $0$ operators, which is then determined by \cref{W0}.

\begin{remark}\label{translation}
After tensoring with $\mathbb C$, we can also identify this representation as arising in the work of Okounkov and Pandharipande on the Gromov--Witten/Hurwitz correspondence for curves~\cite{OP1}. In particular, we use a version of the boson-fermion correspondence, and place our bosonic Fock space $\mathcal H^+_\ell\otimes\mathbb C$ within a fermionic Fock space represented by an infinite wedge space. The infinite wedge space is acted on by operators $\mathcal E_{r}(z)$ from~\cite[Section 2.2]{OP1}. Then, for $r$ and $k$ in $\mathbb Z$, the operators 
\[\frac{(-1)^{k}}{i} \mathcal E_r(i\hbar k):=\sum_{j\in\mathbb Z+\frac 12}\frac{(- e^{i\hbar (j-r/2) })^k}{i} E_{j-r,j}+\frac{(-1)^{k+1}}{2\sin k\hbar/2}\delta_{r,0}
\]
obey the same commutation relations as our operators $\W_{rw_\ell+kn_\ell,\ell}$, and define a highest weight representation with the same weights. 

In particular, equation (2.17) of \cite{OP1} gives that
\[
\left[\frac{(-1)^{k_1}}{i} \mathcal E_{r_1}(i\hbar k_1),\frac {(-1)^{k_2}}{i} \mathcal E_{r_2}(i\hbar k_2)\right]=\frac{(-1)^{k_1+k_2}}{i^2}\zeta\lrb{\det \left[ \begin{smallmatrix} r_1 & i\hbar k_1\\ r_2 & i\hbar k_2 \end{smallmatrix}\right]} \mathcal E_{r_1+r_2}(i\hbar(k_1+k_2)) \ ,
\]
where $\zeta(z):=e^{ z/2}-e^{-z/2}$, so $[n]_q=\zeta(i\hbar n)/i$, and we can rewrite the above expression as 
\[
\left[\frac{(-1)^{k_1}}{i} \mathcal E_{r_1}(i\hbar k_1),\frac {(-1)^{k_2}}{i} \mathcal E_{r_2}(i\hbar k_2)\right]=[r_1k_2-r_2k_1]_q\frac{(-1)^{k_1+k_2}}{i}\mathcal E_{r_1+r_2}(i\hbar(k_1+k_2)) \ ,
\]
which agrees with \cref{Wcommutation} when $(r_1,k_1)\neq(-r_2,-k_2)$. In the case that $(r_1,k_1)=(-r_2,-k_2)$, the right side of the above commutation relation is interpreted in \cite{OP1} as 
\[
\left[\frac{(-1)^{k_1}}{i} \mathcal E_{r}(i\hbar k),\frac {(-1)^{-k}}{i} \mathcal E_{-r}(-i\hbar k)\right]=(-1)^{k+k+1}r=-r \ ,
\]
again agreeing with \cref{Wcommutation}. Moreover, $(-1)^k\mathcal E_0(i\hbar k)/i$ acts on the vacuum vector with weight $(-1)^{k+1}/( 2\sin k\hbar/2 )=(-1)^{k+1}/[k]_q$, agreeing with the weights computed in \cref{W0}. 
\end{remark}

\section{Framing change} \label{sec:framing}

Suppose that we choose a different compactification $\bar X$ of $\mathbb C^2\times \mathbb C^*$ so that the two legs $\ell_1$ and $\ell_2$ have different framings, so $w_{\ell_1}\neq -w_{\ell_2}$. Thinking of $\ell_1$ as incoming and $\ell_2$ as outgoing, and using notation from \cref{adef,projdef2}, we get a bounded $\N$-module homomorphism
\begin{align*}
& F_{\bar X}: \mathcal H^+_{\ell_1}\longrightarrow \mathcal H^-_{\ell_2}=\mathcal H^{+}_{\ell_2^-} \\
& F_{\bar X}\alpha=(\an_{\alpha}Z_{\bar X})_{\ell_2^-}^+ \ .
\end{align*}
So $F_{\bar X}\alpha$ is the projection of $\an_{\alpha}Z_{\bar X}$ onto $\mathcal H^-_{\ell_2}=\mathcal H^+_{\ell_2^-}$. This homomorphism depends on the symplectic form chosen on $\bar X$, but is still bounded when we send the symplectic form to $0$, so define $F_{\bar X}$ in this canonical case where the symplectic form is $0$. Note that $F_{\bar X}1=1$, because, after an abstract perturbation, there are no stable curves with contact only on the boundary of $\bar X$ corresponding to $\ell_2$, so the empty curve is the only `curve' contributing to this invariant.

We will think of $F_{\bar X}$ as a framing change, because $n_{\ell_2^-}=-n_{\ell_2}=n_{\ell_1}$, but $w_{\ell_2^-}\neq w_{\ell_1}$. The gluing formula of \cref{glue} applied to the toric degeneration pictured in \cref{f12} implies that $F_{\bar X}$ is an isomorphism, and preserves the positive definite inner product $\pp\cdot\cdot$ from \cref{ppdef}.

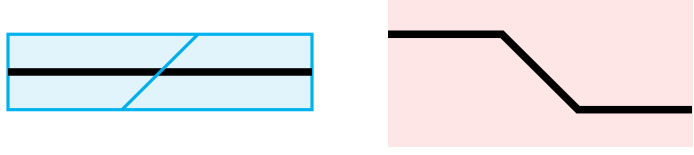
\begin{figure}[ht!]
\centering
\begin{tikzpicture}
\begin{scope}
\filldraw[toric style] (0,0.5) rectangle (4,1.5);
\draw[wall style] (0,1) -- (4,1);
\draw[cyan, very thick] (1.5,0.5) -- (2.5,1.5);
\end{scope}
\begin{scope}[xshift=5cm]
\fill[box style] (0,0) rectangle (4,2);
\draw[wall style] (0,1.5) -- (1.5,1.5) -- (2.5,0.5) -- (4,0.5);
\end{scope}
\end{tikzpicture}
\caption{A toric degeneration of $Y\times\mathbb{CP}^1$ into two pieces, both isomorphic to $\bar X$, leading to the observation that $F_{\bar X}$ is an isomorphism which preserves the inner product.}
\label{f12}
\end{figure}

In fact $F_{\bar X}$ is an isomorphism of representations.

\begin{prop} \label{framing change}
The map $F_{\bar X}$ is an isomorphism of representations. In particular, we have
\[
F_{\bar X}\circ \W_{v,\ell_1}=\W_{v,\ell_2^-}\circ F_{\bar X} \qquad \text{and} \qquad F_{\bar X}(1)=1 \ .
\]
\end{prop}

\begin{proof}
We have already observed that $F_{\bar X}(1)=1$, so it remains to prove that $F_{\bar X}$ intertwines the operators $\W_{v,\ell_1}$ and $\W_{v,\ell_2^-}$.

To prove this, first suppose that 
\[
w_{\ell_1}\wedge v>0 \qquad \text{and} \qquad w_{\ell_2^-}\wedge v>0 \ ,
\]
and compute $Z_{\bar X,h_{-v}}$. We set  the symplectic form to zero so that we can ignore the symplectic energy contributions, which are determined topologically and uninteresting in this case. As in the proof of \cref{zc}, we can choose the constraint corresponding to $h_{-v}$ in two different ways so that different tropical curves contribute to the calculation of $Z_{\bar X,h_{-v}}$. In particular, one constraint leads to the tropical curves on the left in \cref{f13}, and a different constraint leads to the tropical curves on the right.

\begin{figure}[ht!]
\centering
\begin{tikzpicture}
\begin{scope}
\fill [box style] (0,0) rectangle (4,2);
\draw[wall style] (0,1) -- (2,1) -- (3,0);
\draw[curve style] (2,0) -- (2.3333,0.6666);
\draw[wallcurve style] (0,1) -- (2,1) -- (3,0);
\end{scope}
\node at (4.5,1) {$=$};
\begin{scope}[xshift=5cm]
\fill [box style] (0,0) rectangle (4,2);
\draw[wall style] (0,1) -- (2,1) -- (3,0);
\draw[curve style] (0.5,0) -- (1,1);
\draw[wallcurve style] (0,1) -- (2,1) -- (3,0);
\end{scope}
\end{tikzpicture}
\caption{Two ways of calculating $Z_{\bar X,h_{-v}}$ as $F_{\bar X}\circ\W_{v,\ell_1}$ or $\W_{v,\ell_2^-}\circ F_{\bar X}$.}
\label{f13}
\end{figure}
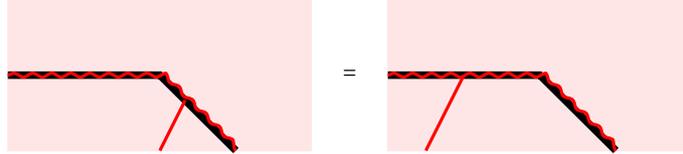

For the constraint on the left, applying the tropical gluing formula of \cref{eq:tgf} with the considerations noted in the proof of \cref{zc} gives that for $\alpha\in\mathcal H^+_{\ell_1}$ and $\beta\in\mathcal H^+_{\ell_2}$, 
\[
\ip{ Z_{\bar X,h_{-v}}}{\alpha\beta}=\ip{ \exp\eta}{h_{-v}\alpha\beta }=\ip{\beta}{F_{\bar X}\circ \W_{v,\ell_1}\alpha} \ ,
\]
whereas using the constraint on the right gives
\[
\ip{ Z_{\bar X,h_{-v}}}{\alpha\beta}=\ip{\beta}{\W_{v,\ell^-_2}\circ F_{\bar X}\alpha} \ .
\]
So we have that $F_{\bar X}\circ \W_{v,\ell_1}=\W_{v,\ell_2^-}\circ F_{\bar X}$ in this case.

Similarly, if 
\[
w_{\ell_1}\wedge v>0 \quad \text{and} \quad w_{\ell_2^-}\wedge v>0 \ ,
\]
we obtain that
\[
\ip{ Z_{\bar X,h_{-v}}}{\alpha\beta}=\ip{\beta}{F_{\bar X}\circ- \W_{v,\ell_1}\alpha}=\ip{\beta}{-\W_{v,\ell^-_2}\circ F_{\bar X}\alpha} \ .
\]
So we also have the required relation in this case.

\begin{figure}[ht!]
\centering
\begin{tikzpicture}
\begin{scope}
\fill [box style] (0,0) rectangle (4,2);
\draw[wall style] (0,1) -- (2,1) -- (3,0);
\draw[curve style] (0,1.5) -- (2.6666,0.3333);
\draw[wallcurve style] (0,1) -- (2,1) -- (3,0);
\end{scope}
\node at (4.5,1) {$=$};
\begin{scope}[xshift=5cm]
\fill [box style] (0,0) rectangle (4,2);
\draw[wall style] (0,1) -- (2,1) -- (3,0);
\draw[curve style] (0,1.5) -- (1,1);
\draw[wallcurve style] (0,1) -- (2,1) -- (3,0);
\end{scope}
\end{tikzpicture}
\caption{In this pictured case, $Z_{\bar X,h_{-v}}=0$ and $F_{\bar X}\circ\W_{v,\ell_1}-\W_{v,\ell_2^-}\circ F_{\bar X}=0$.}
\label{f14}
\end{figure}
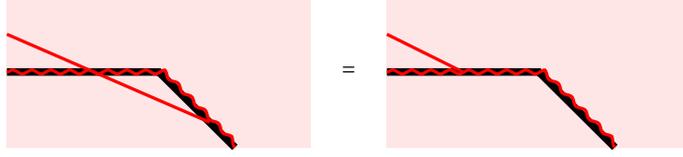

The case that $w_{\ell_i}\wedge v$ have different signs is pictured in \cref{f14}, but we do not need this to complete the proof. We already have the required relation for a set of vectors $v$ generating $\mathbb Z^2$ under addition. As the other operators are obtained as commutators of these operators using \cref{Wcommutation}, it follows that these other operators also satisfy the required relation
\[
F_{\bar X}\circ \W_{v,\ell_1}=\W_{v,\ell_2^-}\circ F_{\bar X} \ . \qedhere
\]
\end{proof}

Note also that $F_{\bar X}$ preserves the $\mathbb Z$-grading, so 
\[
F_{\bar X}\circ \Prop_{x,\ell_1}=\Prop_{x,\ell_2^-}\circ F_{\bar X} \ .
\]

Moreover, \cref{Wadjoint} and the commutation relation implies that $F_{\bar X}$ preserves the inner product. 
\begin{equation}
\pp{\alpha}{\beta}=\pp{F_{\bar X}\alpha}{F_{\bar X}\beta}
\end{equation}
This is because \cref{Walphap} implies that it suffices to check this equation for $\alpha$ and $\beta$ products of the operators $\W_{v,\ell_1}$, then \cref{Wadjoint} allows us to compute this as the vacuum expectation of a product of these operators, which can then be computed using the commutation relations. 

\cref{framing change} completely determines the framing change. Moreover, it motivates us to think of $\mathcal H^+_\ell$ as a highest weight representation of the big quantum torus Lie algebra $\Tlarge$, depending only on $n_\ell$, and not $w_\ell$. From this perspective, the invariant structure of $\mathcal H_\ell^+$ is as a highest weight representation of $\Tlarge$, with its $\N$-module structure, the positive definite inner product $\pp{\cdot}{\cdot}$ and the vacuum vector $1\in\mathcal H_\ell^+$; only the multiplication structure, the integration pairing,   and the preferred basis $\alpha_\ell^{+\mathbf p}$ depend on the framing $w_\ell$.

\section{The topological vertex} \label{sec:tv}

The case when $X=\mathbb C^3$ and $f=z_1z_2z_3$ is of particular interest. In this case, the toric graph has three legs and a single vertex. Choose a compactification $\bar X$ of $\mathbb C^3$ such as the one pictured in \cref{fig:momentpolytope,f15}.

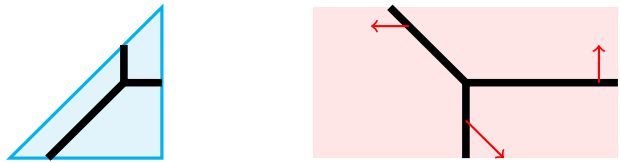
\begin{figure}[ht!]
\centering
\begin{tikzpicture}
\begin{scope}
\filldraw[toric style] (1,0) -- (3,0) -- (3,2) -- cycle;
\draw[wall style] (2.5,1.5) -- (2.5,1) -- (3,1);
\draw[wall style] (2.5,1) -- (1.5,0);
\end{scope}
\begin{scope}[xshift=5cm]
\fill[box style] (0,0) rectangle (4,2);
\draw[wall style] (2,0) -- (2,1) -- (4,1);
\draw[wall style] (1,2) -- (2,1);
\draw[->, red, thick, xshift=0] (2,0.5) -- (2.5,0);
\draw[->, red, thick, xshift=0] (3.75,1) -- (3.75,1.5);
\draw[->, red, thick, xshift=0] (1.25,1.75) -- (0.75,1.75);
\end{scope}
\end{tikzpicture}
\caption{A compactification of $\mathbb C^3$ with framing vectors $(1,0), (-1,1), (0,-1)$, and respective normal vectors $(0,1),(-1,0),(1,-1)$ shown in red.}
\label{f15}
\end{figure}

Regardless of what compactification $\bar X\supset X$ is chosen, the symplectic energy of a holomorphic curve in $\bar X$ is determined by its contact data. For $v\in \mathfrak t^2_{\mathbb Z}$, let $x(v)>0$ be the integral of $\omega$ over a holomorphic sphere in $\bar X$ with contact data $\mathbf p= 1_v$. (Such holomorphic spheres can be found as the closures of orbits of the $\mathbb C^*$-action of weight $v$ on $f^{-1}(0)$.) Then the symplectic energy of a holomorphic curve with contact data $\mathbf p$ supported on $\mathfrak t^2_{\mathbb Z}$ is 
\[
x(\mathbf p):=\sum_v x(v) \, \mathbf p(v)\ .
\]

In this section, we refer to the three legs using the numbers $1$, $2$ and $3$, considered modulo 3 and numbered anticlockwise in the above picture so that $n_\ell\wedge n_{\ell+1}=1$. Consider
\[
Z_{\bar X}\in\mathcal H^-_1\otimes\mathcal H^-_2\otimes\mathcal H^-_3:=\mathcal H^+_{1^-}\otimes\mathcal H^+_{2^-}\otimes\mathcal H^+_{3^-} \ .
\]
Each $\mathcal H^+_{\ell^-}$ has a natural $\mathbb Z$-grading so that the space of elements of degree $n$ has a basis labelled by partitions of $n$. Using the corresponding $\mathbb Z^3$-grading, we can take a limit to complete $\mathcal H^+_{1^-}\otimes\mathcal H^+_{2^-}\otimes\mathcal H^+_{3^-}$ to an $\N$-module 
\[
\widehat{\mathcal H^+_{1^-}\otimes\mathcal H^+_{2^-}\otimes\mathcal H^+_{3^-}}
\]
consisting of infinite sums in $\mathcal H^+_{1^-}\otimes\mathcal H^+_{2^-}\otimes\mathcal H^+_{3^-}$ with a single term in each multi-degree. Then setting $t=1$ in $Z_{\bar X}$ to remove dependence on the symplectic form gives
\[
\T \in \widehat{\mathcal H^+_{1^-}\otimes\mathcal H^+_{2^-}\otimes\mathcal H^+_{3^-}} \ ,
\]
so that
\begin{equation} \label{ZT}
Z_{\bar X}=\Prop_{x(w_1),1^-}\Prop_{x(w_2),2^-}\Prop_{x(w_3),3^-} \T \ .
\end{equation}

We have that each $\mathcal H^+_{\ell^-}$ is a highest weight representation of the big quantum torus Lie algebra $\Tlarge$ using the operators $\W_{v,\ell^-}$, and each of these operators induces a corresponding operator on $\widehat{\mathcal H^+_{1^-}\otimes\mathcal H^+_{2^-}\otimes\mathcal H^+_{3^-}}$. Using the operators $(\W_{v,1^-}+\W_{v,2^-}+\W_{v,3^-})$ gives a $\Tlarge$-representation on this $\N$-module. Moreover, the central elements $v\in\mathfrak t^2\subset \Tlarge$ act on $\mathcal H^+_{i^-}$ by multiplying by the number $v\wedge n_{i^-}$, so they act trivially on $\widehat{\mathcal H^+_{1^-}\otimes\mathcal H^+_{2^-}\otimes\mathcal H^+_{3^-}}$ because $n_{1^-}+n_{2^-}+n_{3^-}=0$. So we actually have a representation of the small quantum torus Lie algebra $\Tsmall$. The following theorem tells us that $\T$ is fixed by the action of the quantum torus.

\begin{thm} \label{tvs}
For any $v\in\mathfrak t^2_{\mathbb Z}$, we have
\[
\left( \W_{v,1^-} + \W_{v,2^-} + \W_{v,3^-} \right) \T = 0 \ .
\]
\end{thm}

\begin{proof}
This follows by calculating $Z_{\bar X,h_{-v}}$, while setting $t=1$ so that we can ignore the symplectic area of curves.

First, suppose that 
\begin{equation} \label{vpos}
w_{1}\wedge v>0, \ w_3\wedge v>0 \text{ and } w_2\wedge v<0 \ .
\end{equation}
As in the proof of \cref{zc}, we can choose the constraint corresponding to $h_{-v}$ so that contributing tropical curves to $Z_{\bar X,h_{-v}}$ are either as depicted in the two pictures on the left in \cref{f16}, or alternatively, as depicted on the right. So for these tropical curves, all edges are either along the singular locus or along a ray in the direction $-v$ starting from the singular locus. Moreover, all vertices are either at the vertex at the centre of the singular locus, at the start of this ray, or are a vertex with valency two that subdivides a ray into two edges, where the ray crosses the singular locus. Other tropical curves do not contribute as they either are not rigid after being constrained, or are eliminated using the tropical gluing formula and the observation that the self intersection of $f^{-1}(0)$ is $0$.

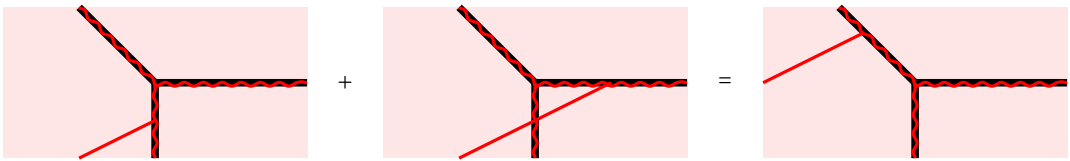
\begin{figure}[ht!]
\centering
\begin{tikzpicture}
\begin{scope}
\fill [box style] (0,0) rectangle (4,2);
\draw[wall style] (1,2) -- (2,1) -- (4,1);
\draw[wall style] (2,0) -- (2,1);
\draw[curve style] (1,0) -- (2,0.5);
\draw[wallcurve style] (1,2) -- (2,1) -- (4,1);
\draw[wallcurve style] (2,0) -- (2,1);
\end{scope}
\node at (4.5,1) {$+$};
\begin{scope}[xshift=5cm]
\fill [box style] (0,0) rectangle (4,2);
\draw[wall style] (1,2) -- (2,1) -- (4,1);
\draw[wall style] (2,0) -- (2,1);
\draw[curve style] (1,0) -- (3,1);
\draw[wallcurve style] (1,2) -- (2,1) -- (4,1);
\draw[wallcurve style] (2,0) -- (2,1);
\end{scope}
\node at (9.5,1) {$=$};
\begin{scope}[xshift=10cm]
\fill [box style] (0,0) rectangle (4,2);
\draw[wall style] (1,2) -- (2,1) -- (4,1);
\draw[wall style] (2,0) -- (2,1);
\draw[curve style] (0,1) -- (1.3333,1.6666);
\draw[wallcurve style] (1,2) -- (2,1) -- (4,1);
\draw[wallcurve style] (2,0) -- (2,1);
\end{scope}
\end{tikzpicture}
\caption{Two ways of calculating $Z_{\bar X,h_{-v}}$ as $-\W_{v,3^-} \T - \W_{v,1^-} \T$ or $\W_{v,2^-} \T$}
\label{f16}
\end{figure}

Applying the tropical gluing formula of \cref{eq:tgf}, with the considerations discussed in the proof of \cref{zc}, gives the equation
\[
Z_{\bar X,h_{-v}} = -\W_{v,3^-} \T - \W_{v,1^-} \T
\]
with one constraint, and 
\[
Z_{\bar X,h_{-v}} = \W_{v,2^-} \T
\]
with a different constraint. To see that the contribution of tropical curves depicted in the middle picture gives the term $-\W_{v,1^-} \T$, we must understand why we can ignore the extra bivalent vertex where our tropical curve crosses the singular locus, as we did when interpreting \cref{f14}. Call this bivalent vertex $q$. The contribution of this vertex $q$ in our tropical gluing formula is a Gromov--Witten invariant $\eta^{[\gamma_q]}$ of the exploded manifold $\expl(\bar X)\tc{q}$, which is isomorphic to the exploded manifold $\expl(Y\times \mathbb{CP}^1)$ considered in \cref{sec:product}. The key calculation is the step gluing the contribution of this vertex along the edge $e$ joining it to the singular locus. Calculate by first gluing along all internal edges apart from $e$, leaving one final integral where $\eta^{[\gamma_q]}$ is paired with $h_{-v}$ at one contact point, and a multiple $m$ of the Poincar\'{e} dual of $f^{-1}(0)$ at the other contact point, coming from the glued-together contribution of the rest of the curve. This Gromov--Witten invariant is then easily calculated to be $m$ by replacing the Poincar\'{e} dual of $f^{-1}(0)$ with the Poincar\'{e} dual of $f^{-1}(2)$. The upshot is that $\an_{h_{-v}}\eta^{[\gamma_q]}$ satisfies \cref{hvdef}, so satisfies the conditions required of $h_{-v}$. So this vertex has no effect on our calculations.

So for $v$ satisfying \cref{vpos}, we have the required equation
\[
\left( \W_{v,1^-} + \W_{v,2^-} + \W_{v,3^-} \right) \T = 0 \ .
\]
Similarly, we have the required equation for all $v$ satisfying \cref{vpos} with the legs $(1,2,3)$ cyclically permuted. Such vectors $v$ generate $\mathfrak t^2_{\mathbb Z}$ under addition, so the other operators $\W_{w,\ell^-}$ are obtained using \cref{Wcommutation} as commutators of operators satisfying the above equation. Therefore the required equation holds for all $v\in\mathfrak t^2_{\mathbb Z}$.
\end{proof}

\cref{tvs}, and \cref{PWc,ZT} imply the following modification when we include the symplectic energy of curves.

\begin{cor} \label{Zsym}
For $X=\mathbb C^3$, we have
\[
\left( t^{x(w_1)(v\wedge n_1) }\W_{v,1^-}+t^{x(w_2)(v\wedge n_2) }\W_{v,2^-}+t^{x(w_3)(v\wedge n_3) }\W_{v,3^-} \right) Z_{\bar X} = 0 \ .
\]
\end{cor}
So $Z_{\bar X}$ is annihilated by a different action of the small quantum torus Lie algebra $\Tsmall$ using the operators $( t^{x(w_1)(v\wedge n_1) }\W_{v,1^-}+t^{x(w_2)(v\wedge n_2) }\W_{v,2^-}+t^{x(w_3)(v\wedge n_3) }\W_{v,3^-})$. Note, however that these operators are only defined on a subset of $\mathcal H^+_{1^-}\otimes\mathcal H^+_{2^-}\otimes\mathcal H^+_{3^-}$ because of the presence of negative powers of $t$.

Let $\T_{\emptyset}$ indicate the degree $(0,0,0)$ part of $\T$. This consists of the contribution of curves with all components constant. This is readily computable using \cite{FP}.  The contribution of all other curves are determined by the commutation relations from \cref{Wcommutation}, and the vanishing relations of \cref{Wgrading} for $\W_{v,\ell}$, \cref{tvs}.

Use $\T_\ell$ for the projection of $\T$ to $\widehat{\mathcal H^+_{\ell^-}}$, the completion of $\mathcal H^+_{\ell^-}\subset \widehat{\mathcal H^+_{1^-}\otimes\mathcal H^+_{2^-}\otimes\mathcal H^+_{3^-}}$ allowing infinite sums with one term in each degree.

\begin{lemma} \label{1leg}
For $\ell = 1, 2, 3$, we have the following equations.
\begin{align*}
\T_\ell &= \exp\lrb{\sum_k\frac{(-1)^{k}}{k[k]_q} \W_{-kn_{\ell+1},\ell^-}} \T_{\emptyset} \\
\T_\ell &= \exp\lrb{\sum_k\frac{(-1)^{k+1}}{k[k]_q} \W_{kn_{\ell-1},\ell^-}} \T_{\emptyset}
\end{align*}
\end{lemma}

\begin{proof}
For notational ease, we assume without loss of generality that $\ell = 3$, so we aim to show that 
\[
\T_3=\exp\lrb{\sum_k\frac{(-1)^{k}}{k[k]_q} \W_{-kn_1,3^-}} \T_{\emptyset} \ .
\]

\Cref{Wgrading} tells us that the operator $\W_{kn_1,\ell^-}$ has degree $kn_1 \wedge n_\ell$. In particular, $\W_{kn_1,2^-}$ has degree $k$, $\W_{kn_1,3^-}$ has degree $-k$ and $\W_{kn_1,1^-}$ has degree $0$. So we can obtain an element with degree $(0,0,n)$ by applying $\W_{kn_1,1^-}$ to an element of degree $(0,0,n)$, or by applying $\W_{kn_1,3^-}$ to an element of degree $(0,0,n+k)$. Therefore, \cref{tvs} implies that 
\begin{equation} \label{t3e}
\left(\left( \W_{kn_1,1^-}+\W_{kn_1,2^-}+\W_{kn_1,3^-} \right) \T_3\right)_3=\left(\left( \W_{kn_1,1^-}+\W_{kn_1,2^-}+\W_{kn_1,3^-} \right) \T\right)_3=0 \ .
\end{equation}
where the subscript indicates that we project to the part with degree $(0,0,*)$.

Using that $n_1=w_{3^-}$, \cref{W0} and \cref{Wadjoint} allow us to simplify this equation to 
\[
\W_{kn_1,3^-} \T_3= \frac {(-1)^{k+1}}{[k]_q} \T_3 \ .
\]
Using the fact that $\T_3$ is a sum of powers of $\W_{-kn_1,3^-}$ applied to $1\otimes 1\otimes 1$ or $\T_{\emptyset}$, and the commutation relation
\[
[\W_{kn_1,3^-},\W_{-k'n_1,3^-}]=k( n_{3^-}\wedge n_1) \delta_{k,k'}=- k\delta_{k,k'}
\]
then gives that
\[
\T_3= \exp\lrb{\sum_k\frac{(-1)^{k}}{k[k]_q} \W_{-kn_1,3^-}} \T_{\emptyset} \ .
\]

The statement using $\W_{kn_2,3^-}$ is analogous, although when simplifying this operator, we get a different sign because $\W_{kn_1,1^-}(1)=-\W_{-kn_2,2^-}$, obtaining 
\[
\T_3=\exp\lrb{\sum_k\frac{(-1)^{k+1}}{k[k]_q} \W_{kn_2,3^-}} \T_{\emptyset} \ . \qedhere
\] 
\end{proof}

We will now describe how to obtain the $2$-legged and $3$-legged topological vertex by applying operators to the $1$-legged topological vertex. Define the operators $\E_{\ell+1,\ell}:\widehat{\mathcal H^+_{1^-}\otimes\mathcal H^+_{2^-}\otimes\mathcal H^+_{3^-}}\longrightarrow \widehat{\mathcal H^+_{1^-}\otimes\mathcal H^+_{2^-}\otimes\mathcal H^+_{3^-}}$ by 
\[
\E_{\ell+1,\ell}:=\exp\lrb{\sum_k \frac {1}{k} \left( \W_{kn_{\ell+1},(\ell+1)^-}+\W_{kn_{\ell+1},(\ell+2)^-} \right) \W_{-kn_{\ell+1},\ell^-} } \ .
\]
The commutation relation $[\W_{sn_{\ell+1},\ell^-},\W_{-kn_{\ell+1},\ell^-}]=-s\delta_{s,k}$ implies that, for $s>0$, 
\[
\left[\W_{sn_{\ell+1},\ell^-}, \lrb{\sum_{k=1}^{\infty} \frac {1}{k} \left( \W_{kn_{\ell+1},(\ell+1)^-}+\W_{kn_{\ell+1},(\ell+2)^-} \right) \W_{-kn_{\ell+1},\ell^-} }\right]= -\left( \W_{sn_{\ell+1},(\ell+1)^-}+\W_{sn_{\ell+1},(\ell+2)^-} \right) \ ,
\]
which commutes with each operator in the sum defining $\E_{\ell+1,\ell}$. Using this to compute $[\W_{sn_{\ell+1},\ell^-},\E_{\ell+1,\ell}]$, we arrive at the key property
\begin{equation} \label{WE}
[\W_{kn_{\ell+1},\ell^-}, \E_{\ell+1,\ell}] =- \left( \W_{kn_{\ell+1},(\ell+1)^-}+\W_{kn_{\ell+1},(\ell+2)^-} \right) \E_{\ell+1,\ell} \ , \quad \text{ for }k>0 .
\end{equation}
So, if $\W_{kn_{\ell+1},\ell^-}\, \alpha=0$, then 
\[
\left( \W_{kn_{\ell+1},\ell^-}+\W_{kn_{\ell+1},(\ell+1)^-}+\W_{kn_{\ell+1},(\ell+2)^-} \right) \E_{\ell+1,\ell} \, \alpha=0\ .
\]
Similarly, define
\[
\E_{\ell-1,\ell}:=\exp\lrb{\sum_k \frac {1}{k} \left( \W_{-kn_{\ell-1},(\ell+1)^-}+\W_{-kn_{\ell-1},(\ell-1)^-} \right) \W_{kn_{\ell-1},\ell^-} } \ ,
\]
which satisfies the key property
\begin{equation} \label{WE2}
[\W_{-kn_{\ell-1},\ell^-},\E_{\ell-1,\ell}]= - \left( \W_{-kn_{\ell-1},(\ell+1)^-}+\W_{-kn_{\ell-1},(\ell-1)^-} \right) \E_{\ell-1,\ell} \ , \quad \text{ for }k>0 .
\end{equation}

In the following lemmas, we prove that 
\[
\T = \E_{1,3} \, \E_{3,2} \T_1 = \E_{1,2} \, \E_{2,3} \T_1 \ .
\]

Use the notation $\T_{\ell,\ell+1}$ for the projection of $\T$ to $\widehat{\mathcal H^+_{\ell^-}\otimes\mathcal H^+_{(\ell+1)^-}}$.

\begin{lemma} \label{2leg}
\[
\T_{\ell,\ell+1} = \E_{\ell+2,\ell+1} \T_\ell = \E_{\ell-1,\ell} \T_{\ell+1}
\]
\end{lemma}

\begin{proof}
For notational ease, we assume without loss of generality that $\ell = 3$. So we must show that
\[
\T_{3,1}=\exp\lrb{\sum_k \frac {1}{k}(\W_{kn_{2},3^-}+\W_{kn_{2},2^-})\W_{-kn_{2},1^-} }\T_3\ .
\]

Note that $\W_{kn_2,1^-}$ has degree $-k$, and therefore annihilates $\T_3$. We also have that $\W_{kn_2,2^-}$ has degree $0$ and $\W_{kn_2,3^-}$ has degree $k$. All these operators commute with projection to degree $(*,0,*)$, so \cref{tvs} implies the equation
\[
(\W_{kn_{2},1^-}+\W_{kn_{2},2^-}+ \W_{kn_{2},3^-} ) \T_{3,1}=0 \ .
\]
As $\W_{kn_2,1^-}$ annihilates $\T_3$, \cref{WE} implies that $\T_{3,1}=\E_{2,1} \T_3$ solves this equation with the initial condition that its projection to $\widehat{\mathcal H^+_{3^-}}$ is $\T_3$.

Similarly, $\W_{-kn_2,3^-}$ annihilates $\T_1$, and degree considerations and \cref{tvs} imply the equation
\[
\left( \W_{-kn_{2},1^-}+\W_{-kn_{2},2^-}+ \W_{-kn_{2},3^-} \right) \T_{3,1}=0 \ .
\]
So $\T_{3,1}=\E_{2,3} \T_1$ solves this equation, with the initial condition that its projection to $\widehat{\mathcal H^+_{1^-}}$ is $\T_1$.
\end{proof}

Using $\T_{2,3}=\E_{1,2}\exp\lrb{\sum_k \frac {(-1)^k}{k[k]_q} \W_{-kn_{1},3^-} }\T_{\emptyset}$, we get an explicit expression for the two-legged topological vertex as a single exponential, because $\E_{1,2}$ and $\W_{-kn_1,3^-}$ commute.

\begin{equation}
\begin{split}\T_{2,3} &=\exp\lrb{\sum_k \frac {1}{k} \left( \W_{-kn_{1},3^-}+\frac {(-1)^{k+1}}{[k]_q} \right) \W_{kn_{1},2^-} }\exp\lrb{\sum_k \frac {(-1)^k}{k[k]_q} \W_{-kn_{1},3^-} }\T_{\emptyset} 
\\&= \exp \lrb{\sum_k \frac{(-1)^{k}}{k[k]_q} \left( \W_{-kn_1,3^-} - \W_{kn_1,2^-} \right) + \frac 1k \W_{-kn_1,3^-} \W_{kn_1,2^-}} \T_{\emptyset} \end{split}
\end{equation}

\begin{lemma} \label{3leg}
\[
\T = \E_{\ell-1,\ell} \T_{\ell+1,\ell+2}=\E_{\ell+1,\ell} \T_{\ell+1,\ell+2}
\]
\end{lemma}

\begin{proof}
For notational ease, we assume without loss of generality that $\ell = 1$. As $\W_{-kn_3,1^-}$ has degree $-k$, it annihilates $\T_{2,3}$. Therefore, \cref{WE2} implies that $\T = \E_{3,1} \T_{2,3}$ satisfies the required symmetry equation
\[
(\W_{-kn_3,1^-} + \W_{-kn_3,2^-} + \W_{-kn_3,3^-}) \T = 0 \ ,
\]
and satisfies the initial condition that its projection to $\widehat{\mathcal H^+_{2^-} \otimes \mathcal H^+_{3^-}}$ is $\T_{2,3}$.

Similarly, $\W_{kn_2,1^-}$ has degree $-k$, so annihilates $\T_{2,3}$, and it follows from \cref{WE} that
\[
\T = \E_{2,1} \T_{2,3} \ . \qedhere
\]
\end{proof}

Using \cref{PWc,ZT}, we can include the symplectic area of curves in our calculations and determine $Z_{\bar X}$. Recall that the symplectic energy of a curve with contact data $v$ is given by $x(v)$ for $x$ some positive function on $\mathfrak t^2_{\mathbb Z}\setminus \{0\}$.

Define
\[
\E_{\ell+1,\ell}(x) := \exp \left( \sum_{k=1}^\infty \frac {1}{k} \left( \W_{kn_{\ell+1},(\ell+1)^-}+t^{kx(w_{\ell+2})} \W_{kn_{\ell+1},(\ell+2)^-} \right) t^{kx(w_\ell)} \W_{-kn_{\ell+1},\ell^-} \right) \ ,
\]
so \cref{PWc} implies that
\[
\Prop_{x(w_1),1^-} \Prop_{x(w_2),2^-} \Prop_{x(w_3),3^-} \E_{\ell+1,\ell} = \E_{\ell+1,\ell}(x)\Prop_{x(w_1),1^-} \Prop_{x(w_2),2^-} \Prop_{x(w_3),3^-} \ .
\]
Similarly, define
\[
\E_{\ell-1,\ell}(x) := \exp \lrb{\sum_{k=1}^\infty \frac {1}{k} \left( t^{kx(w_{\ell+1})}\W_{-kn_{\ell-1},(\ell+1)^-}+\W_{-kn_{\ell-1},(\ell-1)^-} \right) t^{kx(w_\ell)}\W_{kn_{\ell-1},\ell^-} } \ ,
\]
which satisfies the analogous equation
\[
\Prop_{x(w_1),1^-} \Prop_{x(w_2),2^-} \Prop_{x(w_3),3^-}\E_{l-1,\ell} = \E_{l-1,\ell}(x)\Prop_{x(w_1),1^-} \Prop_{x(w_2),2^-} \Prop_{x(w_3),3^-} \ .
\]
Note that because of the positive powers of $t$, the above operators $\E_{\ell \pm 1, \ell}(x)$ define bounded operators on $\mathcal H^+_{1^-}\otimes\mathcal H^+_{2^-}\otimes\mathcal H^+_{3-}$.

Moreover, \cref{1leg,2leg,3leg} together with \cref{ZT} have the following immediate consequence.

\begin{cor} \label{tvcalc}

\[
Z_{\bar X}=\E_{1,2}(x) \, \E_{2,3}(x) \Prop_{x(w_1),1^{-}}\T_1
 = \E_{1,2}(x) \, \E_{2,3}(x)\exp\lrb{\sum_k\frac{(-t^{x(w_1)})^{k}}{k[k]_q} \W_{-kn_2,1^-}} \T_{\emptyset}
\]
\end{cor}

\bibliographystyle{plain}
\bibliography{tropological-vertex}

\end{document}